\documentclass[10pt,a4paper]{amsart}
\title{Orthogonally spherical objects and spherical fibrations}
\author{Rina Anno}
\email{anno@pitt.edu}
\address{Department of Mathematics\\
School of Arts and Sciences\\
University of Pittsburgh\\
301 Thackeray Hall\\
Pittsburgh, PA 15260\\
USA}
\author{Timothy Logvinenko} 
\email{LogvinenkoT@cardiff.ac.uk} 
\address{
Cardiff School of Mathematics\\ 
Cardiff University\\
Senghennydd Road\\
Cardiff, CF24 4AG\\
UK}
\usepackage{amsmath,amsfonts,amssymb,amsthm,epsfig,amscd,psfrag,latexsym,
comment}
\usepackage{caption}
\usepackage{graphicx}
\usepackage{array}
\usepackage{subfigure}
\usepackage[all]{xy}
\input xy
\xyoption{all}

\usepackage[colorlinks=true, pdfpagemode=none, pdfmenubar=false, pdfstartview=FitH, linkcolor=blue, citecolor=blue, urlcolor=blue, pdffitwindow=false]{hyperref}

\DeclareMathOperator{\codim}{codim}

\DeclareMathOperator{\homm}{Hom}
\DeclareMathOperator{\shhomm}{{\it\mathcal{H}om\rm}}

\DeclareMathOperator{\autm}{Aut}

\DeclareMathOperator{\gsl}{SL}

\DeclareMathOperator{\picr}{Pic}

\DeclareMathOperator{\cl}{Cl}

\DeclareMathOperator{\spec}{Spec\;}

\DeclareMathOperator{\ext}{Ext}

\DeclareMathOperator{\supp}{Supp}

\DeclareMathOperator{\cohcat}{Coh}
\DeclareMathOperator{\modd}{\bf Mod}
\DeclareMathOperator{\lder}{\bf L}
\DeclareMathOperator{\rder}{\bf R}

\DeclareMathOperator{\rank}{rk}

\DeclareMathOperator{\id}{Id}

\DeclareMathOperator{\perf}{perf}
\DeclareMathOperator{\fintype}{\mathcal{F}\mathcal{T}}
\DeclareMathOperator{\dperf}{DP}
\DeclareMathOperator{\vectspaces}{\bf Vect}
\DeclareMathOperator{\cone}{Cone}
\begin{document}

\def\bv{\mathbf{v}}
\def\kgc_{K^*_G(\mathbb{C}^n)}
\def\kgchi_{K^*_\chi(\mathbb{C}^n)}
\def\kgcf_{K_G(\mathbb{C}^n)}
\def\kgchif_{K_\chi(\mathbb{C}^n)}
\def\gpic_{G\text{-}\picr}
\def\gcl_{G\text{-}\cl}
\def\trch_{{\chi_{0}}}
\def\regring{{R}}
\def\regrep{{V_{\text{reg}}}}
\def\givrep{{V_{\text{giv}}}}
\def\lbar{{(\mathbb{Z}^n)^\vee}}
\def\genpx_{{p_X}}
\def\genpy_{{p_Y}}
\def\genpcn_{p_{\mathbb{C}^n}}
\def\gnat{gnat}
\def\twalg{{\regring \rtimes G}}
\def\L{{\mathcal{L}}}
\def\O{{\mathcal{O}}}
\def\gcd{\mbox{gcd}}
\def\lcm{\mbox{lcm}}
\def\tf{{\tilde{f}}}
\def\tD{{\tilde{D}}}

\def\mckquiv{\mbox{Q}(G)}
\def\C{{\mathbb{C}}}
\def\sF{{\mathcal{F}}}
\def\sW{{\mathcal{W}}}
\def\sL{{\mathcal{L}}}
\def\O{{\mathcal{O}}}
\def\Z{{\mathbb{Z}}}
\def\hmone{{\mathcal{W}}}

\theoremstyle{definition}
\newtheorem{defn}{Definition}[section]
\newtheorem*{defn*}{Definition}
\newtheorem{exmpl}[defn]{Example}
\newtheorem*{exmpl*}{Example}
\newtheorem{exrc}[defn]{Exercise}
\newtheorem*{exrc*}{Exercise}
\newtheorem*{chk*}{Check}
\newtheorem*{remarks*}{Remarks}
\theoremstyle{plain}
\newtheorem{theorem}{Theorem}[section]
\newtheorem*{theorem*}{Theorem}
\newtheorem{conj}[defn]{Conjecture}
\newtheorem*{conj*}{Conjecture}
\newtheorem{prps}[defn]{Proposition}
\newtheorem*{prps*}{Proposition}
\newtheorem{cor}[defn]{Corollary}
\newtheorem*{cor*}{Corollary}
\newtheorem{lemma}[defn]{Lemma}
\newtheorem*{claim*}{Claim}
\newtheorem{Specialthm}{Theorem}
\renewcommand\theSpecialthm{\Alph{Specialthm}}
\numberwithin{equation}{section}
\renewcommand{\textfraction}{0.001}
\renewcommand{\topfraction}{0.999}
\renewcommand{\bottomfraction}{0.999}
\renewcommand{\floatpagefraction}{0.9}
\setlength{\textfloatsep}{5pt}
\setlength{\floatsep}{0pt}
\setlength{\abovecaptionskip}{2pt}
\setlength{\belowcaptionskip}{2pt}
\begin{abstract}
We introduce a relative version of the spherical objects of Seidel 
and Thomas
\cite{SeidelThomas-BraidGroupActionsOnDerivedCategoriesOfCoherentSheaves}.
Define an object $E$ in the derived category 
$D(Z \times X)$ to be spherical over $Z$ if the corresponding 
functor from $D(Z)$ to $D(X)$ gives rise to autoequivalences 
of $D(Z)$ and $D(X)$ in a certain natural way. Most known examples 
come from subschemes of $X$ fibred over $Z$. This categorifies 
to the notion of an object of $D(Z \times X)$ orthogonal 
over $Z$. We prove that such an object is spherical over $Z$ 
if and only if it possesses certain cohomological properties 
similar to those in the original definition of a spherical object. 
We then interpret this geometrically in the case when our 
objects are actual flat fibrations in $X$ over $Z$. 
\end{abstract}

\maketitle

\section{Introduction} \label{section-intro}

Let $X$ be a smooth projective variety over $\mathbb{C}$ and $D(X)$
be the bounded derived category of coherent sheaves on $X$. 
Following certain developments in mirror symmetry Seidel and Thomas 
introduced
in \cite{SeidelThomas-BraidGroupActionsOnDerivedCategoriesOfCoherentSheaves}
the notion of a \em spherical object\rm:  
\begin{defn}[\cite{SeidelThomas-BraidGroupActionsOnDerivedCategoriesOfCoherentSheaves}]
\label{defn-original-spherical-objects-of-Seidel-and-Thomas}
An object $E$ of $D(X)$ is \em spherical \rm if: 
\begin{enumerate}
\item $\homm^i_{D(X)}(E, E) = 
\begin{cases}
\mathbb{C}, \quad \text{ if } i = 0 \text{ or} \dim X,\\
0, \quad \text{ otherwise }
\end{cases}$ \item $ E \simeq E \otimes \omega_X$ where $\omega_X$ is the canonical 
bundle of $X$. 
\end{enumerate}
\end{defn}
The motivating idea came from considering Lagrangian spheres 
on a symplectic manifold. Given such a sphere one can associate 
to it a symplectic automorphism called the Dehn twist. 
Correspondingly:
\begin{theorem*}[\cite{SeidelThomas-BraidGroupActionsOnDerivedCategoriesOfCoherentSheaves}]
Let $E \in D(X)$. The twist functor $T_E$ is a cone 
we can associate to the natural transformation $E \otimes_\mathbb{C}
\rder\homm_X(E, -) \xrightarrow{\mathrm{eval}}  \id_{D(X)}$. If 
$E$ is spherical, then $T_E$ is an autoequivalence of $D(X)$. 
\end{theorem*}

Spherical twists can be used to construct braid group actions 
on $D(X)$, as was indeed the main concern of
\cite{SeidelThomas-BraidGroupActionsOnDerivedCategoriesOfCoherentSheaves}. 
They also deserve to be studied in their own right as some of 
the simplest non-trivial autoequivalences of $D(X)$
which do not come from autoequivalences
of the underlying abelian category $\cohcat(X)$.
In fact, on smooth toric surfaces or on surfaces of general type 
whose canonical model has at worst $A_n$-singularities
the whole of $\autm D(X)$ is generated by spherical twists, 
$\autm \cohcat(X)$ and the shift functor
(\cite{IshiiUehara-AutoequivalencesOfDerivedCategoriesOnTheMinimalResolutionsOfA_nSingularitiesOnSurfaces},
\cite{BroomheadPloog-AutoequivalencesOfToricSurfaces}). 
In more complicated cases spherical twists are still 
an essential tool in studying the autoequivalences of 
$D(X)$ and stability conditions on it 
(\cite{Bridgeland-StabilityConditionsOnK3Surfaces}, 
\cite{Bridgeland-StabilityConditionsAndKleinianSingularities},
\cite{Bridgeland-StabilityConditionsOnANonCompactCalabiYauThreefold}). 

In this paper we study a relative version of the construction above
where instead of a single object we have a family
of objects in $D(X)$ parametrised by a base $Z$. A geometric example
of this is a subvariety $W$ of $X$ flatly fibred over $Z$. It
can be thought of as a family of subschemes of $X$ parametrised by
$Z$. Even when the structure sheaf of $W$ is not itself spherical 
in sense of  
\cite{SeidelThomas-BraidGroupActionsOnDerivedCategoriesOfCoherentSheaves}
one may still produce an autoequivalence of $D(X)$ by 
exploiting the extra fibration structure which $W$ possesses. 
Moreover, one can do this completely abstractly, working 
with families of arbitrary objects of $D(X)$ and 
not just families of subschemes of $X$. We characterize 
those families of objects of $D(X)$ 
for which this is possible in terms 
of applicable cohomological criteria similar to Definition 
\ref{defn-original-spherical-objects-of-Seidel-and-Thomas} above. 
Our study is a self-contained exercise in derived categories 
of coherent sheaves and doesn't involve mirror symmetry. 
One should mention though that the original
examples of these family twists were inspired by Kontsevich's
proposal that the autoequivalences of $D(X)$ should correspond 
to loops in the moduli space of complex structures on its mirror, cf. 
\cite[\S4.1]{Horja-HypergeometricFunctionsAndMirrorSymmetryInToricVarieties}, 
\cite{Horja-DerivedCategoryAutomorphismsFromMirrorSymmetry}, 
\cite{Szendroi-DiffeomorphismsAndFamiliesOfFourierMukaiTransformsInMirrorSymmetry},
\cite{Szendroi-ArtinGroupActionsOnDerivedCategoriesOfThreefolds}. 
Maybe in future our results could be used to construct further, more 
general examples of this correspondence. 

Consider an object $E$ in the derived category $D(Z \times X)$ of
the product of $Z$ and $X$. We can view $E$ as a 
\em family of objects in $D(X)$ parametrised by $Z$ \rm 
by considering the fibres of $E$ over points of $Z$ to be 
the derived pullbacks of $E$ to the corresponding fibres 
of $X \times Z$ over $Z$: 
\begin{align*}
\xymatrix{
X \ar[r]^{\iota_{X_p}} \ar[d] &
Z \times X \ar[d]^{\pi_Z} \\
\bullet \ar[r]_{\iota_p} &
Z
}
\quad \quad \quad \forall\; p \in Z \quad E_p = \iota^*_{X_p} E
\end{align*}
On the other hand, each object $E \in D(Z \times X)$ defines 
naturally a functor $\Phi_E\colon D(Z) \rightarrow D(X)$ 
called the \em Fourier--Mukai transform with kernel $E$ \rm 
which sends point sheaves $\mathcal{O}_p$ on $Z$ to 
the fibres $E_p \in D(X)$ \cite{Huybrechts-FourierMukaiTransformsInAlgebraicGeometry}. The interplay between these two 
points of view, moduli-theoretic and functorial, led
to a string of celebrated results by Mukai, Bondal and Orlov, 
Bridgeland and others. 

When $Z$ is the point scheme $\spec\mathbb{C}$ 
the above formalism tells us to view an object $E \in D(X)$ as 
a functor $\Phi_E = E \otimes_\mathbb{C} (-)$ 
from $D(\vectspaces)$ to $D(X)$. Then the functor
$E \otimes_\mathbb{C} \rder\homm_X(E, -)$ is the composition 
of $\Phi_E$ with its right adjoint $\Phi^{\mathrm radj}_E$ and
the above definition of the twist functor $T_E$ 
amounts to $T_E$ being a cone of the adjunction co-unit
\begin{align}
\label{eqn-intro-right-adjunction-counit}
\Phi_E \Phi^{\mathrm radj}_E \longrightarrow \id_{D(X)}. 
\end{align}
There is a subtlety involved here: taking cones, infamously, 
is not functorial 
in $D(X)$, so the cone of a morphism between two functors is not 
\em a priori \rm well defined. However 
in \cite{AnnoLogvinenko-OnTakingTwistsOfFourierMukaiFunctors} it
is shown that in a very general context we can represent
both functors in \eqref{eqn-intro-right-adjunction-counit} by
Fourier--Mukai kernels and then represent 
the adjunction co-unit \eqref{eqn-intro-right-adjunction-counit}
by a natural morphism $\mu$ between these kernels. 
We can therefore define the twist functor $T_E$ as 
the Fourier--Mukai transform whose kernel is the cone of $\mu$
and consider the following:

\em Problem:\; Describe the objects $E$ in $D(Z \times X)$ for
which the twist $T_E$ is an autoequivalence of $D(X)$. \rm

A partial answer was provided by Horja in 
\cite{Horja-DerivedCategoryAutomorphismsFromMirrorSymmetry} for 
smooth $Z$ and $X$. He considers only those objects $E$
of $D(Z \times X)$ which come from the derived
category of a smooth subscheme of $X$ flatly fibred over $Z$. 
For these he gives a cohomological criterion 
sufficient for the twist $T_E$ to be an autoequivalence of $D(X)$.  
Another special case was treated by Toda in 
\cite{Toda-OnACertainGeneralizationOfSphericalTwists} who studied 
infinitesimal deformations and so assumed $X$ to be a smooth projective 
variety and $Z$ to be the $\spec$of a local artinian $\mathbb{C}$-algebra. 
In \cite{AnnoLogvinenko-SphericalDGFunctors} we took 
different approach and abstracted out the properties of the functors 
$\Phi_E$ defined by spherical objects of  
\cite{SeidelThomas-BraidGroupActionsOnDerivedCategoriesOfCoherentSheaves}
and
\cite{Horja-DerivedCategoryAutomorphismsFromMirrorSymmetry}
which are exploited in proving that the twists $T_E$ are
autoequivalences. In all these cases not only $T_E$ is an autoequivalence, 
but this autoequivalence identifies naturally the left and 
right adjoints of $\Phi_E$. Moreover, in all these cases
the same is true of the \em co-twist \rm $F_E$,
defined as the cone of the adjunction unit 
$\id_{D(Z)} \rightarrow \Phi_E^{\mathrm radj} \Phi_E$ shifted
by $1$ to the right. 
In \cite[Theorem 5.1]{AnnoLogvinenko-SphericalDGFunctors} we prove
a general result which implies that, in fact, for any 
Fourier-Mukai transform $D(Z) \xrightarrow{\Phi_E} D(X)$ 
we have
\begin{align}
\left\{
\begin{matrix}
F_E \text{ is an autoequivalence } \\
\Phi_E^{\mathrm radj} \simeq F_E \Phi_E^{\mathrm ladj}[1]
\end{matrix}
\right\}
\quad \text{ if and only if }\quad
\left\{
\begin{matrix}
T_E \text{ is an autoequivalence } \\
\Phi_E^{\mathrm ladj} T_E[-1]  \simeq \Phi_E^{\mathrm radj} 
\end{matrix}
\right\}
\end{align}
The functors which possess these equivalent properties are called \em
spherical\rm, cf. \cite{AnnoLogvinenko-SphericalDGFunctors}. 
We thereby define:
\begin{defn*}[Definition \ref{defn-spherical-objects}]
An object $E \in D(Z \times X)$ is \em spherical over $Z$ 
if the corresponding Fourier--Mukai transform $\Phi_E\colon D(Z)
\rightarrow D(X)$ is spherical. In other words, if:
\begin{enumerate}
\item The co-twist $F_E$ is an autoequivalence of $D(Z)$.  
\item The natural transformation $\Phi^{radj}_E
\xrightarrow{\eqref{eqn-sphericity-condition-natural-isomorphism-cotwist}}
 F_E \Phi^{\mathrm ladj}_E[1]$
is an isomorphism of functors. 
\end{enumerate}
\end{defn*}
When $Z = \spec \mathbb{C}$ this is equivalent to
Definition  \ref{defn-original-spherical-objects-of-Seidel-and-Thomas} above
(Example \ref{exmpl-the-case-where-base-is-a-single-point}). 
It also explains why most of the examples over a non-trivial base $Z$ 
came from subschemes of $X$ fibred 
over $Z$. These are the cases when the autoequivalence 
$F_E$ has a particularly nice form. Indeed, for such fibrations
the Fourier--Mukai kernel of $F_E$ must be supported 
on the diagonal $\Delta$ of $Z \times Z$ (Lemma
\ref{lemma-support-of-Q-is-contained-within-the-diagonal}), and 
an autoequivalence of $D(Z)$ is supported on $\Delta$ if and only
if it is simply tensoring by some shifted line bundle $\mathcal{L}_E$ in $D(Z)$
(Prop. \ref{prps-tfae-left-cotwist-is-an-equivalence}). This makes 
the Fourier--Mukai kernel of $\Phi^{\mathrm radj}_E \Phi_E$, a certain
$\rder \shhomm$ complex, into an extension of 
$\Delta_* \mathcal{L}_E$ by $\Delta_* \mathcal{O}_X$. 
Pointwise, this becomes a familiar condition that
a certain $\rder\homm$ complex is $\mathbb{C} \oplus \mathbb{C}[d]$
for some $d \in \mathbb{Z}$. 

In Section \ref{section-orthogonally-spherical-objects} of the
present paper we show that this argument can be made very general. Let 
$Z$ and $X$ be arbitrary schemes of finite type over 
an algebraically closed field $k$ of characteristic $0$. 
No assumptions of smoothness or projectivity are made. 
Instead we make two assumptions on 
the object $E \in D(Z \times X)$: $E$ is perfect 
(locally quasi-isomorphic to a bounded complex of free sheaves) 
and the support of $E$ is proper over $Z$ and over $X$. 
These are necessary for $\Phi_E$ to have adjoints which are again 
Fourier--Mukai transforms. 
We then categorify the notion of ``a subscheme of $X$ fibred 
over $Z$''. The graphs of such subschemes in $Z \times X$ 
are characterised by the property that their fibres over points of $Z$
are mutually disjoint in $X$. In derived categories the notion of
disjointness is expressed by orthogonality - vanishing 
of all $\homm$'s between two objects. Thus the objects we want
are the objects in $D(Z\times X)$ which are {\it orthogonal} over
$Z$, i.e. their fibres over points of $Z$ are pairwise orthogonal in $D(X)$. 
In Lemma \ref{lemma-support-of-Q-is-contained-within-the-diagonal} we
show that $E$ is orthogonal over $Z$ if and only if the support 
of the Fourier--Mukai kernel of the co-twist $F_E$ 
lies within the diagonal $\Delta$ of $Z \times Z$.
It follows that such $F_E$ is an autoequivalence if and only if
it is the functor of tensoring by some invertible 
(locally a shifted line bundle) object of $D(Z)$.
On the other hand, define an object $\mathcal{L}_E$ 
to be the cone of  
\begin{align}
\label{eqn-morphism-defining-L_E-intro}
\mathcal{O}_Z \xrightarrow{
\text{Definition \ref{defn-object-L_E-for-abstract-objects-of-D-ZxX}}}
 \pi_{Z*} \rder\shhomm_{Z \times X}(\pi_{X*} \pi^*_X E, E) 
\quad \quad \quad \pi_Z, \pi_X \text{ are projections } 
Z \times X \rightarrow Z,X
\end{align}
We show in Prop.~\ref{prps-tfae-left-cotwist-is-an-equivalence}
that if $\mathcal{L}_E$ is invertible then 
necessarily $F_E \simeq (-) \otimes \mathcal{L}_E$. 
To check whether $\mathcal{L}_E$ is invertible we restrict 
\eqref{eqn-morphism-defining-L_E-intro} to points of $Z$, 
whence we obtain our main theorem. 
\begin{theorem*}[Theorem \ref{theorem-sphericity-for-orthogonal-objects-of-ZxX}]
Let $Z$ and $X$ be two separable schemes of finite type over $k$.
Let $E$ be a perfect object of $D(Z \times X)$ orthogonal over $Z$ 
and proper over $Z$ and $X$. 
Then $E$ is spherical over $Z$ if and only if:
\begin{enumerate}
\item
\label{item-rhom-E-E_p-intro}
For every closed point $p \in Z$ such that the fibre $E_p$
is not zero  
$$ \rder\homm_{X}(\pi_{X *} E, E_p) = k \oplus k[d_p]
\quad\text{ for some } d_p \in \mathbb{Z}$$
and the natural morphism $\pi_{X *} E 
\xrightarrow{\eqref{eqn-nat-morphism-pi_X*-E-to-E_p}}
E_p$ is non-zero. 
\item 
\label{item-canonical-morphism-alpha-intro}
The canonical morphism $\alpha$ (see Definition \ref{defn-canonical
morphism-alpha-E}) is an isomorphism:
\begin{align*}
E^\vee \otimes \pi^!_X(\mathcal{O}_X) 
\xrightarrow{\alpha}
E^\vee \otimes \pi^!_{Z}(\mathcal{L}_E)
\end{align*}
\end{enumerate}
\end{theorem*}
Interestingly, a similar statement can be made for kernels
of Fourier--Mukai equivalences, cf. Example
\ref{exmpl-the-case-of-Fourier--Mukai-equivalences}.

If $E$ is orthogonally spherical, then  
$d_p$ in $(\ref{item-rhom-E-E_p-intro})$ has to be constant
on every connected component of $Z$. We show further in 
Prop.~\ref{prps-the-shift-of-L_E-is-the-difference-in-dimensions} 
that for any Gorenstein $(z,x) \in \supp_{Z \times X} E$
we have 
\begin{align} 
\label{intro-d_p-in-terms-of-dim-Z-and-X}
d_{z} = -(\dim_{x} X - \dim_{z} Z).
\end{align}

The canonical morphism $\alpha$ in \eqref{item-canonical-morphism-alpha-intro}
is a morphism of Fourier--Mukai kernels which induces the natural
transformation 
$\Phi^{\mathrm ladj} \rightarrow F_E \Phi^{\mathrm radj}[1]$
in Definition \ref{defn-object-L_E-for-abstract-objects-of-D-ZxX}. 
Due to this indirect definition it may be very difficult, 
even in simple cases, to write $\alpha$ 
down explicitly and check that it is an isomorphism. It may be
similarly difficult to check that  
$\pi_{X *} E \xrightarrow{\eqref{eqn-nat-morphism-pi_X*-E-to-E_p}} E_p$ 
is non-zero in $\eqref{item-rhom-E-E_p-intro}$.  In 
\S\ref{section-the-canonical-morphism-alpha} we show that
when applying Theorem \ref{theorem-sphericity-for-orthogonal-objects-of-ZxX} 
in the `if' direction we can omit both of these awkward checks 
whenever the integer $d_p$ in condition \eqref{item-rhom-E-E_p-intro} is always
negative, cf. 
Theorem \ref{theorem-criterion-for-sphericity-when-d_p-negative}. 
For $Z$ and $X$ reasonably nice e.g. abstract varieties
this corresponds by 
\eqref{intro-d_p-in-terms-of-dim-Z-and-X}
to the case where $\dim Z < \dim X$.  

Setting $Z = \spec \mathbb{C}$ in 
Theorem \ref{theorem-sphericity-for-orthogonal-objects-of-ZxX}
turns conditions $(\ref{item-rhom-E-E_p-intro})$ and 
$(\ref{item-canonical-morphism-alpha-intro})$ into  
the original definition of a spherical 
object $E$ in 
\cite{SeidelThomas-BraidGroupActionsOnDerivedCategoriesOfCoherentSheaves}.
Similarly, setting $Z = \spec R$ for some local artinian $\mathbb{C}$-algebra
$R$ yields the definition of an $R$-spherical object $E$ in
\cite{Toda-OnACertainGeneralizationOfSphericalTwists}, \S2. Note
that we also obtain the converse implication  - if $T_E$ is an
auto-equivalence of $D(X)$ which identifies the left and right adjoints
of $\Phi_E$, then $E$ has to satisfy the conditions 
$(\ref{item-rhom-E-E_p-intro})$ and 
$(\ref{item-canonical-morphism-alpha-intro})$ of Theorem
\ref{theorem-sphericity-for-orthogonal-objects-of-ZxX}. 

In Section \ref{section-spherical-fibrations} we reconsider the case 
of flat fibrations. Let $\xi\colon W \hookrightarrow X$ be a subscheme
with $\pi\colon W \rightarrow Z$ a flat and surjective map. 
We apply the results of Section \ref{section-orthogonally-spherical-objects} 
to $\mathcal{O}_W$ in $D(Z \times X)$. One of our goals 
is to understand what geometric properties a spherical fibration
must possess. The two
technical assumptions on the object $E$ in Section 
\ref{section-orthogonally-spherical-objects} translate to
the assumptions of the fibres of $W$ over $Z$ being proper and of 
$\mathcal{O}_W$ being a perfect object of $D(Z \times X)$. 
We first give the most general analogue of  
Theorem \ref{theorem-sphericity-for-orthogonal-objects-of-ZxX} 
which applies to any flat fibration $W$ with the above properties
(Theorem \ref{theorem-sphericity-for-perfect-flat-fibrations}). 
We improve on it for the case when either the fibres of $W$
are Gorenstein schemes or $\xi$ is a Gorenstein map, noting
that for any spherical $W$ these two conditions are, 
in fact, equivalent 
(Prop. \ref{prps-sphericity-for-Gorenstein-flat-fibrations}). 
Finally, we treat the case when the immersion $\xi$ is \em regular\rm, 
i.e. locally on $X$ the ideal of $W$ is generated by a regular
sequence. In such case the cohomology sheaves 
of $\xi^* \xi_* \mathcal{O}_W$ are 
the vector bundles $\wedge^j \mathcal{N}^\vee$ where $\mathcal{N}$
is the normal sheaf of $W$ in $X$. The object $\xi^* \xi_* \mathcal{O}_W$
is the key to computing the $\ext$ complex in the condition 
$(\ref{item-rhom-E-E_p-intro})$
of Theorem \ref{theorem-sphericity-for-orthogonal-objects-of-ZxX}
and therefore $(\ref{item-rhom-E-E_p-intro})$ 
can be deduced via a spectral sequence 
argument from fibrewise vanishing of the 
cohomology of $\wedge^j \mathcal{N}$. In fact, 
the reverse implication can also be obtained if
the complex $\xi^* \xi_* \mathcal{O}_W$ actually splits
up as a direct sum of $\wedge^j \mathcal{N}^\vee[j]$. In 
\cite{ArinkinCaldararu-WhenIsTheSelfIntersectionOfASubvarietyAFibration}
Arinkin and Caldararu had shown that for a smooth $X$ this happens if and
only if $\mathcal{N}$ extends to the first infinitesimal 
neighborhood of $W$ in $X$, e.g. 
when $W$ is carved out by a section of a vector bundle, or when $W$ is the fixed locus of 
a finite group action, or when $\xi$ can be split. For any 
regular immersion $\xi$ we say that it is \em Arinkin-Caldararu \rm 
if $\xi^* \xi_* \mathcal{O}_W$ splits up as the direct sum of its
cohomology sheaves. Then: 
\begin{theorem*}[Theorem \ref{theorem-sphericity-for-regular-immersions}]
Let $W$ be a regularly immersed flat and perfect fibration in $X$ over
$Z$ with proper fibres. Let $\mathcal{N}$ 
be the normal sheaf of $W$ in $X$. Then $W$ is spherical 
if for any closed point $p \in Z$ the fibre $W_p$ is 
a connected Gorenstein scheme and
\begin{enumerate}
\item 
\label{eqn-cohomological-vanishing-of-the-normal-sheaf-intro}
$H^i_{W_p}(\wedge^j \mathcal{N}|_{W_p}) = 0$  unless $i = j = 0$
or $i = \dim W_p \;,\; j = \codim_X W$.
\item 
\label{eqn-restriction-of-the-normal-sheaf-intro}
$(\omega_{W/X})|_{W_p} \simeq \omega_{W_p}$ where 
$\omega_{W_p}$ is the dualizing sheaf of $W_p$
and $\omega_{W/X} = \wedge^{\codim_X W} \mathcal{N}$.
\end{enumerate}
Conversely, if $W$ is spherical, then each fibre $W_p$ is
a connected Gorenstein scheme and
(\ref{eqn-restriction-of-the-normal-sheaf-intro}) holds. 
If $\xi$ is an Arinkin-Caldararu immersion, then 
$(\ref{eqn-cohomological-vanishing-of-the-normal-sheaf-intro})$ also holds. 
\end{theorem*}
The \em `If' \rm implication here generalises the result in
\cite{Horja-DerivedCategoryAutomorphismsFromMirrorSymmetry},
where $Z$, $W$ and $X$ are assumed to be smooth.  
The same argument works for any object in $D(W)$ 
and not just $\mathcal{O}_W$. We also obtain the converse
implication. Any spherical fibration $W$ which 
is Arinkin-Caldararu must therefore satisfy 
$H_{W_p}^i(\mathcal{O}_{W_p}) = 0$ for all $i > 0$, which matches 
the fact that in the known examples the fibres of spherical 
fibrations are Fano varieties.  

Section \ref{section-preliminaries} contains the preliminaries necessary for all of the above. 
In \S\ref{section-adjunction-units-and-fm-transforms} we
work out explicitly the morphisms of kernels which underly the left and 
right adjunction units of a general Fourier--Mukai functor. We 
need this to compute $F_E$ since co-twist functors need to be 
defined as the cones of adjunction units. We get this result
for free from the similar result for adjunction co-units in 
\cite{AnnoLogvinenko-OnTakingTwistsOfFourierMukaiFunctors} using
the Grothendieck duality arguments summarized in 
\S\ref{section-on-duality-theories}. We then review the formalism 
of spherical functors in Section 
\S\ref{section-twists-co-twists-and-spherical-functors}. 

Finally, in the Appendix
we give an example of an orthogonally spherical object which is 
not a spherical fibration and which is a genuine complex and 
not just a shifted sheaf. It arises naturally when constructing 
an affine braid group action on $(n,n)$-fibre of the Grothendieck-Springer 
resolution of the nilpotent cone of
$\mathfrak{s}\mathfrak{l}_{2n}(\mathbb{C})$. 
The authors hope that the tools developed in this paper will allow
to construct more examples of orthogonally spherical objects which 
aren't sheaves and to study explicitly the derived autoequivalences
which they induce. 

\em Acknowledgements: \rm We would like to thank Will Donovan, Miles Reid, 
and Richard Thomas for enlightening discussions in the course of  
this manuscript's preparation. 
The second author did most of his work on this paper at 
the University of Warwick and would like to thank it for being 
a helpful and stimulating research environment. 

\section{Preliminaries}
\label{section-preliminaries}

\em Notation: \rm Throughout the paper we define our schemes over 
the base field $k$ which is assumed to be an 
algebraically closed field of characteristic $0$. We also 
denote by $\vectspaces$ the category of finite-dimensional 
vector spaces over $k$. 
Given a fibre product $X_1 \times \dots \times X_n$ we usually
denote by $\pi_i$ the projection
$X_1 \times \dots \times X_n \rightarrow X_i$ onto the $i$-th
component. 

Let $X$ be a scheme. We denote by $D_{\text{qc}}(X)$, resp. $D(X)$, 
the full subcategory of the derived category of
$\mathcal{O}_X$-$\modd$ consisting of complexes with 
quasi-coherent, resp. bounded and coherent, cohomology.
Given an object $E$ in $D(\mathcal{O}_X\text{-}\modd)$ 
we denote by $\mathcal{H}^i(E)$ the $i$-th cohomology sheaf of $E$
and by $E^\vee$ its derived dual, 
the object $\rder\shhomm_{X}(E, \mathcal{O}_X)$.  

\bf All the functors in this paper are assumed to be derived 
unless mentioned otherwise. \rm

We therefore omit all the usual $\rder$'s and $\lder$'s. An exception 
is made for the derived bi-functor $\rder\homm_X(-,-)$ of taking the 
space of morphisms between a pair of sheaves in $\cohcat(X)$.  
This is to distinguish for any $A,B \in D(X)$ 
the complex $\rder\homm_X(A,B)$ in
$D(\vectspaces)$ from the vector space $\homm_{D(X)}(A,B)$ 
of morphisms from $A$ to $B$ in $D(X)$. Another exception 
was made for the derived bi-functor $\rder\shhomm_X(-,-)$ of taking
the sheaf of morphisms between a pair of sheaves. This is for it 
to still look like a curly version of $\rder\homm_X(-,-)$. 

All the categories we consider are most certainly $1$-categories. 
However given a morphism
$A \rightarrow B$ in a category we can consider it as a 
(trivial) commutative diagram. For two commutative diagrams 
of the same shape there is a well defined notion of them 
being isomorphic, e.g. in our case $A \rightarrow B$ is isomorphic 
to another diagram $A' \rightarrow B'$ if and only if 
there exist isomorphisms which make the square
\begin{align*}
\xymatrix{
A \ar[d]_{\simeq} \ar[r]  &
B \ar[d]^{\simeq} \\
A' \ar[r] &
B'
}
\end{align*}
commute. Sometimes as an abuse of notation we describe this 
by saying that morphism $A \rightarrow  B$ 
is `isomorphic' to morphism $A' \rightarrow B'$. 
Clearly this imposes an equivalence relation on the set of morphisms 
in a given category. This equivalence relation is important in the context 
of triangulated categories because all the morphisms in the same equivalence 
class have isomorphic cones.
 
\subsection{On duality theories}
\label{section-on-duality-theories}

The standard reference on Grothendieck-Verdier duality has for 
some time been \cite{Hartshorne-Residues-and-Duality}. There 
the duality theory is constructed by hand in a (comparatively) geometric and
(comparatively) painful fashion. For a more modern and (comparatively) 
more elegant categorical approach which obtains the existence of
the right adjoint to $f_*$ by pure thought we can recommend 
the reader Lipman's excellent exposition in
\cite{Lipman-NotesOnDerivedFunctorsAndGrothendieckDuality} which 
expands greatly on the Deligne's elegant but brief note
\cite{DeligneCohomologieASupportPropre}. 
Below we give a brief overview of the results we intend to use. 
Our approach relies heavily on the notion of a perfect  
object in a derived category, both in an absolute sense and relative to a
morphism. The reader may find this discussed at length in 
\cite{IllusieGeneralitesSurLesConditionsDeFinitudeDansLesCategoriesDerivees}
and \cite{IllusieConditionsDeFinitudeRelative}. 

Let $S$ be a Noetherian scheme. Let $\fintype_S$ be the category
of separated schemes of finite type over $S$ whose morphisms
are separated $S$-scheme maps of finite type. 
We have the following (relative) duality theory $D_{\bullet/S}$ 
for schemes in $\fintype_S$: for any $X \xrightarrow{f} S$ 
let $D_{X/S}$ denote the functor 
$\rder \shhomm\left(-,f^! \mathcal{O}_X\right)$ 
from $D(\mathcal{O}_X\text{-}\modd)$ to 
$D(\mathcal{O}_X\text{-}\modd)^{op}$. Here $(-)^!$ is the twisted
inverse image pseudo-functor, cf. 
\cite[Theorem 4.8.1]{Lipman-NotesOnDerivedFunctorsAndGrothendieckDuality}. 
It follows from
\cite[Cor. 4.9.2]{IllusieConditionsDeFinitudeRelative}
that $D_{X/S}$ takes $D_{S\text{-}\perf}(X)$, 
the full subcategory of $D(X)$ consisting 
of objects perfect over $S$, to itself in the opposite category 
and the restriction is a self-inverse equivalence 
$$ D_{X/S}\colon\quad D_{S\text{-}\perf}(X) \xrightarrow{\sim}
D_{S\text{-}\perf}(X)^{op}. $$

Now, given any two schemes $X$ and $Y$ in $\fintype_S$ and any exact functor 
$F\colon D_{S\text{-}\perf}(X) \rightarrow D_{S\text{-}\perf}(Y)$
we define its dual under $D_{\bullet/S}$ to be 
the functor $D_{Y/S}\; F\; D_{X/S}\colon 
D_{S\text{-}\perf}(X) \rightarrow D_{S\text{-}\perf}(Y)$. 
The double-dual of a functor is then the functor itself and we say 
that $F$ and $D_{Y/S}\; F\; D_{X/S}$ are \em dual under
$D_{\bullet/S}$\rm. 
The (contravariant) notion of a dual of a morphism of functors is 
defined accordingly. One can then easily see that if a functor has 
a left (resp. right) adjoint then $D_{\bullet/S}$ sends it to 
the right (resp. left) adjoint of its dual and interchanges 
the adjunction units with the adjunction co-units. 

Let $X$ be a scheme in $\fintype_S$ and let $E$ be a perfect 
(in an absolute sense) object of $D(\mathcal{O}_X\text{-}\modd)$. 
Then the functor $E \otimes(-)$ takes $D_{S\text{-}\perf}(X)$
to $D_{S\text{-}\perf}(X)$, its adjoint, 
both left and right, is the functor $E^\vee \otimes(-)$
and for any $F \in D(\mathcal{O}_X\text{-}\modd)$ we have
by 
\cite[Lemma 1.4.6]{AvramovIyengarLipman-ReflexivityAndRigidityForComplexesIISchemes}
a natural isomorphism
\begin{align}
\label{eqn-S-perfect-duality-for-tensor-product}
D_{X/S}(E \otimes F) \xrightarrow{\sim} E^\vee \otimes D_{X/S} F.
\end{align}
Therefore $E \otimes (-)$ and $E^\vee \otimes (-)$ 
are dual under $D_{\bullet/S}$. Consequently,
$D_{\bullet/S}$ interchanges the adjunction 
unit $\id \rightarrow E^\vee \otimes E \otimes (-) $ and 
the adjunction co-unit $E^\vee \otimes E \otimes \rightarrow \id$. 

Let $X \xrightarrow{f} Y$ be a proper map in $\fintype_S$. Then 
$f_*$ sends $D_{S\text{-}\perf}(X)$ to $D_{S\text{-}\perf}(Y)$.
By the sheafifed Grothendieck duality, cf. 
\cite[Cor.
4.4.2]{Lipman-NotesOnDerivedFunctorsAndGrothendieckDuality},
for any $E \in D_{qc}(X)$ the natural map 
\begin{align} \label{eqn-S-perf-duality-for-f-pushforward}
D_{Y/S}(f_* E) \xrightarrow{\sim} f_*(D_{X/S} E). 
\end{align}
is an isomorphism. It follows that $f_*$ is self-dual under $D_{\bullet/S}$. 

On the other hand, let $X \xrightarrow{f} Y$ be any map
in $\fintype_S$. By \cite[Prop. 4.10.1]{Lipman-NotesOnDerivedFunctorsAndGrothendieckDuality} 
there is for any $E \in D(Y)$ a natural isomorphism 
\begin{align} \label{eqn-S-perf-duality-for-f-pullback}
D_{X/S}(f^*E) \xrightarrow{\sim} f^!(D_{Y/S} E). 
\end{align}
If $f^*$ takes 
$D_{S\text{-}\perf}(Y)$ to $D_{S\text{-}\perf}(X)$, e.g. $f$ is perfect, 
it follows that $f^*$ and $f^!$ are dual under $D_{\bullet/S}$. 

Note that $f^*$ is the left adjoint of $f_*$ and, if $f$ is proper, $f^!$
is its right adjoint. So for $f$ proper and perfect $f^*$ and $f^!$
being dual under $D_{\bullet/S}$ is precisely equivalent to $f_*$ 
being self-dual.  

If $f$ is proper, then even 
if $f^*$ doesn't take $S$-perfect objects to $S$-perfect objects, 
it still follows from 
the definitions of maps \eqref{eqn-S-perf-duality-for-f-pushforward} 
and \eqref{eqn-S-perf-duality-for-f-pullback} in
\cite{Lipman-NotesOnDerivedFunctorsAndGrothendieckDuality} that
for any $E \in D(Y)$ the following diagram commutes
\begin{align} \label{eqn-duality-functor-still-sends-unit-to-co-unit}
\xymatrix{
D_{Y/S} \left(f_* f^* E\right) \ar[rr]^<<<<<<<<<<<<<{(\id \rightarrow f_* f^*)^{opp}} 
\ar[d]^{\simeq}_{\eqref{eqn-S-perf-duality-for-f-pushforward} + 
\eqref{eqn-S-perf-duality-for-f-pullback}} 
& \quad & 
D_{Y/S} E 
\\
f_* f^! D_{Y/S} E \ar[urr]_{f_* f^! \rightarrow \id} 
& \quad &
}
\end{align}
i.e. $D_{\bullet/S}$ still send the adjunction unit $\id \rightarrow
f_* f^*$ to the adjunction co-unit $f_* f^! \rightarrow \id$. 
We then also have:
\begin{lemma}
\label{lemma-natural-map-f^!=f^*-f^!-is-iso-for-perf-and-S-perf}
Let $X \xrightarrow{f} Y$ be any map in $\fintype_S$, 
let $E$ be a perfect object in $D(\modd\text{-}Y)$ and 
let $F$ be an $S$-perfect object in $D(\modd\text{-}Y)$. 
Then the natural map 
\begin{align} \label{eqn-natural-map-f^!=f^*-f^!}
f^* E \otimes f^! F \rightarrow f^!(E \otimes F) 
\end{align}
is an isomorphism.
\end{lemma}
\begin{proof}
By compactification
\cite{Nagata-ImbeddingOfAnAbstractVarietyInACompleteVariety}  
such $f$ decomposes into an open immersion followed
by a proper map.

If $f$ is an open immersion, then the map 
\eqref{eqn-natural-map-f^!=f^*-f^!} is by definition the 
isomorphism $f^* E \otimes f^* F \rightarrow f^*(E \otimes F)$. 

It remains to consider the case of $f$ being a proper map. 
Then, by definition, the map $\eqref{eqn-natural-map-f^!=f^*-f^!}$
is the right adjoint with respect to $f_*$ of the composition
\begin{align}
\label{eqn-left-adjoint-of-natural-map-f^!=f^*-f^!}
f_*(f^* E \otimes f^! F) 
\xrightarrow{\text{inverse of projection formula map}}
E \otimes f_* f^! F 
\xrightarrow{f_* f^!  \rightarrow \id}
E \otimes F.
\end{align}
Using the duality isomorphism $D_{\bullet/S} D_{\bullet/S} F \simeq F$,
isomorphisms $\eqref{eqn-S-perfect-duality-for-tensor-product}$-$\eqref{eqn-S-perf-duality-for-f-pullback}$ and  
$\eqref{eqn-duality-functor-still-sends-unit-to-co-unit}$, we 
can re-write $\eqref{eqn-left-adjoint-of-natural-map-f^!=f^*-f^!}$
as 
\begin{align}
\label{eqn-dual-of-left-adjoint-of-natural-map-f^!=f^*-f^!}
D_{\bullet/S} \quad
\left(
E^\vee \otimes D_{\bullet/S} F
\xrightarrow{\id \rightarrow f_* f^*}
E^\vee \otimes f_* f^* D_{\bullet/S} F 
\xrightarrow{\text{projection formula map}}
f_*(f^* E^\vee \otimes f^* D_{\bullet/S} F) 
\right)^{\text{opp}}
\end{align}
which by 
\cite{AnnoLogvinenko-OnTakingTwistsOfFourierMukaiFunctors}, Lemma 2.1
is the same map as 
\begin{align}
\label{eqn-dual-of-left-adjoint-of-dual-of-f^*-AB-=-f^*-A-f^*-B-iso}
D_{\bullet/S} \quad 
\left( E^\vee \otimes D_{\bullet/S} F
\xrightarrow{\id \rightarrow f_* f^*}
f_* f^*(E^\vee \otimes D_{\bullet/S} F)
\xrightarrow{ f^*(\text{-} \otimes \text{-}) \rightarrow f^* \otimes f^* }
f_* \left(f^* E^\vee \otimes f^* D_{\bullet/S} F \right)
\right)^{\text{opp}}.
\end{align} 
Using
$\eqref{eqn-S-perfect-duality-for-tensor-product}$-$\eqref{eqn-S-perf-duality-for-f-pullback}$,
$\eqref{eqn-duality-functor-still-sends-unit-to-co-unit}$ and
$D_{\bullet/S} D_{\bullet/S} F \simeq F$
again, we deduce that \eqref{eqn-left-adjoint-of-natural-map-f^!=f^*-f^!}
is the same map as 
\begin{align}
\label{eqn-left-adjoint-of-dual-of-f^*-AB-=-f^*-A-f^*-B-iso}
f_* (f^* E \otimes f^! F)
\xrightarrow{f_* \alpha} 
f_* f^! (E \otimes F)
\xrightarrow{f_* f^! \rightarrow \id}
E \otimes F
\end{align}
where the map $\alpha$ is isomorphic to  
\begin{align}
\label{eqn-dual-of-f^*-AB-=-f^*-A-f^*-B-iso}
D_{\bullet/S} \quad
\left( f^*(E^\vee \otimes D_{\bullet/S} F) 
\xrightarrow{ f^*(\text{-} \otimes \text{-}) \rightarrow f^* \otimes f^* }
f^*E^\vee \otimes f^* D_{\bullet/S} F \right)^\text{opp}.
\end{align}
Since the right adjoint of 
$\eqref{eqn-left-adjoint-of-dual-of-f^*-AB-=-f^*-A-f^*-B-iso}$ 
with respect to $f_*$ is clearly $\alpha$, we conclude
that the natural map $\eqref{eqn-natural-map-f^!=f^*-f^!}$ is
precisely the map $\alpha$ which is evidently an isomorphism.
\end{proof}

In the special case of $S = \spec k$ the category $\fintype_{k}$
is simply the category of all schemes of finite type over $k$.
For any such scheme $X$ we have $D_{S\text{-}\perf}(X) = D^b_{coh}(X)$. 
The resulting duality theory $D_{\bullet/k}$ is the usual duality theory 
of \cite{Hartshorne-Residues-and-Duality} with $D_{X /k}(\mathcal{O}_X)$ 
being dualizing complexes in sense of \cite{Hartshorne-Residues-and-Duality}, 
Chapter V.

On the other hand we have the perfect duality theory which exists in 
the category of arbitrary schemes. Let $X$ be a scheme and let 
$\dperf_{X}$ denote the functor 
$\rder \shhomm\left(-,\mathcal{O}_X\right)$ 
from $D(\mathcal{O}_X\text{-}\modd)$ to 
$D(\mathcal{O}_X\text{-}\modd)^{op}$, i.e. $\dperf_X (E) = E^\vee$. 
It is shown in
\cite{IllusieGeneralitesSurLesConditionsDeFinitudeDansLesCategoriesDerivees},
\S7 that $\dperf_X$ takes $D_{\perf}(X)$, 
the full subcategory of $D(\modd\text{-}\mathcal{O}_X)$ consisting 
of perfect objects, to itself in the opposite category 
and the restriction is a self-inverse equivalence 
$$ \dperf_X\colon\quad D_{\perf}(X) \xrightarrow{\sim} D_{\perf}(X)^{op}. $$

Then, given any two schemes $X$ and $Y$, we define just as above
the notions of a dual under $\dperf$ of any functor $F\colon 
D_{\perf}(X) \rightarrow D_{\perf}(Y)$ and of any natural
transformation between two such functors. Once again, the duality
interchanges left adjoints with right adjoints and the adjunction 
units with the adjunction counits. 

Let $X \xrightarrow{f} Y$ be any scheme map. Then 
$f^*$ sends $D_{\perf}(Y)$ to $D_{\perf}(X)$ and we have 
(\cite{IllusieGeneralitesSurLesConditionsDeFinitudeDansLesCategoriesDerivees},
Prps. 7.1.2) for any $E \in D_{\perf}(Y)$ a natural isomorphism
\begin{align}
\dperf_X f^* E \xrightarrow{\sim} f^* \dperf_Y E. 
\end{align}
It follows that $f^*$ is self-dual under $\dperf$. 

Now let $X \xrightarrow{f} Y$ be any scheme map such that 
$f_*$ sends $D_{\perf}(X)$ to $D_{\perf}(Y)$, e.g. a quasi-perfect map 
of concentrated schemes
(\cite{Lipman-NotesOnDerivedFunctorsAndGrothendieckDuality}, \S 4.7).
Then, since $f^*$ is self-dual, the dual of $f_*$ under $\dperf$
is the left adjoint $f_{\dagger}$ of $f^*$. And when $f$ is a separated
finite-type perfect map of Noetherian schemes, we know
(\cite{AvramovIyengarLipman-ReflexivityAndRigidityForComplexesIISchemes}, 
Lemma 2.1.10) that $f_\dagger(-)$ is naturally isomorphic to 
$f_*(f^!(\mathcal{O}_Y) \otimes -)$ in a way which makes the 
composition 
$$ f_\dagger f^* \xrightarrow{\sim} f_*\left(f^!(\mathcal{O}_Y) \otimes
f^*(-)\right) \xrightarrow{\sim} f_* f^! \xrightarrow{\text{the adjunction
co-unit}} \id $$
be precisely the adjunction co-unit $f_\dagger f^* \rightarrow \id$. 

\subsection{Adjunction units and Fourier--Mukai transforms}
\label{section-adjunction-units-and-fm-transforms}

The definition of a spherical functor $S$ in
\cite{AnnoLogvinenko-SphericalDGFunctors}
demands for $S$ to be a Fourier-Mukai transform whose left and right 
adjoints $L$ and $R$ are also Fourier Mukai transforms. Moreover, 
to compute the twist $T_S$ and the co-twist $F_S$ of $S$ we need
to write down the units and the co-units of these adjunctions on 
the level of Fourier-Mukai kernels. 

Partly this was achieved in \S3.1 of
\cite{AnnoLogvinenko-OnTakingTwistsOfFourierMukaiFunctors}. 
We give a brief summary here. Quite generally, 
let $X_1$ and $X_2$ be two separated proper schemes 
of finite type over $k$ and let $E$ be
a perfect object in $D(X_1 \times X_2)$. We have a commutative 
diagram of projection morphisms:
\begin{align} \label{eqn-big-projection-tree}
\xymatrix{
& & X_1 \times X_2 \times X_1 \ar[ld]_{\pi_{12}} \ar[d]^{\pi_{13}} \ar[rd]^{\pi_{23}} & & \\
& X_1 \times X_2 \ar[ld]_{\pi_1} \ar[rd]^>>>>>>>{\pi_2} & X_1 \times X_1
\ar[lld]_>>>>>>>>>>>>>>{\tilde{\pi}_1}
\ar[rrd]^>>>>>>>>>>>>>>{\tilde{\pi}_2} & X_2 \times X_1 \ar[ld]_>>>>>>>{\pi_2} \ar[rd]^{\pi_1} & \\
X_1 & & X_2 & & X_1
}
\end{align}

Let $\Phi_E\colon D(X_1) \rightarrow D(X_2)$ be the Fourier--Mukai 
transform $\pi_{2 *}\left(E \otimes \pi^*_1(-)\right)$ with kernel 
$E$, then:
\begin{enumerate}
\item A left adjoint $\Phi^{\mathrm ladj}_E$ to $\Phi_E$ exists and 
is isomorphic to the Fourier--Mukai transform $\Psi_{E^\vee \otimes
\pi_1^!(\mathcal{O}_{X1})}$ from $D(X_2)$ to $D(X_1)$. 
\item A right adjoint $\Phi^{\mathrm radj}_E$ to $\Phi_E$ exists and is isomorphic to 
the Fourier--Mukai transform $\Psi_{E^\vee \otimes
\pi_2^!(\mathcal{O}_{X2})}$ from $D(X_2)$ to $D(X_1)$. 
\item The adjunction co-unit $\Phi^{\mathrm ladj}_E \Phi_E \rightarrow
\id_{D(X_1)}$ is isomorphic to the morphism $\Theta_{Q} \rightarrow 
\Theta_{\mathcal{O}_\Delta}$ of Fourier--Mukai transforms $D(X_1)
\rightarrow D(X_1)$ induced by 
the morphism $Q \rightarrow \mathcal{O}_\Delta$ of objects of 
$D(X_1 \times X_1)$ written down explicitly in
\cite[Theorem 3.1]{AnnoLogvinenko-OnTakingTwistsOfFourierMukaiFunctors}
to which we refer the reader for all the details. 
An analogous statement holds for the adjunction co-unit
$\Phi_E \Phi^{\mathrm radj}_E \rightarrow \id_{D(X_2)}$, cf. 
\cite[Theorem 3.2]{AnnoLogvinenko-OnTakingTwistsOfFourierMukaiFunctors}.  

\item The condition of $X_1$ and $X_2$ being proper can be replaced
by the condition of the support of $E$ being proper over $X_1$ and
over $X_2$, cf. \S2.2 of 
\cite{AnnoLogvinenko-OnTakingTwistsOfFourierMukaiFunctors}.
If $E$ is a pushforward of an object in the
derived category of a closed subscheme $X_1 \times X_2$
proper over $X_1$ and $X_2$, then there is an alternative description
of the morphisms of Fourier--Mukai kernels which produce 
the adjunction co-units
$\Phi^{\mathrm ladj}_E \Phi_E \rightarrow \id_{D(X_1)}$ and 
$\Phi_E \Phi^{\mathrm radj}_E \rightarrow \id_{D(X_2)}$, 
cf. \cite[Theorems 4.1 and 4.2]{AnnoLogvinenko-OnTakingTwistsOfFourierMukaiFunctors}.
\end{enumerate}

What remains to be done is to obtain a similar result for the adjunction 
units $\id_{D(X_1)} \rightarrow \Phi^{\mathrm radj}_E \Phi_E$ 
and $\id_{D(X_2)} \rightarrow \Phi_E \Phi^{\mathrm ladj}_E$. Fortunately
this can be obtained directly from the above results in
\cite{AnnoLogvinenko-OnTakingTwistsOfFourierMukaiFunctors} via
the Grothendieck-Verdier duality in the following way.

The dual of the Fourier--Mukai transform
$$ \Phi_E (-) = \pi_{2 *} \left( E \otimes \pi_1^*(-)\right) $$
under the duality theory $D_{\bullet/k}$ (see Section
\ref{section-on-duality-theories}) is the functor 
\begin{align} \label{eqn-dual-of-fm-transform}
\pi_{2 *} \rder \shhomm\left(E, \pi_1^!(-)\right). 
\end{align}
There are two ways to view this functor. Firstly,
via natural isomorphisms 
\begin{align} \label{eqn-natural-isos-in-dual-of-fm-transform}
\rder\shhomm\left(E,\pi_1^! \mathcal{O}_{X1}\right) \otimes \pi_1^*(-)
\xrightarrow{\sim} \rder\shhomm\left(E,\pi_1^! \mathcal{O}_{X1} \otimes
\pi_1^*(-)\right) \xrightarrow{\sim}
\rder\shhomm\left(E, \pi_1^!(-)\right)
\end{align}
we can identify \eqref{eqn-dual-of-fm-transform} 
with the Fourier--Mukai transform $\Phi_{\rder \shhomm(E,
\pi_1^! \mathcal{O}_{X1})}$ from $D(X_1)$ to $D(X_2)$. Secondly, observe
that \eqref{eqn-dual-of-fm-transform} is 
the right adjoint $\Psi^{\mathrm radj}_E$ of the Fourier--Mukai transform 
$\Psi_E$ from $D(X_2)$ to $D(X_1)$. 

Taking this second point of view, it immediately follows that 
the dual of $\Phi^{\mathrm radj}_E$ is $\Psi_E$ and the 
dual of the adjunction unit 
\begin{align} \label{eqn-fm-right-unit}
\id_{D(X_1)} \rightarrow \Phi^{\mathrm radj}_E \Phi_E 
\end{align}
is the adjunction co-unit
\begin{align} \label{eqn-dual-fm-counit}
\Psi_E \Psi^{\mathrm radj}_E \rightarrow \id_{D(X_1)}. 
\end{align}
By \cite[Theorem 3.2]{AnnoLogvinenko-OnTakingTwistsOfFourierMukaiFunctors},
the adjunction co-unit \eqref{eqn-dual-fm-counit} is isomorphic
to the natural transformation 
\begin{align} \label{eqn-fm-transform-realising-dual-counit} 
\Theta_{\tilde{Q}} \rightarrow \Theta_{\mathcal{O}_\Delta} 
\end{align}
of Fourier--Mukai tranforms $D(X_1) \rightarrow D(X_1)$ induced by
the following morphism of objects of $D(X_1 \times X_1)$: 
\begin{align} \label{eqn-derived-restriction-morphism-radj-version}
\tilde{Q} = \pi_{13 *}\left(\pi_{12}^* E^\vee \otimes \pi_{23}^* E
\otimes \pi_{12}^* \pi^!_1(\mathcal{O}_{X_1})\right) 
\xrightarrow{\id \rightarrow \Delta_* \Delta^*}
\pi_{13 *} \Delta_* \Delta^* 
\left(\pi_{12}^* E^\vee \otimes \pi_{23}^* E \otimes \pi_{12}^*
\pi^!_1(\mathcal{O}_{X_1})\right) \\
\label{eqn-two-squares-isomorphism-radj-version}
\pi_{13 *} \Delta_* \Delta^* 
\left(\pi_{12}^* E^\vee \otimes \pi_{23}^* E \otimes \pi_{12}^*
\pi^!_1(\mathcal{O}_{X_1})\right)
\quad \simeq \quad
\Delta_* \pi_{1 *}  
\left(E^\vee \otimes E \otimes \pi^!_1(\mathcal{O}_{X_1})\right)
\\ 
\label{eqn-E-E-dual-trace-morphism-radj-version}
\Delta_* \pi_{1 *}  
\left(E \otimes E^\vee \otimes \pi^!_1(\mathcal{O}_{X_1})\right) 
\xrightarrow{E^\vee \otimes E \otimes \rightarrow \id}
\Delta_* \pi_{1 *} \left(\pi^!_1(\mathcal{O}_{X_1})\right)  
\\
\label{eqn-rpi_1-trace-morphism-radj-version}
\Delta_* \pi_{1 *} \left(\pi^!_1(\mathcal{O}_{X_1})\right) 
\xrightarrow{\pi_{1 *} \pi^!_1 \rightarrow \id}
\Delta_* \mathcal{O}_{X_1}. 
\end{align}

Identifying\footnote{One has to be a little careful here since
$\mathcal{O}_\Delta$, unlike $\tilde{Q}$, is not a perfect object 
of $D(X_1 \times X_1)$. However, both natural maps in 
the analogue of \eqref{eqn-natural-isos-in-dual-of-fm-transform}
are still isomorphisms so we can still 
make the same identification.} the duals of $\Theta_{\tilde{Q}}$ and 
$\Theta_{\mathcal{O}_\Delta}$ under $D(\bullet/k)$ with 
$\Theta_{\rder \shhomm(\tilde{Q}, \tilde{\pi}_1^! \mathcal{O}_{X_1})}$
and $\Theta_{\rder \shhomm(\mathcal{O}_\Delta, \tilde{\pi}_1^!
\mathcal{O}_{X_1})}$, we see that the dual of 
\eqref{eqn-fm-transform-realising-dual-counit} under $D_{\bullet/k}$
is the morphism of Fourier--Mukai transforms induced by the morphism
\begin{align} 
\rder \shhomm(\mathcal{O}_\Delta, \tilde{\pi}_1^!  \mathcal{O}_{X_1})
\rightarrow \rder \shhomm(\tilde{Q}, \tilde{\pi}_1^!
\mathcal{O}_{X_1})
\end{align} 
obtained by applying the relative dualizing functor
$D_{X_1 \times X_1 /X_1} = \rder \shhomm( - , \tilde{\pi}_1^!  \mathcal{O}_{X_1})$ to 
$\eqref{eqn-derived-restriction-morphism-radj-version}-
\eqref{eqn-rpi_1-trace-morphism-radj-version}$. 

Treating 
$\eqref{eqn-derived-restriction-morphism-radj-version}-
\eqref{eqn-rpi_1-trace-morphism-radj-version}$ as morphisms 
of functors in $\mathcal{O}_{X_1}$ and applying the results of
Section \ref{section-on-duality-theories}, we
see that $D_{X_1 \times X_1 / X_1}$ applied to 
$\eqref{eqn-derived-restriction-morphism-radj-version}-
\eqref{eqn-rpi_1-trace-morphism-radj-version}$ yields:
\begin{align}
\label{eqn-pi1-unit-in-fm-right-adjunction-unit-moprhism-prelim}
\Delta_* D_{X_1/X_1}(\mathcal{O}_{X_1})
\xrightarrow{\id \rightarrow \pi_{1 *} \pi^*_1}
\Delta_* \pi_{1 *} \pi^*_1 D_{X_1/X_1}(\mathcal{O}_{X_1})
\\ 
\label{eqn-E-E-dual-unit-in-fm-right-adjunction-unit-moprhism-prelim}
\Delta_* \pi_{1 *} \pi^*_1 D_{X_1/X_1}(\mathcal{O}_{X_1})
\xrightarrow{\id \rightarrow E \otimes E^\vee \otimes }
\Delta_* \pi_{1 *} \left(E \otimes E^\vee \otimes 
\pi^*_1 D_{X_1/X_1}(\mathcal{O}_{X_1})\right)
\\
\label{eqn-rearranging-isos-in-fm-right-adjunction-unit-moprhism-prelim}
\Delta_* \pi_{1 *} \left(E \otimes E^\vee \otimes 
\pi^*_1 D_{X_1/X_1}(\mathcal{O}_{X_1})\right)
\quad \simeq \quad
\pi_{13 *} \Delta_* \Delta^! 
\left(\pi_{12}^* E \otimes \pi_{23}^* E^\vee \otimes \pi_{12}^!
\pi^*_1 D_{X_1/X_1}(\mathcal{O}_{X_1})\right)
\\ 
\label{eqn-delta-unit-in-fm-right-adjunction-unit-moprhism-prelim}
\pi_{13 *} \Delta_* \Delta^! 
\left(\pi_{12}^* E \otimes \pi_{23}^* E^\vee \otimes \pi_{12}^!
\pi^*_1 D_{X_1/X_1}\mathcal{O}_{X_1}\right)
\xrightarrow{\Delta_* \Delta^! \rightarrow \id}
\pi_{13 *} 
\left(\pi_{12}^* E \otimes \pi_{23}^* E^\vee \otimes \pi_{12}^!
\pi^*_1 D_{X_1/X_1}\mathcal{O}_{X_1}\right) 
\end{align}
By the above
$\eqref{eqn-pi1-unit-in-fm-right-adjunction-unit-moprhism-prelim}$-$\eqref{eqn-delta-unit-in-fm-right-adjunction-unit-moprhism-prelim}$
induces a natural transformation of Fourier--Mukai transforms isomorphic 
to the dual of $\eqref{eqn-fm-transform-realising-dual-counit}$. 
Since $\eqref{eqn-fm-transform-realising-dual-counit}$ is itself
isomorphic to the dual of $\id_{X1} \rightarrow \Phi^{\mathrm radj}_E \Phi_E$, we
conclude that the natural transformation induced by $\eqref{eqn-pi1-unit-in-fm-right-adjunction-unit-moprhism-prelim}$-$\eqref{eqn-delta-unit-in-fm-right-adjunction-unit-moprhism-prelim}$
is isomorphic to $\id_{X1} \rightarrow \Phi^{\mathrm radj}_E \Phi_E$. 
Finally, since $D_{X_1/X_1}(\mathcal{O}_{X_1}) \simeq \mathcal{O}_{X_1}$
and $\pi^!_{12} \pi^*_1(\mathcal{O}_{X_1}) \simeq \pi^*_{23}
\pi^!_2(\mathcal{O}_{X_2})$, we obtain:
\begin{prps} 
\label{prps-right-adjunction-unit-morphism}
Let $X_1$ and $X_2$ be two separated proper schemes of finite type over $k$
and let $E$ be a perfect object of $D(X_1 \times X_2)$. 
Then the adjunction unit $\id_{X1} \rightarrow \Phi^{\mathrm radj}_E \Phi_E$
is isomorphic to the morphism of Fourier--Mukai transforms induced
by the following morphism of objects of $D(X_1 \times X_1)$:
\begin{align}
\label{eqn-pi1-unit-in-fm-right-adjunction-unit-morphism}
\Delta_* (\mathcal{O}_{X_1})
\xrightarrow{\id \rightarrow \pi_{1 *} \pi^*_1}
\Delta_* \pi_{1 *} \pi^*_1 (\mathcal{O}_{X_1})
\\ 
\label{eqn-E-E-dual-unit-in-fm-right-adjunction-unit-morphism}
\Delta_* \pi_{1 *} \pi^*_1 (\mathcal{O}_{X_1})
\xrightarrow{\id \rightarrow E \otimes E^\vee \otimes }
\Delta_* \pi_{1 *} \left(E \otimes E^\vee \otimes 
\pi^*_1 (\mathcal{O}_{X_1})\right)
\\
\label{eqn-rearranging-isos-in-fm-right-adjunction-unit-morphism}
\Delta_* \pi_{1 *} \left(E \otimes E^\vee \otimes 
\pi^*_1 (\mathcal{O}_{X_1})\right)
\quad \simeq \quad
\pi_{13 *} \Delta_* \Delta^! 
\left(\pi_{12}^* E \otimes \pi_{23}^* E^\vee \otimes \pi_{23}^*
\pi^!_2 (\mathcal{O}_{X_2})\right)
\\ 
\label{eqn-delta-unit-in-fm-right-adjunction-unit-morphism}
\pi_{13 *} \Delta_* \Delta^! 
\left(\pi_{12}^* E \otimes \pi_{23}^* E^\vee \otimes \pi_{23}^*
\pi^!_2(\mathcal{O}_{X_2})\right)
\xrightarrow{\Delta_* \Delta^! \rightarrow \id}
\pi_{13 *} 
\left(\pi_{12}^* E \otimes \pi_{23}^* E^\vee \otimes \pi_{23}^*
\pi^!_2(\mathcal{O}_{X_2})\right) 
\end{align}
\end{prps}
In a similar fashion we also obtain:
\begin{prps}
\label{prps-left-adjunction-unit-morphism}
Let $X_1$ and $X_2$ be two separated proper schemes of finite type over $k$
and let $E$ be a perfect object of $D(X_1 \times X_2)$. 
Then the adjunction unit $\id_{X1} \rightarrow \Psi_E \Psi^{\mathrm ladj}_E$
is isomorphic to the morphism of Fourier--Mukai transforms induced
by the following morphism of objects of $D(X_1 \times X_1)$:
\begin{align}
\label{eqn-pi1-unit-in-fm-left-adjunction-unit-morphism}
\Delta_* (\mathcal{O}_{X_1})
\xrightarrow{\id \rightarrow \pi_{1 *} \pi^*_1}
\Delta_* \pi_{1 *} \pi^*_1 (\mathcal{O}_{X_1})
\\ 
\label{eqn-E-E-dual-unit-in-fm-left-adjunction-unit-morphism}
\Delta_* \pi_{1 *} \pi^*_1 (\mathcal{O}_{X_1})
\xrightarrow{\id \rightarrow E \otimes E^\vee \otimes }
\Delta_* \pi_{1 *} \left(E \otimes E^\vee \otimes 
\pi^*_1 (\mathcal{O}_{X_1})\right)
\\
\label{eqn-rearranging-isos-in-fm-left-adjunction-unit-morphism}
\Delta_* \pi_{1 *} \left(E \otimes E^\vee \otimes 
\pi^*_1 (\mathcal{O}_{X_1})\right)
\quad \simeq \quad
\pi_{13 *} \Delta_* \Delta^! 
\left(\pi_{12}^* E^\vee \otimes \pi_{23}^* E \otimes \pi_{12}^*
\pi^!_2 (\mathcal{O}_{X_2})\right)
\\ 
\label{eqn-delta-unit-in-fm-left-adjunction-unit-morphism}
\pi_{13 *} \Delta_* \Delta^! 
\left(\pi_{12}^* E^\vee \otimes \pi_{23}^* E \otimes \pi_{12}^*
\pi^!_2(\mathcal{O}_{X_2})\right)
\xrightarrow{\Delta_* \Delta^! \rightarrow \id}
\pi_{13 *} 
\left(\pi_{12}^* E^\vee \otimes \pi_{23}^* E \otimes \pi_{12}^*
\pi^!_2(\mathcal{O}_{X_2})\right) 
\end{align}
\end{prps}

If $X_1$ and $X_2$ are not proper, but the support of $E$ 
is proper over $X_1$ and $X_2$, one can still apply 
the above results via compactification as described in \S3.2 of 
\cite{AnnoLogvinenko-OnTakingTwistsOfFourierMukaiFunctors}.
If $E$ is a pushforward of an object in the derived category
of a closed subscheme $W \hookrightarrow X_1 \times X_2$ proper over
both $X_1$ and $X_2$ one can also dualize Theorems 4.1 and 4.2 of 
\cite{AnnoLogvinenko-OnTakingTwistsOfFourierMukaiFunctors} to
obtain an alternative description of morphisms of kernels which induce
both adjunction units. We leave this as an exercise for the reader. 
\subsection{Twists, co-twists and spherical functors} 
\label{section-twists-co-twists-and-spherical-functors}

Let $X_1$ and $X_2$ be, as before, two separated proper schemes 
of finite type over $k$. Let $E$ be a perfect object in 
$D(X_1 \times X_2)$ and let $\Phi_E$ be the Fourier--Mukai transform
from $D(X_1)$ to $D(X_2)$ with kernel $E$.
In Section \ref{section-adjunction-units-and-fm-transforms}
we've produced 
morphisms of Fourier-Mukai kernels which induce the adjunction units
and co-units of $\Phi_E$, $\Phi^{ladj}_E$ and $\Phi^{radj}_E$. 
Taking cones of these morphisms allows us to construct the functorial
exact triangles in the following definition:
\begin{defn}
We define \em the twist $T_E$\rm, \em the dual twist $T'_E$\rm,
\em the co-twist $F_E$,\rm and \em the dual co-twist $F'_E$ \rm of
$\Phi_E$ by the functorial exact triangles
\begin{align} \label{eqn-twist-triangle}
\Phi_E  \Phi^{\mathrm radj}_E \rightarrow &\id_{D(X_2)} \rightarrow T_E, \\
\notag 
T'_E \rightarrow &\id_{D(X_2)} \rightarrow \Phi_E  \Phi^{\mathrm ladj}_E, \\ 
\label{eqn-cotwist-triangle} 
F_E \rightarrow &\id_{D(X_1)} \rightarrow 
\Phi^{\mathrm radj}_E  \Phi_E, \\
\notag \Phi^{\mathrm ladj}_E  \Phi_E \rightarrow &\id_{D(X_1)} 
\rightarrow F'_E
\end{align}
constructed via the morphisms of Fourier-Mukai kernels 
produced in Section \ref{section-adjunction-units-and-fm-transforms}. 
\end{defn}
In \cite[Prop. 5.3]{AnnoLogvinenko-SphericalDGFunctors} we proved
that $T'_E$ and $F'_E$ are the left adjoints of $T_E$ and $F_E$,
respectively. 

Consider now the following two natural transformations 
\begin{align}\label{eqn-sphericity-condition-natural-isomorphism-twist}
\Phi^{\mathrm ladj}_E T_E [-1]
\xrightarrow{T_E[-1] \rightarrow \Phi_E \Phi^{\mathrm radj}_E 
\text{ in } \eqref{eqn-twist-triangle}}
\Phi^{\mathrm ladj}_E \Phi_E \Phi^{\mathrm radj}_E
\xrightarrow{\Phi^{\mathrm ladj}_E\Phi_E \rightarrow \id} 
\Phi^{\mathrm radj}_E, 
\end{align}
\begin{align}\label{eqn-sphericity-condition-natural-isomorphism-cotwist}
\Phi^{\mathrm radj}_E
\xrightarrow{\id \rightarrow \Phi_E \Phi^{\mathrm ladj}_E} 
\Phi^{\mathrm radj}_E \Phi_E \Phi^{\mathrm ladj}_E
\xrightarrow{\Phi^{\mathrm radj}_E \Phi_E \rightarrow F_E[1] \text{ in
} 
\eqref{eqn-cotwist-triangle}}
F_E [1] \Phi^{\mathrm ladj}_E. 
\end{align}

The following key notion was introduced in 
\cite{AnnoLogvinenko-SphericalDGFunctors}:
\begin{defn}\label{defn-spherical-functor}
We say that the Fourier--Mukai transform $\Phi_E$ is \em a spherical 
functor \rm if:
\begin{enumerate}
\item $T_E$ is an autoequivalence of $D(X_2)$, 
\item $F_E$ is an autoequivalence of $D(X_1)$, 
\item  
$ \Phi^{\mathrm ladj}_E T_E [-1]
\xrightarrow{ \eqref{eqn-sphericity-condition-natural-isomorphism-twist} }
\Phi^{\mathrm radj}_E$
is an isomorphism of functors (``the twist identifies the adjoints''),
\item
$\Phi^{\mathrm radj}_E
\xrightarrow{ \eqref{eqn-sphericity-condition-natural-isomorphism-cotwist} }
F_E [1] \Phi^{\mathrm ladj}_E
$
is an isomorphism of functors (``the co-twist identifies the adjoints'').
\end{enumerate}
\end{defn}

The following is the main result of \cite{AnnoLogvinenko-SphericalDGFunctors}: 
\begin{theorem}[\cite{AnnoLogvinenko-SphericalDGFunctors}, Theorem 5.1]
Any two of the conditions in Definition \ref{defn-spherical-functor}
imply all four. 
\end{theorem}
\begin{cor}\label{cor-spherical-functors-for-dummies} 
If $F_E$ is an autoequivalence of $D(X_1)$ and if
$\Phi^{\mathrm radj}_E
\xrightarrow{ \eqref{eqn-sphericity-condition-natural-isomorphism-cotwist} }
F_E [1] \Phi^{\mathrm ladj}_E
$ is an isomorphism, then $\Phi_E$ is a spherical functor. 
\end{cor}

\begin{lemma} 
\label{lemma-sphericity-condition-the-commutation-condition}
The composition $\eqref{eqn-sphericity-condition-natural-isomorphism-cotwist}$
is the unique morphism 
$\Phi^{\mathrm radj}_E \xrightarrow{\alpha} F_E[1] \Phi^\mathrm{ladj}_E$ which 
makes the following diagram commute:
\begin{align} \label{eqn-sphericity-condition-the-commutation-condition}
\xymatrix{
\Phi^{\mathrm radj}_E \Phi_E  \ar^{\eqref{eqn-cotwist-triangle}}[r] 
\ar_{\alpha}[d] 
& F_E[1] \\
F_E[1] \Phi^{\mathrm ladj}_E \Phi_E \ar_{\text{ adj. co-unit}}[ur] &
}
\end{align}
\end{lemma}

\begin{proof}

We first show that the composition
$\eqref{eqn-sphericity-condition-natural-isomorphism-cotwist}$ 
makes $\eqref{eqn-sphericity-condition-the-commutation-condition}$ 
commute. Indeed, composing each term with $\Phi_E$ and 
composing the whole isomorphism with the adjunction co-unit 
$\Phi^\mathrm{ladj}_E \Phi_E \rightarrow \id_{D(X_1)}$ we obtain
\begin{align}
\label{eqn-canonical-iso-plugged-into-commutation-condition-1}
\Phi^{\mathrm radj}_E \Phi_E \xrightarrow{\text{adj. unit}} 
\Phi^{\mathrm radj}_E \Phi_E \Phi^{\mathrm ladj}_E \Phi_E
\xrightarrow{\eqref{eqn-cotwist-triangle}}
F_E[1] \Phi^{\mathrm ladj}_E \Phi_E \xrightarrow{\text{adj. co-unit}}
F_E[1].
\end{align}
Since clearly the following square commutes
\begin{align}
\xymatrix{ 
\Phi^{\mathrm radj}_E \Phi_E \Phi^{\mathrm ladj}_E \Phi_E 
\ar^{\eqref{eqn-cotwist-triangle}}[r]
\ar_{\text{adj. co-unit}}[d] &
F_E[1] \Phi^{\mathrm ladj}_E \Phi_E 
\ar^{\text{adj. co-unit}}[d] \\
\Phi^{\mathrm radj}_E \Phi_E 
\ar_{\eqref{eqn-cotwist-triangle}}[r] &
F_E[1]
}
\end{align}
the composition 
$\eqref{eqn-canonical-iso-plugged-into-commutation-condition-1}$
equals to
\begin{align}
\label{eqn-canonical-iso-plugged-into-commutation-condition-2}
\Phi^{\mathrm radj}_E \Phi_E 
\xrightarrow{\text{adj. unit}} 
\Phi^{\mathrm radj}_E \Phi_E \Phi^{\mathrm ladj}_E \Phi_E
\xrightarrow{\text{adj. co-unit}}
\Phi^{\mathrm radj}_E \Phi_E
\xrightarrow{\eqref{eqn-cotwist-triangle}}
F_E[1] 
\end{align}
and is therefore simply $\Phi^{\mathrm radj}_E \Phi_E 
\xrightarrow{\eqref{eqn-cotwist-triangle}} F_E[1]$, as required.

Conversely, let $\alpha\colon 
\Phi^{\mathrm radj}_E \rightarrow F_E \Phi^\mathrm{ladj}_E[1]$ be a
morphism which makes 
$\eqref{eqn-sphericity-condition-the-commutation-condition}$ commute. 
We then have a commutative diagram
\begin{align}
\xymatrix{
\Phi^{\mathrm radj}_E 
\ar^{\text{adj. unit}}[rr] 
\ar^{\alpha}[d] 
& &
\Phi^{\mathrm radj}_E \Phi_E \Phi^{\mathrm ladj}_E 
\ar^{\eqref{eqn-cotwist-triangle}}[rr]
\ar^{\alpha}[d] 
& &
F_E [1] \Phi^{\mathrm ladj}_E
\ar^{=}[d] 
\\
F_E[1] \Phi^\mathrm{ladj}_E
\ar_{\text{adj. unit}}[rr] 
& &
F_E[1] \Phi^\mathrm{ladj}_E \Phi_E \Phi^{\mathrm ladj}_E 
\ar_{\text{adj. co-unit}}[rr]
& &
F_E [1] \Phi^{\mathrm ladj}_E  
} 
\end{align}
Since the bottom row is the identity morphism, we conclude
that $\alpha$ equals to the morphism given by the top row, 
i.e. to the composition $\eqref{eqn-sphericity-condition-natural-isomorphism-cotwist}$. 
\end{proof}

\subsection{Miscellaneous}
\label{section-miscellaneous}

In this section we give two technical lemmas we make use
of throughout the paper. 

Recall that the \em support \rm $\supp E$ of an object $E$ 
in the derived category $D(\mathcal{O}_X\text{-}\modd)$ of a scheme $X$ is 
the union of the supports of its cohomology sheaves $\mathcal{H}^i(E)$. 
The \em support \rm of a coherent sheaf $\mathcal{F}$ on a scheme $X$
is defined, as per \cite{Harts77}, to be the set of all $x \in X$
such that the stalk $\mathcal{F}_x$ is not zero. 
By \cite[Ex. 5.6(c)]{Harts77} the support of a coherent sheaf on 
a noetherian scheme is closed. These are the definitions employed in e.g.
\cite{AvramovIyengarLipman-ReflexivityAndRigidityForComplexesIISchemes}
whose results we make use of. 

\begin{lemma}
\label{lemma-pullback-nonzero-iff-the-point-lies-in-the-support}
Let $X$ be a noetherian scheme and let $E \in D(X)$.  
A point $x \in X$ lies in $\supp E$ if and only if 
$\iota^*_x E \neq 0$. 
\end{lemma}
\begin{proof}
 
First, we claim that given a coherent sheaf $\mathcal{F}$ on $X$ a
point $x \in X$ lies in $\supp \mathcal{F}$ if and only if the ordinary,
non-derived pullback $\lder^0 \iota^*_x \mathcal{F} \neq 0$. This is because 
the stalk $\mathcal{F}_x$ is a finite $\mathcal{O}_{X,x}$-module and by 
\cite{Mats86} any finite module for a Noetherian local ring 
has a minimal free resolution
$$ \dots \rightarrow L_2 \rightarrow L_1 \rightarrow L_0 $$
whose differentials die under $\iota^*_x$, i.e.  
$\dim \lder^i \iota^*_x \mathcal{F} = \rank L_i$. By definition 
$x \in \supp \mathcal{F}$ if and only $\mathcal{F}_x \neq 0$. 
On the other hand, $\mathcal{F}_x \neq 0$ if and only if $L_0 \neq 0$, 
which by above is equivalent to $\lder^0 \iota^*_x \mathcal{F} \neq 0$. 

Now let $E$ be an object of $D^b_{coh}(X)$. Consider 
the standard spectral sequence 
$$ \lder^{p} \iota^{*}_{x} \mathcal{H}^{q} E 
\Rightarrow \lder^{p + q}  \iota^{*}_{x} E. $$
Suppose $\lder^{0} \iota^{*}_{x} \mathcal{H}^{q} E \neq 0$ for some 
$q \in \mathbb{Z}$. Take minimal $q$ for which this holds --- 
we can do that since $\mathcal{H}^j(E) \neq 0$ for only finite number of 
$j \in \mathbb{Z}$. 
Then $\lder^{0} \iota^{*}_{x} \mathcal{H}^{q} E$ is the lower-left
corner of the non-zero terms of the spectral sequence and hence survives
yielding $\lder^{q} \iota^{*}_{x} E \neq 0$.
On the other hand, if $\lder^{0} \iota^{*}_{x} \mathcal{H}^{q} E = 0$ 
for all $q \in \mathbb{Z}$, all the higher pullbacks
$\lder^{p} \iota^{*}_{x} \mathcal{H}^{q} E = 0$ also vanish 
by the minimal free resolution argument above. Thus all terms of 
the spectral sequence are zero and thus $\iota^{*}_{x} E = 0$. 

We have thus shown that $\iota^{*}_{x} E \neq 0$ if and only 
if $\lder^{0} \iota^{*}_{x} \mathcal{H}^{q} E \neq 0$ for
some $q \in \mathbb{Z}$. By the first claim, this is equivalent
to $x \in \supp \mathcal{H}^{q} E$ for some $q \in \mathbb{Z}$
and that is the definition of $x$ lying in $\supp E$. 
\end{proof}

\begin{lemma}
\label{lemma-morphism-of-FM-kernels-induces-iso-of-FMs-iff-itis-an-iso-itself}
Let $X_1$ and $X_2$ be two noetherian schemes.  
Let $E_1$ and $E_2$ be two objects of $D(X_1 \times X_2)$
and let $\alpha$ be a morphism from $E_1$ to $E_2$. Then 
$\alpha$ is an isomorphism if and only if the induced morphism 
of functors $\Phi_{E_1} \rightarrow \Phi_{E_2}$ is an isomorphism.
\end{lemma}
\begin{proof}
The `only if' statement is obvious. For the `if' statement
we use the fact that for any closed point $p \in X_1$ and 
any $A$ in $D(X_1)$ we have a natural isomorphism 
$\Phi_{A}(\mathcal{O}_p) \xrightarrow{\sim}
\iota^*_p (A)$ which is functorial in $A$. 
So if $\Phi_{E_1} \rightarrow \Phi_{E_2}$
is an isomorphism then the pullback of $\alpha$ 
to any closed point of $X_1$ is an isomorphism.
This implies that the pullback of the cone $\alpha$ to 
any closed point of $X_1$ is $0$. 
By Lemma \ref{lemma-pullback-nonzero-iff-the-point-lies-in-the-support} 
the cone of $\alpha$ is itself $0$, and thus $\alpha$ is an isomorphism. 
\end{proof}

\section{Orthogonally spherical objects}
\label{section-orthogonally-spherical-objects}

Let $Z$ and $X$ be two separable schemes of finite type over $k$.
Given a closed point $p$ in $Z$ we denote by $\iota_p$ the
closed immersion $\spec k \hookrightarrow Z$ and by 
$\iota_{Xp}$ the corresponding immersion $X \hookrightarrow 
Z \times X$:
\begin{align} \label{eqn-inclusion-of-fibre-over-p}
\xymatrix{
X\; \ar[d]_{\pi_{k}} \ar@{^{(}->}[r]^>>>>>>{\iota_{Xp}} &
Z \times X \ar[d]^{\pi_{Z}} \ar[dr]^{\pi_{X}}\\
\spec k \; \ar@{^{(}->}[r]^>>>>>>>{\iota_{p}} & Z & X
}
\end{align}
Given a perfect object $E$ in $D(Z \times X)$ we define
\em the fibre $E_p$ of $E$ at $p$ \rm to be the object $\iota^*_{Xp} E$ in
$D(X)$. In this way we can think of any perfect object in 
$D(Z \times X)$ as a family of objects of $D(X)$ parametrised
by $Z$. 
We assume throughout this section that either $Z$ and $X$
are both proper or that the support of the object $E$ in $Z \times X$ 
is proper over both $Z$ and $X$. This ensures that all of our Fourier--Mukai 
transforms take complexes
with bounded coherent cohomologies to complexes with bounded coherent
cohomologies. It also makes applicable 
the results in Section \ref{section-adjunction-units-and-fm-transforms}
on the adjunctions units/co-units for Fourier--Mukai transforms.

\subsection{Orthogonal objects}

Our first goal is to come up with a categorification of the 
notion of a subscheme $W$ of $X$ fibred over $Z$. 
Our motivation is the following geometric example: 

\begin{exmpl}
Let $W$ be a flat fibration in $X$ over $Z$ with proper
fibres. By this we mean a scheme $W$ equipped with a morphism 
$\xi\colon W \hookrightarrow X$ which is a closed immersion 
and a morphism $\pi\colon W \rightarrow Z$ which is flat and proper. 
Denote by $\iota_W$ the map 
$W \hookrightarrow Z \times X$ given by the product 
of $\pi$ and $\xi$. We set $E$ to be the structure sheaf 
of the graph of $W$ in $Z \times X$, that is - 
the object $\iota_{W *} \mathcal{O}_W$ in $D(Z \times X)$. 
\end{exmpl}

An arbitrary subscheme $W'$ of $Z \times X$ is a graph 
of some subscheme $W$ of $X$ fibred over $Z$ if and only
if the fibres of $W'$ over closed points of $Z$ are disjoint 
as subschemes of $X$. In derived categories the notion of disjointness 
corresponds to the notion of orthogonality, that is, 
to the vanishing of all the $\ext$'s between them. 
This suggests the following as a categorification of 
the notion of a subscheme of $X$
fibred over $Z$:
\begin{defn}
\label{defn-orthogonal-objects}
Let $E$ be a perfect object of $D(Z \times X)$. We say that $E$ is \em
orthogonal over $Z$ \rm if for any two distinct
points $p$ and $q$ in $Z$ the fibres $E_p$ and $E_q$ are orthogonal
in $D(X)$. Or in other words
\begin{align}
\homm^i_{D(X)}(E_p, E_q) = 0 \quad \quad \text{for all } i \in
\mathbb{Z}.
\end{align}
\end{defn}

Since $E$ is a perfect object we have $(E^\vee)_p = (E_p)^\vee$.
So if $E$ is orthogonal over $Z$, then its dual $E^\vee$ is also 
orthogonal over $Z$.

Any object whose support in $Z \times X$ is the graph of a subscheme
of $X$ fibred over $Z$ is immediately orthogonal over $Z$ --- as 
all the $\ext$'s between two objects with disjoint supports must vanish. 
Another class of examples comes from Fourier--Mukai equivalences:
\begin{exmpl}
\label{exmpl-the-case-of-Fourier--Mukai-equivalences}
The kernel $F$ of any fully faithful Fourier--Mukai transform
$\Phi_F\colon D(Z) \xrightarrow{\sim} D(X)$ is orthogonal over $Z$, 
since for any $p \in Z$ the fibre $F_p$ is the 
image under $\Phi_F$ of the skyscraper sheaf $\mathcal{O}_p$. 
Moreover, we have also $\Phi^{\mathrm radj}_F \Phi_F(\mathcal{O}_p) \simeq
\pi_{Z *} \rder\shhomm_{Z \times X}(F, \pi^!_X(F_p))$. 
The adjunction unit 
$\mathcal{O}_p \rightarrow \Phi^{\mathrm radj}_F \Phi_F(\mathcal{O}_p)$ 
is an isomorphism as $\Phi_F$ is fully faithful.  Applying 
$\pi_{k *}$, where $\pi_k$ is the structure morphism 
$Z \rightarrow \spec k$, to this adjunction unit we obtain
$\rder\homm_{X}(\pi_{X *} F, F_p) = k$. It is possible using the 
same techniques as in the proof of Proposition 
\ref{prps-tfae-left-cotwist-is-an-equivalence} below to show 
that the converse is also true, i.e. $\Phi_F$ is fully faithful if 
and only if $F$ is orthogonal over $Z$ and 
$\rder\homm_{X}(\pi_{X *} F, F_p) = k$ for all $p \in Z$. 
Suppose now that $\Phi_F$ is further an
equivalence, then all its adjunction units and co-units are isomorphisms. 
By Lemma 
$\ref{lemma-morphism-of-FM-kernels-induces-iso-of-FMs-iff-itis-an-iso-itself}$
the morphisms of Fourier--Mukai kernels which induce them are 
also isomorphisms. In particular, the isomorphism of functors
\begin{align}
\Phi^{\mathrm radj}_F \xrightarrow{\text{adj. unit}} 
\Phi^{\mathrm radj}_F \Phi_F \Phi^{\mathrm ladj}_F
\xrightarrow{\text{inverse of adj. co-unit}}
\Phi^{\mathrm ladj}_F 
\end{align}
must come from an isomorphism 
$F^\vee \otimes \pi^!_X(\mathcal{O}_X) \rightarrow 
F^\vee \otimes \pi^!_{Z}(\mathcal{O}_{Z})$ of their 
Fourier--Mukai kernels. Conversely, any isomorphism 
$F^\vee \otimes \pi^!_X(\mathcal{O}_X) \rightarrow 
F^\vee \otimes \pi^!_{Z}(\mathcal{O}_{Z})$ 
induces an isomorphism $\Phi^{\mathrm radj}_F \xrightarrow{\sim} \Phi^{\mathrm ladj}_F$. 
On the other hand, when $X$ is connected
by \cite[Theorem 3.3]{Bridg97}
$\Phi_F$ being fully faithful and 
$\Phi^{\mathrm radj}_F$ being isomorphic to $\Phi^{\mathrm ladj}_F$ imply together 
that $\Phi_F$ is an equivalence.
We conclude that when $X$ is connected the kernels of Fourier--Mukai 
equivalences are precisely the objects which are orthogonal 
over $Z$ and for which 
\begin{align}
\rder\homm_{X}(\pi_{X *} F, F_p) = k \text{ for all } p \in Z \\
\label{eqn-twist-by-dualizing-sheaves-condition-for-equivalences}
F^\vee \otimes \pi^!_X(\mathcal{O}_X) \simeq 
F^\vee \otimes \pi^!_{Z}(\mathcal{O}_{Z}).
\end{align}
\end{exmpl}

Our goal is to show that the orthogonal objects
which are one step up from that, in the sense that  
$\rder\homm_{X}(\pi_{X *} F, F_p) = k \oplus k[d]$ for some $d \in \mathbb{Z}$
and a similar condition to 
\eqref{eqn-twist-by-dualizing-sheaves-condition-for-equivalences} 
holds, are kernels of spherical Fourier--Mukai transforms.  

\subsection{Spherical objects}

\begin{defn}
\label{defn-spherical-objects}
Let $E$ be a perfect object of $D(Z \times X)$. We say that $E$ is
\em spherical over $Z$ \rm if the Fourier--Mukai transform
$\Phi_E\colon D(Z) \rightarrow D(X)$ is a spherical functor, 
cf. Defn.  \ref{defn-spherical-functor} and Cor. 
\ref{cor-spherical-functors-for-dummies}. In other words, if:
\begin{enumerate}
\item The co-twist $F_E$ is an autoequivalence of $D(Z)$,  
\item The natural map $\Phi^{\mathrm radj}_E 
\xrightarrow{ \eqref{eqn-sphericity-condition-natural-isomorphism-cotwist} }
F_E \Phi^{\mathrm ladj}_E[1]$
is an isomorphism of functors. 
\end{enumerate}
If $E$ is also orthogonal over $Z$ we say further that 
$E$ is \em orthogonally spherical over \rm $Z$. 
\end{defn}

\begin{exmpl}
\label{exmpl-the-case-where-base-is-a-single-point} 

The spherical objects introduced by Seidel and Thomas in 
\cite{SeidelThomas-BraidGroupActionsOnDerivedCategoriesOfCoherentSheaves}
can be thought of as the objects spherical over $\spec k$. 
Indeed let $Z = \spec k$ and let $X$ be a smooth variety over $k$. 
Then $\pi_X$ is an isomorphism which identifies $\spec k \times X$
with $X$. Under this identification $\pi^!_X(\mathcal{O}_X)$ becomes simply
$\mathcal{O}_X$ and $\pi^!_{k}(k)$ becomes 
the dualizing complex $D_{X/k}$ which is isomorphic to $\omega_X[\dim X]$
since $X$ is smooth. Therefore the Fourier--Mukai kernel of the right adjoint
$\Phi^{\mathrm{radj}}_E$ is $E^\vee$ and the Fourier--Mukai kernel of 
the left adjoint $\Phi^{\mathrm{ladj}}_E$ is $E^\vee \otimes 
\omega_X[\dim X]$. The triple product $\spec k \times X \times \spec k$ is
identified with $X$ by the projection $\pi_2$ and 
under this identification the projection $\pi_{13}$ 
becomes the map $\pi_k \colon X \rightarrow \spec k$. Therefore 
the Fourier--Mukai kernel of the composition 
$\Phi^{\mathrm{radj}}_E \Phi_E$ is 
$$\pi_{k *} (E^\vee \otimes E) \simeq \pi_{k *} \rder\shhomm_{X}(E,E) \simeq 
\rder\homm_{X}(E,E)$$ 
and by the results of Section
\ref{section-adjunction-units-and-fm-transforms} the adjunction 
unit $\id_{D(\vectspaces)} \rightarrow \Phi^{\mathrm{radj}}_E \Phi_E$
comes from the natural morphism $k \rightarrow \rder\homm_{X}(E,E)$ of
Fourier--Mukai kernels which sends $1$ to the identity
automorphism of $E$. Denote this morphism by $\gamma$. 

The first condition 
for $\Phi_E$ to be a spherical functor is for the co-twist $F_E$
to be an autoequivalence of $D(\vectspaces)$. The only autoequivalences
of $D(\vectspaces)$ are the shifts $(-)[d]$ by some $d \in \mathbb{Z}$
and their Fourier--Mukai kernels are $k[d]$. 
The Fourier--Mukai kernel of $F_E$ is the shift 
by $1$ to the left of the cone of 
$k \xrightarrow{\gamma} \rder\homm_{X}(E,E)$. If $E$ is non-zero
the morphism $\gamma$ is non-zero and then $F_E$ is an autoequivalence 
if and only if $\rder\homm_{X}(E,E)$ is $k \oplus k[d]$ 
for some $d \in \mathbb{Z}$. If this does hold then $F_E=(-)[d-1]$. 
If $E$ is $0$, then the cone of $\gamma$ is $k$ 
and therefore $F_E$ is the identity functor $\id_{D(\vectspaces)}$.  
Note that $E$ is trivially isomorphic to its
single fibre over the single closed point of $\spec k$. Hence we've
shown that $F_E$ is an autoequivalence
if and only if for every closed point $p \in Z$ such that the fibre 
$E_p$ is non-zero we have 
$\rder\homm_X(\pi_{X *} E, E_p) = k \oplus k[d]$ for some $d \in \mathbb{Z}$. 

By Lemma \ref{lemma-sphericity-condition-the-commutation-condition}
the second condition for $\Phi_E$ to be spherical 
is an isomorphism 
$\alpha\colon E^\vee \xrightarrow{\sim} E^\vee \otimes \omega_X[\dim X + d]$
which makes the diagram
$\eqref{eqn-sphericity-condition-the-commutation-condition}$
commute. If $E$ is $0$ then this condition is trivially
satisfied, so assume otherwise. Since $E^\vee$ and 
$E^\vee \otimes \omega_X$ are bounded 
complexes with non-zero cohomologies in exactly the same degrees, the
isomorphism $\alpha$ 
is only possible when $d = - \dim X$. On the other hand, the diagram 
$\eqref{eqn-sphericity-condition-the-commutation-condition}$
on the level of the corresponding Fourier--Mukai kernels is just
\begin{align}
\xymatrix{
k \oplus k[d] \ar_<<<<{\alpha'}[d] \ar[r]^<<<<<{0 \oplus \id} & k[d]  \\
k \oplus k[d] \ar[ur]_{0 \oplus \id} & 
}
\end{align}
where $\alpha'$ is the isomorphism induced by $\alpha$. 
The diagram commutes if $\alpha'$ restricts to the identity morphism
on the component $k[d]$ and we can achieve that by 
multiplying any given $\alpha$ by an appropriate scalar in $k$. 

We conclude that $E$ is spherical over $\spec k$
if and only if either $E$ is $0$ or if 
$\rder\homm_{X}(E,E) = k \oplus k[ -\dim X]$ 
and $E \simeq E \otimes \omega_X$, which is precisely the definition
given in
\cite{SeidelThomas-BraidGroupActionsOnDerivedCategoriesOfCoherentSheaves}.
And since the base $\spec k$ is a single point, 
any object spherical over $\spec k$ is orthogonally spherical. 
\end{exmpl}

\subsection{A cohomological criterion for sphericity}

We now introduce the object in the derived category $D(Z)$ of the
base $Z$ which is relative case version of 
the cone of the natural morphism $k \rightarrow \rder\homm_X(E, E)$
of the Example $\ref{exmpl-the-case-where-base-is-a-single-point}$
where the base $Z$ is just the single point $\spec k$:

\begin{defn}
\label{defn-object-L_E-for-abstract-objects-of-D-ZxX}
For any perfect object $E$ of $D(Z \times X)$ denote by
$\mathcal{L}_E$ the object of $D(Z)$ which is the cone of 
the following composition of morphisms:
\begin{align} \label{eqn-shhom-pipiE-E-map}
\mathcal{O}_{Z} \rightarrow \pi_{Z *} \mathcal{O}_{Z \times X} 
\rightarrow \pi_{Z *}\rder\shhomm_{Z \times X}(E, E)
\rightarrow \pi_{Z *}\rder\shhomm_{Z \times X}(\pi^*_X \pi_{X *} E, E).
\end{align}
Here the first morphism is induced by the adjunction unit $\id_{D(Z)} 
\rightarrow \pi_{Z *} \pi^{*}_{Z}$, 
the second by the adjunction unit $\id_{D(Z \times X)}
\rightarrow \rder\shhomm(E, E\otimes -)$ 
and the third by the adjunction co-unit
$\pi^*_X \pi_{X *} \rightarrow \id_{D(Z \times X)}$.  
\end{defn}

Let $p$ be any closed point of the base $Z$. Apply the pullback
functor $\iota^*_p$ to the composition 
$\eqref{eqn-shhom-pipiE-E-map}$
to obtain a morphism 
$k \rightarrow 
\iota^*_p \pi_{Z *}\rder\shhomm_{Z \times X}(\pi^*_X \pi_{X *} E, E)$. 
We have a sequence of natural isomorphisms:
\begin{align} 
\label{eqn-pullback-of-shhom-pipiE-E-base-change}
\iota_{p}^{*} \pi_{Z *} \rder\shhomm_{Z \times X}(\pi^*_X \pi_{X *} E, E) 
\xrightarrow{\text{ base change iso. around }
\eqref{eqn-inclusion-of-fibre-over-p}} 
\pi_{k*} \iota_{Xp}^{*} \rder\shhomm_{Z \times X}(\pi^*_X \pi_{X *} E, E)  \\
\label{eqn-pullback-of-shhom-pipiE-E-pullback-commutes-with-rhom}
\pi_{k*} \iota_{Xp}^{*} \rder\shhomm_{Z \times X}(\pi^*_X \pi_{X *} E, E)  
\xrightarrow{\text{\cite{IllusieGeneralitesSurLesConditionsDeFinitudeDansLesCategoriesDerivees},
Prps. 7.1.2}}
\pi_{k*}\rder\shhomm_{X}(\iota_{Xp}^{*}\pi^*_X \pi_{X *} E, \iota_{Xp}^{*}E)  \\
\label{eqn-pullback-of-shhom-pipiE-E-rewrite}
\pi_{k*}\rder\shhomm_X(\iota_{Xp}^{*}\pi^*_X \pi_{X *} E, \iota_{Xp}^{*}E) 
\xrightarrow{\pi_{X} \circ \iota_{Xp} = \id_{X}}
\pi_{k*} \rder\shhomm_X(\pi_{X *} E, E_{p}) \simeq 
\rder\homm_{X}(\pi_{X *} E, E_p).
\end{align}
One can check that these natural isomorphisms
identify $\iota^*_p \; \eqref{eqn-shhom-pipiE-E-map}$ with the morphism 
\begin{align}
\label{eqn-plback-shhom-pipiE-E-map}
k \rightarrow \rder\homm_{D(X)}(\pi_{X *} E, E_p)
\end{align}
which sends $1$ to the natural composition 
\begin{align}
\label{eqn-nat-morphism-pi_X*-E-to-E_p} 
\iota^*_{X_p} 
\left( \pi^*_{X} \pi_{X *} E \xrightarrow{\text{adj. co-unit}} 
E \right)
\end{align}
where we identify the LHS with $\pi_{X *} E$ via the scheme map identity 
$\pi_X \circ \iota_{X_p} = \id_X$. Thus pointwise restriction of 
\eqref{eqn-shhom-pipiE-E-map} gives a natural morphism
$k \rightarrow \rder\homm_{D(X)}(\pi_{X *} E, E_p)$ 
for each closed point $p \in Z$. It turns out
that for an orthogonal $E$ the criterion for the co-twist $F_E$
of $E$ to be an autoequivalence of $D(Z)$ is for 
the cone of each of these morphisms to be $k[d]$ for some $d \in
\mathbb{Z}$:

\begin{prps}\label{prps-tfae-left-cotwist-is-an-equivalence}
Let $E$ be a perfect object of $D(Z \times X)$ orthogonal over $Z$. 
The following are equivalent:
\begin{enumerate}
\item \label{item-sphericity-condition-on-exts-for-fibres}
For every closed point $p \in Z$ such that the fibre $E_p$
is non-zero  
$$ \rder\homm_{D(X)}(\pi_{X *} E, E_p) = k \oplus k[d_p]
\quad\text{ for some } d_p \in \mathbb{Z}$$
and the natural morphism $\pi_{X *} E
\xrightarrow{ \eqref{eqn-nat-morphism-pi_X*-E-to-E_p} } E_p$
is not zero. 
\item \label{item-invertibility-of-L_E}
The object $\mathcal{L}_E$ is an invertible object of $D(Z)$. That is - 
on every connected component of $Z$ it is isomorphic to a shift of a line
bundle, cf. 
\cite[\S1.5]{AvramovIyengarLipman-ReflexivityAndRigidityForComplexesIISchemes}. 
\item \label{item-the-left-co-twist-is-an-equivalence}
The co-twist $F_E$ of the transform $\Phi_E\colon 
D(Z) \rightarrow D(X)$ is an autoequivalence of $D(Z)$. 
\end{enumerate}

When the conditions above are satisfied:
\begin{itemize}
\item Locally around any closed point $p \in Z$ we have
$\mathcal{L}_E \simeq \mathcal{O}_Z[d_p]$ where
\begin{align*} d_p = 
\begin{cases}
\text{ the same integer as in } 
(\ref{item-sphericity-condition-on-exts-for-fibres}) &\text{if } E_p \neq 0 \\
1 &\text{ if } E_p = 0
\end{cases}
\end{align*}
\item $F_E$ is isomorphic the functor $\mathcal{L}_E \otimes (\text{-}) [-1]$
\end{itemize}
\end{prps}

We see therefore that for the orthogonally spherical objects 
the geometric meaning of the object $\mathcal{L}_E$ defined 
above is that its restriction to each connected component 
of $Z$ is a (shifted) line bundle which 
induces the co-twist autoequivalence $F_E$ of $D(Z)$. 

To prove Proposition \ref{prps-tfae-left-cotwist-is-an-equivalence}
we need two technical lemmas. Recall that by Proposition
\ref{prps-right-adjunction-unit-morphism} the adjunction unit
$\id_{D(Z)} \rightarrow \Phi^{\mathrm radj}_E \Phi_E$ is
isomorphic to the morphism of Fourier--Mukai transforms induced
by the morphism  
\begin{align}
\Delta_* \mathcal{O}_{Z} 
\xrightarrow{\eqref{eqn-pi1-unit-in-fm-right-adjunction-unit-morphism}-
\eqref{eqn-delta-unit-in-fm-right-adjunction-unit-morphism}}
Q = \pi_{13 *} \left( \pi^*_{12} E \otimes \pi^*_{23} E^\vee \otimes
\pi^*_{23} \pi^!_X \mathcal{O}_X \right)
\end{align}
in $D(Z \times Z)$. Here $\pi_{ij}$ are 
the natural projection morphisms in the following commutative diagram:
\begin{align} \label{eqn-big-projection-tree-for-x1-x}
\xymatrix{
& & Z \times X \times Z \ar[ld]_{\pi_{12}} \ar[d]^{\pi_{13}} \ar[rd]^{\pi_{23}} & & \\
& Z \times X \ar[ld]_{\pi_{Z}} \ar[rd]^>>>>>>>{\pi_X} & Z \times Z
\ar[lld]_>>>>>>>>>>>>>>{\tilde{\pi}_1}
\ar[rrd]^>>>>>>>>>>>>>>{\tilde{\pi}_2} & X \times Z 
\ar[ld]_>>>>>>>{\pi_X} \ar[rd]^{\pi_{Z}} & \\
Z & & X & & Z 
}
\end{align}

\begin{lemma}
\label{lemma-pullback-of-L-E-morphism-is-iso-to-pushdown-of-the-adjunction}
Let $E$ be a perfect object of $D(Z \times X)$ and let $p \in Z$ 
be a closed point. Then the following two morphisms in $D(Z)$ 
are isomorphic:
\begin{align} \label{eqn-pushdown-of-adjunction-morphism}
\tilde{\pi}_{1 *} \left(\Delta_* \mathcal{O}_{Z}
\xrightarrow{\eqref{eqn-pi1-unit-in-fm-right-adjunction-unit-morphism}-
\eqref{eqn-delta-unit-in-fm-right-adjunction-unit-morphism}}
Q\right)
\end{align}
and 
\begin{align}
\mathcal{O}_{Z}
\xrightarrow{\eqref{eqn-shhom-pipiE-E-map}}
\pi_{Z *}\rder\shhomm(\pi^*_X \pi_{X *} E, E).
\end{align}
Consequently, for every closed point $p \in Z$
the natural morphism $k \xrightarrow{\eqref{eqn-plback-shhom-pipiE-E-map}} 
\rder\homm_X(\pi_{X*} E, E_p)$
is isomorphic to $\pi_{k *} \left(\mathcal{O}_p \xrightarrow{\text{adj.  unit}}
\Phi^{\mathrm radj}_E \Phi_E (\mathcal{O}_p) \right)$
and therefore $\pi_{k *} F_E(\mathcal{O}_p) \simeq \iota^*_p
\mathcal{L}_E[-1]$. 
\end{lemma}
\begin{proof}
For the first claim we need to show that $\mathcal{O}_{Z}
\xrightarrow{\eqref{eqn-shhom-pipiE-E-map}}
\pi_{Z *}\rder\shhomm(\pi^*_X \pi_{X *} E, E)$ 
is isomorphic to:
\begin{align}
\label{eqn-tildepi1-pi-X1-unit-in-fm-right-adjunction-unit-morphism}
\tilde{\pi}_{1 *} \Delta_* (\mathcal{O}_{Z})
\xrightarrow{\id \rightarrow \pi_{Z *} \pi^*_{Z}}
\tilde{\pi}_{1 *} \Delta_* \pi_{Z *} \pi^*_{Z} (\mathcal{O}_{Z})
\\ 
\label{eqn-tildepi1-E-E-dual-unit-in-fm-right-adjunction-unit-morphism}
\tilde{\pi}_{1 *}\Delta_* \pi_{Z *} \pi^*_{Z} (\mathcal{O}_{Z})
\xrightarrow{\id \rightarrow E \otimes E^\vee \otimes }
\tilde{\pi}_{1 *}\Delta_* \pi_{Z *} \left(E \otimes E^\vee \otimes 
\pi^*_{Z} (\mathcal{O}_{Z})\right)
\\
\label{eqn-tildepi1-rearranging-isos-in-fm-right-adjunction-unit-morphism}
\tilde{\pi}_{1 *} \Delta_* \pi_{Z *} \left(E \otimes E^\vee \otimes 
\pi^*_{Z} (\mathcal{O}_{Z})\right)
\quad \simeq \quad
\tilde{\pi}_{1 *} \pi_{13 *} \Delta_* \Delta^! 
\left(\pi_{12}^* E \otimes \pi_{23}^* E^\vee \otimes \pi_{23}^*
\pi^!_X (\mathcal{O}_{X})\right)
\\ 
\label{eqn-tildepi1-delta-unit-in-fm-right-adjunction-unit-morphism}
\tilde{\pi}_{1 *} \pi_{13 *} \Delta_* \Delta^! 
\left(\pi_{12}^* E \otimes \pi_{23}^* E^\vee \otimes \pi_{23}^*
\pi^!_X(\mathcal{O}_{X})\right)
\xrightarrow{\Delta_* \Delta^! \rightarrow \id}
\tilde{\pi}_{1 *} \pi_{13 *} 
\left(\pi_{12}^* E \otimes \pi_{23}^* E^\vee \otimes \pi_{23}^*
\pi^!_X(\mathcal{O}_{X})\right) 
\end{align}

By the scheme map identity $\tilde{\pi}_{1} \circ \Delta = \id_{Z}$ we have
$\tilde{\pi}_{1 *} \Delta_* \simeq \id_{D(Z)}$ and this identifies   
$\eqref{eqn-tildepi1-pi-X1-unit-in-fm-right-adjunction-unit-morphism}$ and 
$\eqref{eqn-tildepi1-E-E-dual-unit-in-fm-right-adjunction-unit-morphism}$
with the first and the second morphisms in
the composition $\eqref{eqn-shhom-pipiE-E-map}$. 
It remains to show that
\begin{align}
\label{eqn-third-morphism-in-shhom-pipiE-E-map}
\pi_{Z *}\rder\shhomm_X(E, E)
\xrightarrow{\pi^*_X \pi_{X *} \rightarrow \id} \pi_{Z *}\rder\shhomm(\pi^*_X \pi_{X *} E, E)
\end{align}
is isomorphic to 
\eqref{eqn-tildepi1-delta-unit-in-fm-right-adjunction-unit-morphism}.

By the scheme map identity $\tilde{\pi}_1 \circ \pi_{13} =
\pi_{Z} \circ \pi_{12}$ from \eqref{eqn-big-projection-tree-for-x1-x} 
we have
$\tilde{\pi}_{1*} \pi_{13 *} \simeq \pi_{Z *} \pi_{12 *}$.
By the independent fibre square
\begin{align}
\xymatrix{
Z \times X \times Z \ar^{\pi_{23}}[r] \ar_{\pi_{12}}[d] &
X \times Z \ar^{\pi_{X}}[d] \\
Z \times X \ar_{\pi_{X}}[r] &
X
}
\end{align}
we also have $\pi_{23}^* \pi^!_X \simeq \pi^{!}_{12} \pi^{*}_{X}$, 
cf.  \cite{Lipman-NotesOnDerivedFunctorsAndGrothendieckDuality}, \S3.10.
We can therefore rewrite
\eqref{eqn-tildepi1-delta-unit-in-fm-right-adjunction-unit-morphism} as
\begin{align}
\label{eqn-delta-unit-in-fm-right-adjunction-unit-morphism-standalone-rewritten}
\pi_{Z *} \pi_{12 *} \Delta_* \Delta^! 
\left(\pi_{12}^* E \otimes \pi_{23}^* E^\vee \otimes \pi_{12}^!
\mathcal{O}_{Z \times X}\right)
\xrightarrow{\Delta_* \Delta^! \rightarrow \id}
\pi_{Z *} \pi_{12 *} 
\left(\pi_{12}^* E \otimes \pi_{23}^* E^\vee \otimes \pi_{12}^!
\mathcal{O}_{Z \times X}\right). 
\end{align}
Now observe that $\pi_{12}^* E$ is perfect, while 
$\pi_{23}^* E^\vee \otimes \pi_{12}^! \mathcal{O}_{Z \times X}$ 
is a tensor product of a perfect object and a $\pi_{12}$-perfect 
object and therefore itself $\pi_{12}$-perfect. Hence, 
even though $\Delta$ is not perfect, by Lemma      
\ref{lemma-natural-map-f^!=f^*-f^!-is-iso-for-perf-and-S-perf}
the natural map
$\Delta^!(\pi_{12}^* E \otimes \pi_{23}^* E^\vee \otimes \pi_{12}^!
\mathcal{O}_{Z \times X}) \rightarrow \Delta^*(\pi_{12}^* E)
\otimes \Delta^!(\pi_{23}^* E^\vee \otimes \pi_{12}^!
\mathcal{O}_{Z \times X})$ is still an isomorphism. 
It therefore follows 
from \cite[Lemma 2.2]{AnnoLogvinenko-OnTakingTwistsOfFourierMukaiFunctors}
that  
\eqref{eqn-delta-unit-in-fm-right-adjunction-unit-morphism-standalone-rewritten}
is isomorphic to 
\begin{align}
\pi_{Z *} \left(E \otimes \pi_{12 *} \Delta_* \Delta^! 
\left(\pi_{23}^* E^\vee \otimes \pi_{12}^!
\mathcal{O}_{Z \times X}\right)\right)
\xrightarrow{\Delta_* \Delta^! \rightarrow \id}
\pi_{Z *} \left( E \otimes \pi_{12 *} 
\left(\pi_{23}^* E^\vee \otimes \pi_{12}^!
\mathcal{O}_{Z \times X}\right)\right). 
\end{align}
It remains to show that 
\begin{align}
\label{eqn-part-third-morphism-in-shhom-pipiE-E-map}
\rder\shhomm_X(E, \mathcal{O}_{Z \times X})
\xrightarrow{\pi^*_X \pi_{X *} \rightarrow \id} 
\rder\shhomm(\pi^*_X \pi_{X *} E, \mathcal{O}_{Z \times X})
\end{align}
is isomorphic to 
\begin{align}
\label{eqn-part-delta-unit-in-fm-right-adjunction-unit}
\pi_{12 *} \Delta_* \Delta^! 
\rder\shhomm\left(\pi_{23}^* E, \pi_{12}^! \mathcal{O}_{Z \times X}\right)
\xrightarrow{\Delta_* \Delta^! \rightarrow \id}
\pi_{12 *}  
\rder\shhomm\left(\pi_{23}^* E, \pi_{12}^! \mathcal{O}_{Z \times X}\right).
\end{align}
Rewriting $\eqref{eqn-part-third-morphism-in-shhom-pipiE-E-map}$
and $\eqref{eqn-part-delta-unit-in-fm-right-adjunction-unit}$
in terms of the relative duality theory $D_{\bullet/Z \times X}$ 
(see Section \ref{section-on-duality-theories}) we obtain  
\begin{align*} 
D_{\bullet/Z \times X} \; \left( \pi_X^* \pi_{X *} E 
\xrightarrow{\pi_X^* \pi_{X *} \rightarrow \id } E \right)^\text{opp}
\quad\quad \text{ and } \quad\quad
D_{\bullet/Z \times X}
\left( 
\pi_{12 *} \pi_{23}^* E
\xrightarrow{\id \rightarrow \Delta_* \Delta^* }
\pi_{12 *} \Delta_* \Delta^* \pi_{23}^* E 
\right)^\text{opp}
\end{align*}
respectively and these are isomorphic by 
\cite{AnnoLogvinenko-OnTakingTwistsOfFourierMukaiFunctors}, Lemma 3.2. 
This settles the first claim of this lemma. 

For the second claim, we have an independent fibre square 
\begin{align} \label{eqn-iota-p-X1-commutative-square}
\xymatrix{
Z \; \ar[d]_{\pi_{k}} \ar@{^{(}->}[r]^>>>>>>{\iota_{p, Z}} &
Z \times Z \ar[d]^{\tilde{\pi}_{1}} \\
\spec k \; \ar@{^{(}->}[r]^>>>>>>>{\iota_{p}} & Z 
}. 
\end{align}
and for any $A \in D(Z \times Z)$ we have a standard 
isomorphism 
\begin{align}
\Phi_A (\mathcal{O}_p) \xrightarrow{\sim} \iota^*_{p,Z} A
\end{align}
which is functorial in $A$. The adjunction unit morphism
$\mathcal{O}_p \rightarrow \Phi^{\mathrm radj}_E \Phi_E (\mathcal{O}_p)$ 
is isomorphic to the morphism $\Phi_{\Delta_* \mathcal{O}_{Z}} (\mathcal{O}_p)
\rightarrow \Phi_Q (\mathcal{O}_p)$ induced by 
$\Delta_* \mathcal{O}_{Z}
\xrightarrow{\eqref{eqn-pi1-unit-in-fm-right-adjunction-unit-morphism}-
\eqref{eqn-delta-unit-in-fm-right-adjunction-unit-morphism}}
Q$ and is therefore isomorphic to
$\iota^*_{p,Z} \left(\Delta_* \mathcal{O}_{Z}
\xrightarrow{\eqref{eqn-pi1-unit-in-fm-right-adjunction-unit-morphism}-
\eqref{eqn-delta-unit-in-fm-right-adjunction-unit-morphism}}
Q\right)$. By the base change around 
$\eqref{eqn-iota-p-X1-commutative-square}$ we have 
$\pi_{k *} \iota^*_{p,Z} \simeq \iota_p^* \tilde{\pi}_{1 *}$ and
therefore
$\pi_{k *}\left(\mathcal{O}_p \rightarrow \Phi^{\mathrm radj}_E \Phi_E
(\mathcal{O}_p)\right)$ is isomorphic to
$\iota_p^* \tilde{\pi}_{1 *} \left(\Delta_* \mathcal{O}_{Z}
\xrightarrow{\eqref{eqn-pi1-unit-in-fm-right-adjunction-unit-morphism}-
\eqref{eqn-delta-unit-in-fm-right-adjunction-unit-morphism}}
Q\right)$ and hence, by the first claim, to $\iota^*_p \left(
\mathcal{O}_{Z} \xrightarrow{\eqref{eqn-shhom-pipiE-E-map}}
\pi_{Z *}\rder\shhomm(\pi^*_X \pi_{X *} E, E) \right)$, 
which is precisely the natural morphism 
$k \xrightarrow{\eqref{eqn-plback-shhom-pipiE-E-map}} 
\rder\homm_X(\pi_{X*} E, E_p)$. This settles the second claim
of the lemma and the last claim follows immediately by taking cones. 
\end{proof}

\begin{lemma} 
\label{lemma-support-of-Q-is-contained-within-the-diagonal}
Let $E$ be a perfect object of $D(Z \times X)$. 
Then $E$ is orthogonal over $Z$ if and only if 
the support of the object $Q = 
\pi_{13 *} \left( \pi^*_{12} E \otimes \pi^*_{23} E^\vee \otimes
\pi^*_{23} \pi^!_X \mathcal{O}_X \right)$ is contained within 
the diagonal $\Delta \subset Z \times Z$. Consequently, if
$E$ is orthogonal over $Z$ then for any closed point 
$p \in Z$ the object $F_E(\mathcal{O}_p)$, if
non-zero, is supported at precisely the point $p$. 
\end{lemma}
\begin{proof}
Let $q_{1}$ and $q_{2}$ be closed points of $Z$, let $\bar{q} =
(q_{1}, q_{2})$ be the corresponding point of $Z \times Z$ and 
denote by $\iota_{\bar{q}}$
the closed embedding $\bar{q} \hookrightarrow Z \times Z$. Since 
$Q \in D(Z \times Z)$ its cohomology sheaves are coherent and
only finite number of them are non-zero. 
It follows from the standard spectral sequence $\lder^{i} \iota^{*}_{\bar{q}}
\mathcal{H}^{j} Q \Rightarrow \lder^{i + j}  \iota^{*}_{\bar{q}} Q$
that $\bar{q} \in \supp_{Z \times Z} Q$ if and only if $
\iota^{*}_{\bar{q}} Q \ne 0$.

We have an independent fibre square
\begin{align}
\xymatrix{
X\; \ar[d]_{\pi_{k}} \ar@{^{(}->}[r]^>>>>>>{\iota_{X\bar{q}}} &
Z \times X \times Z \ar[d]^{\pi_{13}} \\
\spec k \; \ar@{^{(}->}[r]^>>>>>>>{\iota_{\bar{q}}} & Z \times Z 
}
\end{align}
and thus 
\begin{align}
\iota^*_{\bar{q}} Q & = \iota^*_{\bar{q}} 
\pi_{13 *} \left( \pi^*_{12} E \otimes \pi^*_{23} E^\vee 
\otimes \pi^*_{23} \pi^!_X \mathcal{O}_X \right)
\simeq \\ \notag
& \simeq \pi_{k *} \iota^*_{X\bar{q}} 
\left( \pi^*_{12} E \otimes \pi^*_{23} E^\vee 
\otimes \pi^*_{23} \pi^!_X \mathcal{O}_X \right) \simeq \\ \notag
& \simeq \pi_{k *} \left( \rder\shhomm(E_{q_2}, E_{q_1}) \otimes
\iota^*_{Xq_2} 
\pi^!_X \mathcal{O}_X\right)
\end{align}
We have a pair of independent fibre squares
\begin{align}
\xymatrix{
X \ar^{\iota_{X_{q_2}}}[r] \ar_{\pi_{k}}[d] & 
Z \times X \ar^{\pi_{X}}[r] \ar_{\pi_{Z}}[d] & 
X \ar^{\pi_{k}}[d] \\
\spec k \ar_{\iota_{q_2}}[r] &
Z \ar_{\pi_{k}}[r] & 
\spec k
}
\end{align}
and thus $\pi^{!}_{X} \mathcal{O}_X
\simeq \pi^{*}_{Z} D_{Z/k}$, where $D_{Z/k}$ is the
dualizing complex $\pi^{!}_{k}(k)$ on $Z$. Therefore
$$\iota^*_{X_{q_2}} \pi^!_X \mathcal{O}_X \simeq 
\iota^*_{X_{q_2}} \pi^{*}_{Z} D_{Z/k} \simeq \pi^*_k \iota^*_{q_2}
D_{Z/k},$$ and so finally: 
\begin{align} \label{eqn-pullback-of-Q-final}
\iota^*_{\bar{q}} Q & = \pi_{k *} \left( \rder\shhomm(E_{q_2}, E_{q_1}) \otimes
\iota^*_{X_{q_2}} 
\pi^!_X \mathcal{O}_X\right) \simeq \\ 
\notag & \simeq
\pi_{k *} \left( \rder\shhomm(E_{q_2}, E_{q_1}) \otimes
\pi^*_k \iota^*_{q_2} D_{Z/k} \right) \simeq \\
\notag & \simeq
\pi_{k *} \rder\shhomm(E_{q_2}, E_{q_1}) \otimes
\iota^*_{q_2} D_{Z/k} \simeq 
\rder \homm_{D(X)}(E_{q_2}, E_{q_1}) \otimes \iota^*_{q_2} D_{Z/k}.
\end{align}

By
\cite[Lemma 1.3.7]{AvramovIyengarLipman-ReflexivityAndRigidityForComplexesIISchemes}
the support of any semi-dualizing (and, in particular,
of any dualizing) complex on a noetherian scheme is the whole of
the scheme. Therefore $\iota^*_{q_2} D_{Z/k}$ is non-zero for any $q_2
\in Z$. It then follows from \eqref{eqn-pullback-of-Q-final}
that $\iota^*_{\bar{q}} Q \neq 0$ if and only if
$\homm^i_{D(X)}(E_{q_2}, E_{q_1}) \neq 0$ for some $i \in \mathbb{Z}$.
Therefore the support of $Q$ in $Z \times Z$ consists precisely
of all points $(q_1,q_2)$ for which 
$\homm^i_{D(X)}(E_{q_2}, E_{q_1}) \neq 0$ for some $i \in \mathbb{Z}$. 
Whence the assertion that $E$ is orthogonal over $Z$ if and only 
if the support of $Q$ lies within the diagonal of $Z \times Z$. 

For the second assertion, recall that 
$\Phi^{\mathrm radj}_E \Phi_E (\mathcal{O}_p) \simeq \iota_{p,Z}^* Q$ 
and therefore $\iota_{p,Z}^* Q$ fits into 
an exact triangle
$$ \mathcal{O}_p \rightarrow \iota_{p,Z}^* Q \rightarrow
F_E(\mathcal{O}_p)[1] $$
in $D(Z)$. Since the support of $\mathcal{O}_p$ is $p$ and
the support of $\iota_{p,Z}^* Q$ lies within 
$\iota_{p,Z}^{-1} \supp_{Z \times Z} Q \subseteq \iota_{p,Z}^{-1}
\Delta = p$, it follows that the support of $F_E(\mathcal{O}_p)$
also lies within the point $p$. If the object $F_E(\mathcal{O}_p)$ is
non-zero its support is closed and non-empty and must therefore be
precisely $p$. 
\end{proof}
\begin{proof}[Proof of Proposition
 \ref{prps-tfae-left-cotwist-is-an-equivalence}]
$(\ref{item-sphericity-condition-on-exts-for-fibres})
\Leftrightarrow (\ref{item-invertibility-of-L_E})$:
By \cite{AvramovIyengarLipman-ReflexivityAndRigidityForComplexesIISchemes},
Theorem 1.5.2 the object $\mathcal{L}_E$ is invertible if and only if
for every closed point $p \in Z$ it is isomorphic in the neighborhood
of $p$ to $\mathcal{O}_Z[d_p]$ for some $d_p \in \mathbb{Z}$. 
This is equivalent to having $\iota^*_p \mathcal{L}_E = k[d_p]$. 
We have an exact triangle 
\begin{align} \label{eqn-pullback-of-shhom-pipiE-E-map}
k 
\xrightarrow{\eqref{eqn-plback-shhom-pipiE-E-map}} 
\rder\homm_{D(X)}(\pi_{X *} E, E_p) \rightarrow \iota_{p}^{*} \mathcal{L}_{E}
\end{align}
in $D(\vectspaces)$. Hence 
$\iota^*_p \mathcal{L}_E = k[d_p]$ for some $d_p \in \mathbb{Z}$ is
equivalent to either $\rder\homm_{D(X)}(\pi_{X *} E, E_p) = 0$
and $d_p = 1$ or to $\rder\homm_{D(X)}(\pi_{X *} E, E_p) = k \oplus k[d_p]$ and 
$\eqref{eqn-plback-shhom-pipiE-E-map}$, and hence 
\eqref{eqn-nat-morphism-pi_X*-E-to-E_p}, not being the zero morphism. 
Therefore to establish 
$(\ref{item-sphericity-condition-on-exts-for-fibres})
\Leftrightarrow (\ref{item-invertibility-of-L_E})$
and the first of the two assertions in the end it remains only to show
that if $E_p \neq 0$ then $\rder\homm_{D(X)}(\pi_{X *} E, E_p) \neq 0$.

By Lemma 
\ref{lemma-pullback-of-L-E-morphism-is-iso-to-pushdown-of-the-adjunction}
the morphism $\eqref{eqn-plback-shhom-pipiE-E-map}$
is isomorphic to $\pi_{k *}$ applied to the adjunction unit 
$\mathcal{O}_p \rightarrow \Phi^{\mathrm{radj}}_E \Phi_E (\mathcal{O}_p)$.  
If $E_p = \Phi_E(\mathcal{O}_p) \neq 0$, then this adjunction unit
is non-zero and hence $\Phi^{\mathrm{radj}}_E \Phi_E (\mathcal{O}_p) \neq 0$. 
In Lemma \ref{lemma-support-of-Q-is-contained-within-the-diagonal} we've
shown that both 
$\mathcal{O}_p$ and $\Phi^{\mathrm{radj}}_E \Phi_E (\mathcal{O}_p)$
are supported at $p \in Z$. It suffices therefore to show that
the functor $\pi_{k *}$ is injective on objects
of the full subcategory $D_{p}(Z)$ of $D(Z)$
consisting of the complexes whose cohomology is supported at $p \in
Z$. Indeed, let $U$ be any open 
affine subset of $Z$ containing $p$, let $\iota_U$ be the
corresponding open immersion and observe
that $\iota_{U *}$ restricts to an equivalence 
$\iota_{U *}\colon D_{p}(U) \xrightarrow{\sim} D_{p}(X)$
whose inverse is $\iota^*_{U}$. On the other hand, 
$D_p(U) \xrightarrow{\pi_{k *}} D(\vectspaces)$ decomposes as
\begin{align} \label{eqn-decomposition-of-the-pushdown-to-spec-k}
D_p(U) \xrightarrow{\Gamma^*} D_p(\mathcal{O}_X(U)\text{-}\modd) 
\xrightarrow{\text{forgetful}} D(\vectspaces)
\end{align}
Here $\Gamma^*$ is the derived global sections functor and 
it is an equivalence since $U$ is affine. The 
functor of forgetting the $\mathcal{O}_X(U)$-module 
structure is injective on objects. The claim now follows.

$\eqref{item-invertibility-of-L_E} \Leftrightarrow 
\eqref{item-the-left-co-twist-is-an-equivalence}$:
The object $\mathcal{L}_E$ is invertible if and only if
for every closed point $p \in Z$ we have 
$\iota^*_p(\mathcal{L}_E) = k[d_p]$ for some $d_p \in \mathbb{Z}$. 
By Lemma 
\ref{lemma-pullback-of-L-E-morphism-is-iso-to-pushdown-of-the-adjunction}
we have $\iota^*_p(\mathcal{L}_E) = \pi_{k *} F_E(\mathcal{O}_p)[1]$. By
Lemma \ref{lemma-support-of-Q-is-contained-within-the-diagonal}
the object $F_E(\mathcal{O}_p)$ lies in 
the full subcategory $D_{p}(Z)$ of $D(Z)$
consisting of the complexes whose cohomology is supported at 
$p$. Finally, the decomposition 
$\eqref{eqn-decomposition-of-the-pushdown-to-spec-k}$ makes it
clear that the only object of $D_{p}(Z)$ whose image in
$D(\vectspaces)$ under $\pi_{k *}$ is precisely $k$ is the point
sheaf $\mathcal{O}_p$. We conclude that $\mathcal{L}_E$ is invertible
if and only if 
\begin{align}
\label{eqn-F_E-of-a-point-sheaf-is-a-sheaf-of-a-point-sheaf}
\forall\; p \in Z, \quad F_E(\mathcal{O}_p) = \mathcal{O}_p[d] 
\quad\quad \text{ for some } d \in \mathbb{Z}.
\end{align}

Suppose $\eqref{eqn-F_E-of-a-point-sheaf-is-a-sheaf-of-a-point-sheaf}$ holds. 
Let $Q'$ be the Fourier--Mukai kernel of the co-twist $F_E$. 
Since $\iota^*_{p,Z} Q' \simeq F_E(\mathcal{O}_p)$ 
by the semicontinuity theorem 
(\cite{Grothendieck-EGA-III-2}, \em Th{\'e}or{\`e}me 7.6.9\rm)
the shift $d$ in \eqref{eqn-F_E-of-a-point-sheaf-is-a-sheaf-of-a-point-sheaf}
is constant on every connected component of $Z$. Let $U$ be such 
a connected component, then the spectral sequence argument 
of \cite{Bridg97}, Lemma 4.3 
shows that the restriction of $Q'$ to $U \times Z$ is
the shift by $d$ of a coherent sheaf flat over $U$,
whose restriction to the fibre $\{p\} \times Z$ over every point 
$p \in U$ is precisely $\mathcal{O}_p$. Any such sheaf is necessarily
a line bundle supported on the diagonal 
$U \hookrightarrow U \times Z$. We conclude that
globally $Q'= \Delta_* \mathcal{L}'$ for some invertible object 
$\mathcal{L}'$ of $D(Z)$. This immediately implies that 
the corresponding Fourier--Mukai transform $F_E$ is an equivalence. 

Conversely, suppose $F_E$ is an equivalence. Let $p$ be any closed 
point of $Z$. As $F_E$ is an equivalence we have $\homm^{<0}_{D(Z)}
\left( F_E(\mathcal{O}_p), F_E(\mathcal{O}_p) \right) = 0$ and 
$\homm^{0}_{D(Z)}\left( F_E(\mathcal{O}_p), F_E(\mathcal{O}_p)
\right) = k$. By Lemma 
\ref{lemma-support-of-Q-is-contained-within-the-diagonal} the support
of $F_E(\mathcal{O}_p)$ is precisely $p$. Now the same spectral
sequence argument as in Proposition $2.2$ of
\cite{BondalOrlov-ReconstructionOfAVarietyFromTheDerivedCategory}
shows that $F_E(\mathcal{O}_p) = \mathcal{O}_p[d]$ for some $d \in
\mathbb{Z}$.  

For the second of the two assertions in the end:
it follows from the definition of $F_E$ that $Q'$ is the object
$$ \cone \left(\Delta_* \mathcal{O}_{Z}
\xrightarrow{\eqref{eqn-pi1-unit-in-fm-right-adjunction-unit-morphism}-
\eqref{eqn-delta-unit-in-fm-right-adjunction-unit-morphism}}
Q\right) [-1] $$
of $D(Z \times Z)$. Therefore by Lemma
\ref{lemma-pullback-of-L-E-morphism-is-iso-to-pushdown-of-the-adjunction}
we have $\tilde{\pi}_{1 *} Q' \simeq \mathcal{L}_E[-1]$. 
Above we've shown that $Q' = \Delta_* \mathcal{L}'$ for some
invertible object $\mathcal{L}' \in D(Z)$
and since $\tilde{\pi}_{1 *} \Delta_* \simeq \id_{D(Z)}$ it 
follows that $\mathcal{L}' \simeq \mathcal{L}_E [-1]$. 
\end{proof}

The following lemma shows that when verifying 
the condition \eqref{item-sphericity-condition-on-exts-for-fibres}
of Prop. \ref{prps-tfae-left-cotwist-is-an-equivalence} 
one doesn't have to check that
$\pi_{X *} E \xrightarrow{ \eqref{eqn-nat-morphism-pi_X*-E-to-E_p} } E_p$
is non-zero provided the integer $d_p$ is non-positive:
  
\begin{lemma}
\label{lemma-d_p-less-than-zero-morphism-pi_X*-E-to-E_p-doesnt-vanish}
Let $E$ be a perfect object of $D(Z \times X)$ orthogonal over $Z$. 
Let $p \in Z$ be such that
$$ \rder\homm_{D(X)}(\pi_{X *} E, E_p) = k \oplus k[d_p]
\quad\text{ for some } d_p \in \mathbb{Z}.$$
If $d_p \leq 0$ then the natural morphism 
$\pi_{X *} E 
\xrightarrow{ \eqref{eqn-nat-morphism-pi_X*-E-to-E_p} }
E_p$ is non-zero. 
\end{lemma}

\begin{proof}
Recall that  
$k \xrightarrow{\eqref{eqn-plback-shhom-pipiE-E-map}}
\rder\homm_{D(X)}(\pi_{X *} E, E_p)
$
sends $1$ to the morphism \eqref{eqn-nat-morphism-pi_X*-E-to-E_p}.
By Lemma 
\ref{lemma-pullback-of-L-E-morphism-is-iso-to-pushdown-of-the-adjunction}
the morphism $\eqref{eqn-plback-shhom-pipiE-E-map}$
is isomorphic to $\pi_{k *}$ applied to the adjunction unit 
$\mathcal{O}_p \rightarrow \Phi^{\mathrm{radj}}_E \Phi_E (\mathcal{O}_p)$. 
By adjunction 
$$\pi_{k *}\left(
\mathcal{O}_p \rightarrow \Phi^{\mathrm{radj}}_E \Phi_E (\mathcal{O}_p)
\right)$$
being non-zero is equivalent to
\begin{align*}
\pi^*_k \pi_{k *} \mathcal{O}_p = \mathcal{O}_Z \twoheadrightarrow \mathcal{O}_p \rightarrow
\Phi^{\mathrm{radj}}_E \Phi_E (\mathcal{O}_p)
\end{align*}
being non-zero.
It suffices therefore to show that the induced morphism 
$\mathcal{O}_p \rightarrow 
\mathcal{H}^{0}(\Phi^{\mathrm{radj}}_E \Phi_E (\mathcal{O}_p))$ is
non-zero. 

By Lemma \ref{lemma-support-of-Q-is-contained-within-the-diagonal}
the support of $\Phi^{\mathrm{radj}}_E \Phi_E (\mathcal{O}_p)$ is $p$. 
Hence $\pi_{k *} \Phi^{\mathrm{radj}}_E \Phi_E (\mathcal{O}_p) = k
\oplus k[d_p]$ implies that the only non-zero cohomology sheaves
of $\Phi^{\mathrm{radj}}_E \Phi_E$ are $\mathcal{O}_p$ in degree $0$
and $-d_p$. The standard spectral sequence 
$$
\ext^i_Z\left(\mathcal{O}_p, \mathcal{H}^j ( 
\Phi^{\mathrm{radj}}_E \Phi_E (\mathcal{O}_p)) \right) 
\Rightarrow \homm^{i+j}_{D(Z)}\left(\mathcal{O}_p, 
\Phi^{\mathrm{radj}}_E \Phi_E (\mathcal{O}_p) \right) $$ 
and the fact that $d_p \leq 0$ imply that 
$$
\homm_{D(Z)}\left(\mathcal{O}_p, 
\Phi^{\mathrm{radj}}_E \Phi_E (\mathcal{O}_p) \right) 
\simeq
\homm_Z\left(\mathcal{O}_p, 
\mathcal{H}^0(\Phi^{\mathrm{radj}}_E \Phi_E (\mathcal{O}_p)) \right), $$
the isomorphism being given by restriction to the $0$th cohomology 
sheaves. Therefore it suffices to show that the adjunction unit
$\mathcal{O}_p \rightarrow \Phi^{\mathrm{radj}}_E \Phi_E (\mathcal{O}_p)$ 
itsef is a non-zero 
morphism. But since $\rder\homm_{D(X)}(\pi_{X *} E, E_p) \neq 0$, we must
also have $E_p \neq 0$ and hence the adjunction unit
is non-zero. 
\end{proof}

Suppose now that $E$ satisfies the equivalent conditions 
of Proposition \ref{prps-tfae-left-cotwist-is-an-equivalence}.
Then the co-twist $F_E$ is an autoequivalence of $D(Z)$
whose Fourier--Mukai kernel is $\Delta_* \mathcal{L}_E[-1]$.
Recall the definition of $\Phi_E$ being spherical given 
in Defn.~\ref{defn-spherical-functor}. The four functorial 
exact triangles in it are constructed on the level of Fourier-Mukai
kernels. In other words, we
have implicitly fixed once and for all a completion to exact
triangles of the four morphisms given in Section
\ref{section-adjunction-units-and-fm-transforms}
which underlie the adjunction units and co-units of $\Phi_E$. 

Let $\kappa$ be the morphism in the chosen completion of
\eqref{eqn-pi1-unit-in-fm-right-adjunction-unit-morphism}-\eqref{eqn-delta-unit-in-fm-right-adjunction-unit-morphism}
to an exact triangle
\begin{align} \label{eqn-L_E-exact-triangle}
\Delta_* \mathcal{O}_{Z} 
\xrightarrow{\eqref{eqn-pi1-unit-in-fm-right-adjunction-unit-morphism}-
\eqref{eqn-delta-unit-in-fm-right-adjunction-unit-morphism}}
\pi_{13 *} \left(\pi_{12}^* E \otimes \pi_{23}^*(E^\vee \otimes
\pi^!_X \mathcal{O}_{X})\right)
\xrightarrow{\kappa} 
\Delta_* \mathcal{L}_E.
\end{align}
Denote by $\kappa_{FM}$ the corresponding natural 
transformation in the corresponding functorial exact triangle
\eqref{eqn-cotwist-triangle}.

By Cor.~\ref{cor-spherical-functors-for-dummies} the functor $\Phi_E$
is spherical, and thus the object $E$ is spherical over $Z$, if and only if
$F_E$ is an autoequivalence and the composition
\begin{align}
\label{eqn-canonical-morphism-of-left-adjoint-into-twisted-right-one}
\Phi^{\mathrm radj}_E \xrightarrow{\Phi^{\mathrm radj}_E(\id \rightarrow \Phi_E
\Phi^{\mathrm ladj}_E)}
\Phi^{\mathrm radj}_E \Phi_E \Phi^{\mathrm ladj}_E  
\xrightarrow{\kappa_{FM}}
F_E[1] \Phi^{\mathrm ladj}_E 
\end{align}
is an isomorphism of functors. The Fourier--Mukai kernel of 
$\Phi^{\mathrm radj}_E$ is $E^\vee \otimes \pi^!_X(\mathcal{O}_X)$ 
and the Fourier--Mukai kernel of $F_E[1] \Phi^{\mathrm ladj}_E$ is 
$$ \pi^*_{Z}(\mathcal{L}_E) \otimes \left( 
E^\vee \otimes \pi^!_{Z}(\mathcal{O}_{Z})\right)
\simeq E^\vee \otimes \pi^!_{Z}(\mathcal{L}_E).$$
We can therefore define:
\begin{defn}
\label{defn-canonical morphism-alpha-E} 
Define $\alpha$ to be the morphism 
$$ 
E^\vee \otimes \pi^!_X(\mathcal{O}_X) \xrightarrow{\alpha}
E^\vee \otimes \pi^!_Z (\mathcal{L}_E)
$$ 
of Fourier Mukai kernels which underlies the natural moprhism 
\eqref{eqn-canonical-morphism-of-left-adjoint-into-twisted-right-one}
if we choose 
\eqref{eqn-pi1-unit-in-fm-left-adjunction-unit-morphism}-\eqref{eqn-delta-unit-in-fm-left-adjunction-unit-morphism}
and $\kappa$ as underlying morphisms of 
$\id \rightarrow \Phi_E \Phi^{\mathrm ladj}_E$ and $\kappa_{FM}$,
respectively. 
\end{defn}
By Lemma 
\ref{lemma-morphism-of-FM-kernels-induces-iso-of-FMs-iff-itis-an-iso-itself}
the composition
\eqref{eqn-canonical-morphism-of-left-adjoint-into-twisted-right-one}
is an isomorphism if and only if the underlying morphism $\alpha$ is. 
We therefore obtain immediately the main result of this section:
\begin{theorem} 
\label{theorem-sphericity-for-orthogonal-objects-of-ZxX}
Let $X$ and $Z$ be two separable schemes of finite type over $k$.
Let $E$ be a perfect object of $D(Z \times X)$ orthogonal over $Z$. 
Then $E$ is spherical over $Z$ if and only if 
\begin{enumerate}
\item
For every closed point $p \in Z$ such that the fibre $E_p$
is non-zero  
$$ \rder\homm_{D(X)}(\pi_{X *} E, E_p) = k \oplus k[d_p]
\quad\text{ for some } d_p \in \mathbb{Z}$$
and the natural morphism $\pi_{X *} E
\xrightarrow{ \eqref{eqn-nat-morphism-pi_X*-E-to-E_p} } E_p$
is not zero. 
\item The canonical morphism 
\begin{align*}
E^\vee \otimes \pi^!_X(\mathcal{O}_X) 
\xrightarrow{\alpha}
E^\vee \otimes \pi^!_{Z}(\mathcal{L}_E)
\quad\quad\quad \text{(see Definition \ref{defn-canonical
morphism-alpha-E})}
\end{align*}
is an isomorphism.
\end{enumerate}
\end{theorem}
Whenever $E$ is orthogonally spherical over $Z$, the object $\mathcal{L}_E$ 
is invertible in $D(Z)$ and so locally around any closed point 
$p \in Z$ we have $\mathcal{L}_E \simeq \mathcal{O}_Z[d_p]$
for some $d_p \in \mathbb{Z}$. Over the locus where $Z$ and $X$ are not 
too degenerate this integer is the
difference in dimensions between $X$ and $Z$:
\begin{prps}
\label{prps-the-shift-of-L_E-is-the-difference-in-dimensions}
Let $E$ be an object of $D(Z \times X)$ orthogonally spherical over
$Z$. Let $(p,q) \in Z \times X$ be a Gorenstein point in the support
of $E$ if such exists. Then 
$$ d_p = - (\dim_q X - \dim_p Z). $$
\end{prps}
\begin{proof}
If $E$ is spherical over $Z$ the canonical map
\begin{align*} 
E^\vee \otimes \pi^!_X(\mathcal{O}_X) \xrightarrow{\alpha}
E^\vee \otimes \pi^!_Z(\mathcal{L}_E) 
\end{align*}
is an isomorphism. Let us restrict it to $\spec \mathcal{O}_{Z \times X,
(p,q)}$. Since $\mathcal{O}_{Z \times X, (p,q)} 
= \mathcal{O}_{Z,p} \otimes \mathcal{O}_{X,q}$ is Gorenstein, 
the structure map $\spec \mathcal{O}_{Z \times X, (p,q)} \rightarrow \spec k$ 
is Gorenstein.  Therefore the projections $\pi_{Z,p}$
and $\pi_{X,q}$ are Gorenstein, since we can filter 
$\spec \mathcal{O}_{Z \times X, (p,q)} \rightarrow \spec k$
through them 
\begin{align}
\xymatrix{
\spec \mathcal{O}_{Z \times X, (p,q)}
\ar[d]_{\pi_{Z,q}}
\ar[r]^<<<<{\pi_{X,q}}
&
\spec \mathcal{O}_{X,q} \ar[d]
\\
\spec \mathcal{O}_{Z,p} \ar[r]
& k
} 
\end{align}
and for perfect maps (and therefore for flat maps such as these)
the composition of two maps is Gorenstein
if and only if both composants are, 
\cite[Prop. 2.3]{AvramovFoxby-GorensteinLocalHomomorphisms}.
Therefore
\begin{align*}
\pi^!_{Z,p}(\mathcal{O}_{Z,p}) = \mathcal{O}_{Z \times X, (p,q)} [
\dim \mathcal{O}_{Z \times X, (p,q)} - \dim \mathcal{O}_{Z,p}] = 
\mathcal{O}_{Z \times X,(p,q)}[\dim \mathcal{O}_{X,q}] \\
\pi^!_{X,q}(\mathcal{O}_{X,q}) = \mathcal{O}_{Z \times X, (p,q)} [
\dim \mathcal{O}_{Z \times X, (p,q)} - \dim \mathcal{O}_{X,q}] = 
\mathcal{O}_{Z \times X,(p,q)}[\dim \mathcal{O}_{Z,p}] 
\end{align*}
and so $\alpha$ restricts to $\spec 
\mathcal{O}_{Z \times X, (p,q)}$ as
$$
E^\vee|_{\mathcal{O}_{Z \times X, (p,q)}} [\dim \mathcal{O}_{Z,p}] 
\xrightarrow{\sim}
E^\vee|_{\mathcal{O}_{Z \times X, (p,q)}} [\dim \mathcal{O}_{X,q} + d_p] 
$$
Since $(p,q)$ lies in the support of $E$, 
the restriction $E^\vee|_{\mathcal{O}_{Z \times X, (p,q)}}$ is a
non-zero bounded complex. So 
$$ \dim \mathcal{O}_{Z,p} = \dim \mathcal{O}_{X,q} + d_p $$
whence the claim.  
\end{proof}

\subsection{Concerning the canonical morphism $\alpha$}
\label{section-the-canonical-morphism-alpha}

A reader who wasn't at all disturbed by the words 
``the canonical morphism $\alpha$ is an isomorphism'' 
in Theorem \ref{theorem-sphericity-for-orthogonal-objects-of-ZxX}
probably doesn't need to read this section. 
However to apply 
Theorem \ref{theorem-sphericity-for-orthogonal-objects-of-ZxX}
to show that an object is spherical one needs to compute 
the canonical morphism 
$$ E^\vee \otimes \pi^!_X(\mathcal{O}_X) \xrightarrow{\alpha}
E^\vee \otimes \pi^!_Z(\mathcal{L}_E) $$
described in Definition \ref{defn-canonical morphism-alpha-E} 
and show it to be an isomorphism. In all but few very 
simple examples computing this morphism directly, by computing 
the morphisms of the kernels underlying both terms 
of \eqref{eqn-canonical-morphism-of-left-adjoint-into-twisted-right-one}
and then composing them, is not a pleasant endeavour.  

Fortunately Lemma \ref{lemma-sphericity-condition-the-commutation-condition}
gives us a different characterisation of $\alpha$ by telling us that
$\alpha$ induces the unique natural transformation $\alpha_{FM}$ 
which makes the diagram
\begin{align} 
\label{eqn-commutation-condition-on-the-level-of-FM-transforms}
\vcenter{
\xymatrix{
\Phi^{\mathrm radj}_E \Phi_E  \ar^{\kappa_{FM}}[r] 
\ar_{\alpha_{FM}}[d] 
& F_E[1] \\
F_E[1] \Phi^{\mathrm ladj}_E \Phi_E 
\ar_<<<<<<<<{\quad\quad F_E[1](\Phi^{\mathrm ladj}_E \Phi_E \rightarrow \id)}[ur] 
&
}
}
\end{align}
commute. It follows by Lemma 
\ref{lemma-morphism-of-FM-kernels-induces-iso-of-FMs-iff-itis-an-iso-itself}
that showing $\alpha$ to be an isomorphism is 
equivalent to exhibiting some isomorphism 
$E^\vee \otimes \pi^!_X(\mathcal{O}_X) \xrightarrow{\sim} 
E^\vee \otimes \pi^!_Z(\mathcal{L}_E)$ which 
makes diagram \eqref{eqn-commutation-condition-on-the-level-of-FM-transforms}  
commute. 

\begin{prps}
\label{prps-getting-rid-of-canonical-map-alpha}
Let $E$ be a perfect object of $D(Z \times X)$ orthogonal over $Z$
and suppose that $\mathcal{L}_E$ is an invertible object of $D(Z)$.
Assume (for simplicity) that $Z$ is connected, 
then $\mathcal{L}_E = L[d]$ for some line bundle 
$L \in \picr Z$ and $d \in \mathbb{Z}$. 
Assume further that $d < 0$ or, more generally, that $d \neq 0, 1$ 
and 
$$ \ext_{Z \times Z}^d(\Delta_* \mathcal{O}_Z, \Delta_* L) = 0. $$

If $E^\vee \otimes \pi^!_X (\mathcal{O}_X) \xrightarrow{\sim}  
E^\vee \otimes \pi^!_Z (\mathcal{L}_E)$ then 
the canonical map $\alpha$ is an isomorphism. 
\end{prps}
\begin{proof}
Let $\alpha'$ denote some isomorphism 
$E^\vee \otimes \pi^!_X (\mathcal{O}_X) \xrightarrow{\sim}  
E^\vee \otimes \pi^!_Z (\mathcal{L}_E)$. 
By \cite[Theorem 3.1]{AnnoLogvinenko-OnTakingTwistsOfFourierMukaiFunctors}
the natural transformation 
$F_E[1]\Phi^{\mathrm ladj}_E\Phi_E \rightarrow F_E[1]$
is induced by the following morphism of Fourier-Mukai kernels:
 \begin{align} 
 \label{eqn-derived-restriction-morphism-L_E-version}
 \pi_{13 *}\left(\pi_{12}^* E \otimes 
 \pi_{23}^* (E \otimes \pi^!_Z(\mathcal{L}_E))\right) 
 \xrightarrow{\id \rightarrow \Delta_* \Delta^*}
 \Delta_* \Delta^* \pi_{13 *} \left( 
 \pi_{12}^* E \otimes \pi_{23}^* (E^\vee \otimes \pi^!_Z(\mathcal{L}_E))\right) 
 \\
 \label{eqn-two-squares-isomorphism-L_E-version}
 \Delta_* \Delta^* \pi_{13 *} 
 \left(\pi_{12}^* E \otimes 
 \pi_{23}^*(E^\vee \otimes \pi^!_Z(\mathcal{L}_E)\right)
 \xrightarrow{
 \Delta^*\pi_{13 *}(\pi^*_{12} \text{-} \otimes \pi^*_{13}\text{-})
 \;\simeq\; \pi_{Z*}(\text{-}\otimes\text{-})}
 \Delta_* \pi_{Z *}  
 \left(E \otimes (E^\vee \otimes \pi^!_Z(\mathcal{L}_E))\right)
 \\ 
 \label{eqn-E-E-dual-trace-morphism-L_E-version}
 \Delta_* \pi_{Z *}  
 \left(E \otimes (E^\vee \otimes \pi^!_Z(\mathcal{L}_E))\right) 
 \xrightarrow{E \otimes E^\vee \otimes \rightarrow \id}
 \Delta_* \pi_{Z *} \pi^!_Z(\mathcal{L}_E)  
 \\
 \label{eqn-rpi_1-trace-morphism-L_E-version}
 \Delta_* \pi_{Z *} \pi^!_Z(\mathcal{L}_E)
 \xrightarrow{\pi_{1 *} \pi^!_1 \rightarrow \id}
 \Delta_* \mathcal{L}_E.
 \end{align}
 
 Thus the analogue of 
\eqref{eqn-commutation-condition-on-the-level-of-FM-transforms}
for $\alpha'$ is induced by the following 
diagram of Fourier--Mukai kernel morphisms
 
 \begin{align} 
 \label{eqn-commutation-condition-on-the-morphism-alpha}
 \vcenter{
 \xymatrix{
 \pi_{13 *} \left(\pi_{12}^* E \otimes \pi_{23}^*(E^\vee \otimes
 \pi^!_X \mathcal{O}_{X})\right)
 \ar^<<<<<{\kappa}[r] 
 \ar_{\pi_{13 *} (\pi_{12}^* E \otimes \pi_{23}^*(\alpha'))}[d] 
 & \Delta_* \mathcal{L}_E \\
 \pi_{13 *} \left(\pi_{12}^* E \otimes \pi_{23}^*(E^\vee \otimes
 \pi^!_Z \mathcal{L}_E)\right)
 \ar_{\quad \quad
 \eqref{eqn-derived-restriction-morphism-L_E-version}-\eqref{eqn-rpi_1-trace-morphism-L_E-version}}[ur] &
 }
 }
 \end{align}
It is enough to show that $\alpha'$ can be chosen so that
\eqref{eqn-commutation-condition-on-the-morphism-alpha} commutes.

Denote by $Q$ the object 
$ \pi_{13 *} (\pi^*_{12} E \otimes 
\pi^*_{23}(E^\vee \otimes \pi^!_X(\mathcal{O}_X)))$. We have 
an exact triangle in $D(Z \times Z)$:
$$ \Delta_* \mathcal{O}_{Z} 
\xrightarrow{\eqref{eqn-pi1-unit-in-fm-right-adjunction-unit-morphism}-
\eqref{eqn-delta-unit-in-fm-right-adjunction-unit-morphism}}
Q
\xrightarrow{\kappa} 
\Delta_* L[d] $$
Denote by $\mathcal{\mathcal{H}}^i$ the functor of taking $i$-th cohomology of a
complex. Since $d \neq 0,1$, the associated long exact sequence of 
cohomologies shows that the complex $Q$ has exactly two 
non-zero cohomologies: $\Delta_* \mathcal{O}_Z$ in degree $0$ 
and $\Delta_* L$ in degree $-d$. More precisely, it shows that 
the morphisms
$$\Delta_* \mathcal{O}_Z \xrightarrow{\mathcal{H}^0
(\eqref{eqn-pi1-unit-in-fm-right-adjunction-unit-morphism}-\eqref{eqn-delta-unit-in-fm-right-adjunction-unit-morphism})}
\mathcal{H}^0(Q) \quad\text{ and }\quad 
\mathcal{H}^{-d}(Q) \xrightarrow{\mathcal{H}^{-d}(\kappa)} \Delta_* L $$
are isomorphisms. Use them from now on to identify the spaces involved.
Tautologically, the map 
\begin{align} \label{eqn-dth-cohomology-of-map-kappa}
\homm_{D(Z \times Z)}(Q, \Delta_* L[d]) 
\xrightarrow{\mathcal{H}^{-d}(\text{-})}
\homm_{Z \times Z}(\Delta_* L, \Delta_* L) 
\xrightarrow{\sim}
\homm_{Z}(L, L)
\xrightarrow{\sim}
\Gamma(\mathcal{O}_Z)
\end{align}
sends $\kappa$ to the element $1$ of $\Gamma(\mathcal{O}_Z)$. 

\vskip 0.75em
\em Claim: The map \eqref{eqn-dth-cohomology-of-map-kappa} is
injective.\rm 

\begin{proof}
Clearly it suffices to show that the map $\mathcal{H}^{-d}$ in 
\eqref{eqn-dth-cohomology-of-map-kappa} is an isomorphism. 
Choose an injective resolution $I^\bullet$ of $\Delta_*$. 
There is a standard spectral 
sequence associated to the filtration by columns of the total complex 
of the bicomplex $\homm(Q^\bullet, I^\bullet)$:
$$ E^{p,-q}_2 = \ext^p_{Z \times Z}(H^{q}(Q), \Delta_* L)
\quad \Rightarrow \quad 
E^{p-q}_\infty = \homm^{p-q}_{D(Z \times Z)}(Q, \Delta_* L). $$
We are interested in the space $\homm_{D(Z \times X)} (Q, \Delta_*
L[d])$ which is the limit $E^{d}_{\infty}$ of the above spectral
sequence. Since the complex $Q$ has cohomology only in two degrees, 
there are only two potentially non-zero terms $E^{p,-q}_2$ with $p - q = d$. 
These are $E_2^{0,-d} = \homm_{Z \times Z}(\Delta_* L, \Delta_* L)$
and $E_2^{d,0} = \ext^d_{Z \times Z}(\Delta_* \mathcal{O}_Z, 
\Delta_* L)$. The space $E_2^{d,0}$ is $0$ by assumption, so we have 
$E^{d}_\infty = E^{0,-d}_\infty$. But observe that 
there are no non-zero elements $E_2^{p,q}$
with $p < 0$, and therefore at every page of the spectral sequence 
the incoming differential $E^{-r, -d+r-1}_r \rightarrow E^{0,-d}_r$
will be zero. Thefore we have natural inclusions $E^{0,-d}_{r+1}
\hookrightarrow E^{0,-d}_r$ and the spectral sequence 
technology ensures that the natural map 
$$ \homm_{D(Z \times Z)}(Q, \Delta_* L [d])
\xrightarrow{\mathcal{H}^{-d}(\text{-})} E^{0,-d}_2 $$
lifts along each of these inclusions. Let $\beta$ denote
the map we obtain at the limit:
\begin{align}
\xymatrix{
\homm_{D(Z \times Z)}(Q, \Delta_* L [d]) 
\ar^<<<<<<<<<<{\mathcal{H}^{-d}(\text{-})}[rr] 
\ar[drr]_<<<<<<<<<<<<<<<{\beta}
& &
E^{0,-d}_2 \\
& & E^{0,-d}_\infty \ar@{^{(}->}[u]
}
\end{align}
Since $E^{0,-d}_\infty$ is the only surviving component of 
$E^{d}_\infty$ the map $\beta$ is an isomorphism, and the claim follows. 
\end{proof}

Let $Q'$ denote $\pi_{13 *} (\pi^*_{12} E \otimes 
\pi^*_{23}(E^\vee \otimes \pi^!_Z(\mathcal{L}_E)))$, 
let $\lambda$ denote the composition $Q 
\xrightarrow{ \alpha'}
Q' 
\xrightarrow{
\eqref{eqn-derived-restriction-morphism-L_E-version}-\eqref{eqn-rpi_1-trace-morphism-L_E-version}}
\Delta_* L[d]$ in diagram 
\eqref{eqn-commutation-condition-on-the-morphism-alpha}
and let $f \in \Gamma(\mathcal{O}_Z)$ be the image of $\lambda$ under 
map \eqref{eqn-dth-cohomology-of-map-kappa}. 
Since \eqref{eqn-dth-cohomology-of-map-kappa} is injective, 
$\lambda$ is then necessarily the composition
$Q \xrightarrow{\kappa} \Delta_* L [d] \xrightarrow{\pi_1^* f} \Delta_* L
[d]$. Thus showing that $\lambda = \kappa$, i.e. 
\eqref{eqn-commutation-condition-on-the-morphism-alpha} commutes, 
is equivalent to showing $f = 1$. In fact, it is enough to show that 
$f$ is invertible, as then scaling $\alpha$ by $\pi^*_{Z} \frac{1}{f}$ 
would scale $f$ to $1$.  

Suppose $f$ isn't invertible, then there exists
some closed point $p \in Z$ such that $f(p) = 0$. Let $\iota_p$ be
the inclusion $p \hookrightarrow Z$ and $\iota_{Z_p}$ to be the
corresponding inclusion $Z \xrightarrow{(p,-)} Z \times Z$. 
It is enough
to show that 
\begin{align}
\label{eqn-adjunction-counit-on-FM-kernel-level-pulled-back-to-Z_p}
\iota_{Z_p}^* \left(
Q' 
\xrightarrow{\eqref{eqn-derived-restriction-morphism-L_E-version}-\eqref{eqn-rpi_1-trace-morphism-L_E-version}}
\Delta_* L[d] \right)
\end{align}
is the zero map, as then 
$\Phi^{\mathrm ladj}_E \Phi_E(\mathcal{O}_p) \xrightarrow{\text{adj.
co-unit}} \mathcal{O}_p$ would also be a zero map, implying 
$E_p = \Phi_E(\mathcal{O}_p) = 0$. By 
Proposition \ref{prps-tfae-left-cotwist-is-an-equivalence} we would
then have $d = 1$ which contradicts our assumptions. 

Since $Q \xrightarrow{\lambda} \Delta_* L[d]$ is a composition
$Q' \xrightarrow{\eqref{eqn-derived-restriction-morphism-L_E-version}-\eqref{eqn-rpi_1-trace-morphism-L_E-version}}
\Delta_* L[d]$ and an isomorphism, it suffices to show that  
$\iota_{Z_p}^* (Q \xrightarrow{\lambda} \Delta_* L[d])$ 
is the zero map. By adjunction this is equivalent to 
the following composition vanishing:
$$
Q \xrightarrow{\lambda} \Delta_* L[d] \xrightarrow{\text{ adj. unit}} 
\iota_{Z_p *} \iota_{Z_p}^* 
\Delta_* L[d] = \mathcal{O}_{p,p}[d] $$
This holds since $\lambda$ decomposes as 
$Q \xrightarrow{\kappa} \Delta_* L [d] \xrightarrow{\pi^*_1 f }
\Delta_* L [d]$ and $\pi_1^* f(p,p) = f(p) = 0$. 
\end{proof}

Together with Lemma 
\ref{lemma-d_p-less-than-zero-morphism-pi_X*-E-to-E_p-doesnt-vanish}
this allows us to significantly strengthen the ``if'' direction of 
Thorem \ref{theorem-sphericity-for-orthogonal-objects-of-ZxX}  
when the integer $d_p$ is everywhere negative. Note that by Lemma 
\ref{prps-the-shift-of-L_E-is-the-difference-in-dimensions}
all objects spherical over $Z$ necessarily have $d_p < 0$ if 
$\dim Z <  \dim X$ and $Z$ and $X$ are not too degenerate. 
\begin{theorem}
\label{theorem-criterion-for-sphericity-when-d_p-negative}
Let $X$ and $Z$ be two separable schemes of finite type over $k$.
Let $E$ be a perfect object of $D(Z \times X)$ orthogonal over $Z$. 
Then $E$ is spherical over $Z$ if 
\begin{enumerate}
\item
For every closed point $p \in Z$ such that the fibre $E_p$
is non-zero  
$$ \rder\homm_{D(X)}(\pi_{X *} E, E_p) = k \oplus k[d_p]
\quad\text{ for some } d_p \in \mathbb{Z}_{<0}.$$
\item 
$E^\vee \otimes \pi^!_X(\mathcal{O}_X) \simeq 
E^\vee \otimes \pi^!_{Z}(\mathcal{L}_E)$.
\end{enumerate}
\end{theorem}

\section{Spherical fibrations}
\label{section-spherical-fibrations}

The results of Section
\ref{section-orthogonally-spherical-objects}
are rather general and category-theoretic owing to a rather general
and category-theoretic nature of the objects it considers: arbitrary 
complexes in the derived category of the fibre product $Z \times X$. 
We now choose to restrict ourselves to a more geometric setup 
and study what our results imply in that context.  

Firstly, we assume $Z$ and $X$ to be abstract varieties. Previously 
we have assumed them to only be separated schemes of finite type over
$k$, now we assume them to also be reduced and irreducible. Strictly
speaking, neither assumption is essential for what we prove below.  
Omitting them, however, would make the arguments more technically 
involved and the results less clear. 

Secondly, and more importantly, we restrict the range of objects 
we consider from arbitrary complexes in $D(Z \times X)$ to 
subschemes of $X$ flatly fibred over $Z$.

\subsection{Flat and perfect fibrations with proper fibres}

\begin{defn}
A \em flat fibration $W$ in $X$ over $Z$ \rm is 
a scheme $W$ equipped with a closed immersion 
$\xi\colon W \hookrightarrow X$ and a flat surjective map $\pi\colon W
\rightarrow Z$. For any closed point $p \in Z$ we denote 
by $W_p$ the set-theoretic fibre of $W$ over $p$: 
\begin{align}
\label{eqn-the-fibre-square-for-W-over-Z}
\vcenter{
\xymatrix{
W_p\; \ar[d]_{\pi_{k}} \ar@{^{(}->}[r]^>>>>>>{\iota_{Wp}} &
W\; \ar[d]^{\pi} \ar@{^{(}->}[r]^{\xi} & X
\\
\spec k \; \ar@{^{(}->}[r]_>>>>>{\iota_{p}} & Z & 
}
}
\end{align}
\end{defn}

Denote by $\iota_W$ the map $W \hookrightarrow Z \times X$ given 
by the product of $\pi$ and $\iota$.
We have $\xi = \pi_X \circ \iota_W$ and $\pi = \pi_Z \circ \iota_W$. 
Denote by $\xi_p$ the composition $\xi \circ \iota_{W_p}$, it is 
the inclusion of the fibre $W_p$ into $X$. Let $E$ denote the object
$\iota_{W *} \mathcal{O}_W$ of $D(Z \times X)$. We think of 
this object as representing $W$ in the derived category 
$D(Z \times X)$. 

The flatness of $W$ over $Z$ ensures that
the category-theoretic notion of a fibre considered in the Section 
\ref{section-orthogonally-spherical-objects} coincides for $W$
with the usual set-theoretic one:
\begin{lemma} 
\label{lemma-fibres-of-O_W-in-Z-x-X-are-precisely-O_Wp}
Let $W$ be a flat fibration in $X$ over $Z$, let $\mathcal{E}$
be an object of $D(W)$ and let $E = \iota_{W *}
\mathcal{E}$ be the corresponding object in $D(Z \times X)$. 
For any closed point $p \in Z$ denote by $\mathcal{E}_p$
the fibre $\iota^*_{W_p} \mathcal{E}$. We have
\begin{align*}
E_p \simeq \xi_{p *} \mathcal{E}_p 
\end{align*}
as objects of $D(X)$. In particular, when $\mathcal{E} =
\mathcal{O}_W$ we have
\begin{align*}
E_p \simeq \xi_{p *} \mathcal{O}_{W_p} .
\end{align*}
\end{lemma}
\begin{proof}
The fibre square in 
the diagram \eqref{eqn-the-fibre-square-for-W-over-Z} 
decomposes into two fibre squares:
\begin{align}
\label{eqn-decomposition-of-the-fibre-square-for-W-over-Z}
\vcenter{
\xymatrix{
W_p\; \ar@{^{(}->}[d]_{\xi_{p}} \ar@{^{(}->}[r]^>>>>>>>{\iota_{W_p}} &
W\; \ar@{^{(}->}[d]^{\iota_W} \ar@{^{(}->}[dr]^{\xi} & 
\\
X\; \ar[d]_{\pi_{k}} \ar@{^{(}->}[r]^>>>>>>>{\iota_{Xp}} &
Z \times X \ar[d]^{\pi_Z} \ar[r]_>>>>>>{\pi_X} & X
\\
\spec k \; \ar@{^{(}->}[r]_>>>>>>>{\iota_{p}} & Z & 
}
}
\end{align}
The fibre $E_p$ was defined to be 
the object $\iota^*_{X_p} E$ of $D(X)$. We have therefore 
$E_p = \iota^*_{X_p} \iota_{W *} \mathcal{E}$. 
Consider the base change map 
\begin{align}
\label{eqn-base-change-map-for-D-in-ZxX}
\iota^*_{X_p} \iota_{W_*} \rightarrow \xi_{p *} \iota^*_{W_p} 
\end{align}
for the top fibre square in the diagram  
\eqref{eqn-decomposition-of-the-fibre-square-for-W-over-Z}. Applying
it to $\mathcal{E}$ yields a morphism  
$$ E_p \rightarrow \xi_{p *} \mathcal{E}_p.$$
To show \eqref{eqn-base-change-map-for-D-in-ZxX}
to be an isomorphism it suffices to prove 
that the top fibre square in 
\eqref{eqn-decomposition-of-the-fibre-square-for-W-over-Z} 
is independent 
in the sense of 
\cite{Lipman-NotesOnDerivedFunctorsAndGrothendieckDuality}, \S3.10. 
This follows via \cite[Prop. 3.2.9]{Wiebel-aItHA} from $\pi_Z$
and $\pi = \pi_Z \circ \iota_W$ both being flat. 
\end{proof}

In particular, this makes it clear that $\iota_{W*} \mathcal{O}_W$
is an object of $D(Z \times X)$ which is orthogonal over $Z$. 
For any two distinct points $p$ and $q$ of $Z$ the fibres $W_p$ and
$W_q$ are disjoint in $X$ and therefore all Hom's between
$\xi_{p *} \mathcal{O}_{W_p}$ and $\xi_{q *} \mathcal{O}_{W_q}$
vanish.

In Section \ref{section-orthogonally-spherical-objects} we had
to make two technical assumptions on the object $E$ of $D(Z \times X)$. 
These were necessary for 
the adjoints of the Fourier--Mukai transform $\Phi_E$ to exist and 
to behave in a sensible way. The first assumption was that the support of $E$
is proper over $Z$. The support of $\iota_{W *} \mathcal{O}_W$ in 
$Z \times X$ is the image of $W$ under $\iota_W$, so this assumption is
equivalent to saying that the fibration morphism 
$\pi\colon W \rightarrow Z$ is proper. And $\pi$ being proper is 
equivalent to all the fibres of $W$ over closed points of $Z$ 
being schemes proper over $k$, cf. 
\cite[\em Corolaire 5.4.5\rm]{Grothendieck-EGA-II}.   

The second assumption was that $E$ is a perfect object of $D(Z \times X)$. 
This corresponds to $\iota_{W *} \mathcal{O}_W$ being perfect. We call
a fibration $W$ \em perfect \rm if this holds. Since $\pi$
is flat this condition can also be checked fibrewise: 
\begin{lemma}
\label{lemma-fibrewise-condition-for-perfection-of-O_W-in-ZxX}
Let $W$ be a flat fibration in $X$ over $Z$. Then 
it is perfect if and
only if for every closed $p \in Z$ the object $\xi_{p *} \mathcal{O}_{W_p}$
is perfect in $D(X)$. 
\end{lemma}
\begin{proof}
We first claim that $\iota_{W *} \mathcal{O}_W$ is perfect relative
to $\pi_Z\colon Z \times X \rightarrow Z$. There is  
a commutative diagram 
\begin{align}
\vcenter{
\xymatrix{
W\; \ar[d]_{\pi} \ar@{^{(}->}[r]^<<<<{\iota_W} & Z \times X \ar[dl]^{\pi_Z} \\
Z
}
}
\end{align}
and since $\iota_W$ is a closed immersion, and therefore proper, 
it takes $\pi$-perfect object to $\pi_Z$-perfect objects, cf.
\cite[Prop. 4.8]{IllusieConditionsDeFinitudeRelative}.
Since $\pi$ is flat the structure sheaf $\mathcal{O}_W$ 
is $\pi$-flat and therefore most certainly $\pi$-perfect. We conclude
that $\iota_{W *} \mathcal{O}_W$ is $\pi_Z$-perfect. 

By the fibrewise criterion for perfection 
\cite[\em Corollaire 4.6.1\rm]{IllusieConditionsDeFinitudeRelative}
an object of $D(Z \times X)$ is globally perfect if and only if 
it is $\pi_Z$-perfect and its fibre over every closed point of $Z$
is globally perfect. By Lemma 
\ref{lemma-fibres-of-O_W-in-Z-x-X-are-precisely-O_Wp}
the fibre of $\iota_{W *} \mathcal{O}_W$ over any closed $p \in Z$ is 
$\xi_{p *} \mathcal{O}_{W_p}$. The claim now follows. 
\end{proof}

The following are the two typical situations in which $\iota_{W*}
\mathcal{O}_W$ would be perfect in $D(Z \times X)$:
\begin{cor}
Let $W$ be a flat fibration in $X$ over $Z$. Any one of the
following conditions is sufficient for $W$ to be perfect:
\begin{enumerate}
\item $X$ is smooth. 
\item $Z$ is smooth and $\xi\colon W \hookrightarrow X$ is a regular
immersion. 
\end{enumerate}
\end{cor}
\begin{proof}
By Lemma \ref{lemma-fibrewise-condition-for-perfection-of-O_W-in-ZxX}
it suffices to prove that for every closed $p \in Z$ the object
$\xi_{p *} \mathcal{O}_{W_p}$ is perfect in $D(X)$. 

If $X$ is smooth, then every object of $D(X)$ is perfect and the 
claim follows trivially.

Suppose that $Z$ is smooth and $\xi$ is a regular immersion. 
To prove that $\xi_{p *} \mathcal{O}_{W_p}$ is perfect in $D(X)$
it suffices to prove that $\xi_p$ is a regular immersion. 
A regular immersion is both proper and perfect and hence 
takes perfect objects to perfect objects,
cf. \cite[\em Corollaire 4.8.1\rm]{IllusieConditionsDeFinitudeRelative}. 

Recall the commutative diagram 
$\eqref{eqn-the-fibre-square-for-W-over-Z}$. Smoothness of $Z$ 
is equivalent to $\iota_p$ being a regular immersion for
every closed point $p$ of $Z$. Since $\pi$ is faithfully flat, 
$\iota_p$ is regular if and only if 
$\iota_{W_p}$ is regular. Since
$\xi_p$ is the composition 
$$ W_p \xrightarrow{\iota_{W_p} } W \xrightarrow{\xi} X $$
and since a composition of two regular immersions is again a regular
immersion, we conclude that $\xi_p$ is regular
for every closed $p \in \mathbb{Z}$.
\end{proof}

Thus we arrive at the class of objects we want to work with:
flat and perfect fibrations in $X$ over $Z$ with proper fibres. For such
fibrations the results of Section \ref{section-orthogonally-spherical-objects} 
can be restated in a more natural way and improved upon. Our goal is 
to give a satisfying description of the following:
\begin{defn}
Let $W$ be a flat and perfect fibration in $X$ over $Z$ with proper fibres. 
We say that $W$ is a \em spherical fibration \rm if the object 
$E = \iota_{W *} \mathcal{O}_W$ is spherical over 
$Z$ in $D(Z \times X)$. 
\end{defn}

So let $W$ be a flat and perfect fibration in $X$ with proper fibres
and let $E = \iota_{W *} \mathcal{O}_W$ be the corresponding object in $D(Z
\times X)$. Recall that the co-twist $F_E$ of the Fourier--Mukai transform
$\Phi_E$ was defined as the cone of the adjunction unit 
$\id_{D(Z)} \rightarrow \Phi_E \Phi^{\mathrm radj}_E$ and that the first of
the two conditions for $E$ to be spherical was for $F_E$ to be 
an autoequivalence of $D(Z)$. 

Denote by $\mathcal{L}_W$ the object $\mathcal{L}_E$ of $D(Z)$.
It was defined in 
Defn.~\ref{defn-object-L_E-for-abstract-objects-of-D-ZxX}
as the cone of a certain composition \eqref{eqn-shhom-pipiE-E-map} of 
morphisms in $D(Z)$. This composition was later shown in Lemma
\ref{lemma-pullback-of-L-E-morphism-is-iso-to-pushdown-of-the-adjunction}
to be the pushdown from $Z \times Z$ to $Z$ via
$\tilde{\pi}_{1 *}$ of the composition  
\eqref{eqn-pi1-unit-in-fm-right-adjunction-unit-morphism}-\eqref{eqn-delta-unit-in-fm-right-adjunction-unit-morphism}.
Recall that 
\eqref{eqn-pi1-unit-in-fm-right-adjunction-unit-morphism}-\eqref{eqn-delta-unit-in-fm-right-adjunction-unit-morphism}
is the morphism of the Fourier--Mukai kernels which induces 
the adjunction unit 
$\id_{D(Z)} \rightarrow \Phi_E \Phi^{\mathrm radj}_E$. In \S4
of \cite{AnnoLogvinenko-OnTakingTwistsOfFourierMukaiFunctors}
we have demonstrated that whenever the object $E$ of $D(Z \times X)$
is a pushforward of an object from some closed subscheme of $Z \times X$, 
as is the case here, there exists a better, more economical
decomposition of this morphism of Fourier--Mukai kernels than 
\eqref{eqn-pi1-unit-in-fm-right-adjunction-unit-morphism}-\eqref{eqn-delta-unit-in-fm-right-adjunction-unit-morphism}. 
A pushdown of this more economical decomposition to $Z$
via $\tilde{\pi}_1$ could be expected to produce a better description of 
the defining morphism of $\mathcal{L}_W$ than the composition 
\eqref{eqn-shhom-pipiE-E-map}. For the sake of simplicity we choose 
to state this better description directly and prove directly that it is
isomorphic to the composition \eqref{eqn-shhom-pipiE-E-map}. An
interested reader could check that dualising the composition 
in \cite[Theorem 4.2]{AnnoLogvinenko-OnTakingTwistsOfFourierMukaiFunctors}
as described in Section \ref{section-adjunction-units-and-fm-transforms}
of this paper and applying $\tilde{\pi}_{1 *}$ would yield
the following:

\begin{prps}
\label{prps-alternative-description-of-L_W}

Let $W$ be a flat and perfect fibration in $X$ with proper fibres and let
$E = \iota_{W *} \mathcal{O}_W$ be the corresponding object in $D(Z
\times X)$.

Then the composition \eqref{eqn-shhom-pipiE-E-map}
\begin{align*}
\mathcal{O}_{Z} \xrightarrow{\id \rightarrow \pi_{Z *} \pi^*_{Z}}
\pi_{Z *} \mathcal{O}_{Z \times X} 
\xrightarrow{\id \rightarrow \rder\shhomm(E, E \otimes -)}
 \pi_{Z *}\rder\shhomm_{Z \times X}(E, E)
\xrightarrow{(\pi^*_X \pi_{X *} \rightarrow \id)^\text{opp}}
\pi_{Z *}\rder\shhomm_{Z \times X}(\pi^*_X \pi_{X *} E, E)
\end{align*}
is isomorphic to 
\begin{align}
\label{eqn-the-defining-morphism-of-L_W}
\mathcal{O}_Z \xrightarrow{\id \rightarrow \pi_* \pi^*} 
\pi_* \mathcal{O}_W 
\xrightarrow{(\xi^* \xi_* \rightarrow \id)^{\text{opp}}} 
\pi_* (\xi^* \xi_* \mathcal{O}_W)^\vee. 
\end{align}
In particular, the object $\mathcal{L}_W$ is isomorphic to the cone of 
\eqref{eqn-the-defining-morphism-of-L_W}.
\end{prps}
\begin{proof}
We have $\pi = \pi_Z \circ \iota_W$ and $\xi = \pi_X \circ \iota_W$. 
Decomposing $\id \rightarrow \pi_* \pi^*$ 
as $\id \rightarrow \pi_{Z *} \pi^*_Z 
\rightarrow \pi_{Z *} \iota_{W *} \iota^*_W \pi^*_Z$
we see that   
\eqref{eqn-the-defining-morphism-of-L_W} is the composition of 
$\id \rightarrow \pi_{Z *} \pi^*_Z$ with the image under 
$\pi_{Z *}$ of 
\begin{align*}
\mathcal{O}_{Z \times X} 
\xrightarrow{\id \rightarrow \iota_{W *} \iota^*_W}
\iota_{W *} \mathcal{O}_{W}
\xrightarrow{\gamma(\mathcal{O}_W)} 
\iota_{W *}
\rder\shhomm_{W}(\mathcal{O}_{W}, \mathcal{O}_{W}) 
\xrightarrow{(\xi^* \xi_* \rightarrow \id)^{\text{opp}}} 
\iota_{W *}
\rder\shhomm_{W}(\iota^*_W \pi^*_X \pi_{X *} \iota_{W *}
\mathcal{O}_{W}, \mathcal{O}_{W})
\end{align*}
where given an object $A$  
we denote by $\gamma(A)$ 
the adjunction unit $\id \rightarrow \rder\shhomm(A, A \otimes \text{-})$. 
On the other hand \eqref{eqn-shhom-pipiE-E-map} is the composition 
of 
$\id \rightarrow \pi_{Z *} \pi^*_Z$ with the image under $\pi_{Z *}$ of 
\begin{align*}
\mathcal{O}_{Z \times X} 
\xrightarrow{\gamma(\iota_{W *} \mathcal{O}_W)}
\rder\shhomm_{Z \times X}(\iota_{W *} \mathcal{O}_W, \iota_{W *}
\mathcal{O}_W) 
\xrightarrow{(\pi^*_X \pi_{X *} \rightarrow \id)^{\text{opp}}}  
\rder\shhomm_{Z \times X}(\pi^*_X \pi_{X *} \iota_{W *} \mathcal{O}_W,
\iota_{W *} \mathcal{O}_W). 
\end{align*}
We claim that these two compositions are identified by 
$$
\iota_{W *}
\rder\shhomm_{W}(\iota^*_W \pi^*_X \pi_{X *} \iota_{W *}
\mathcal{O}_{W}, \mathcal{O}_{W})
\xrightarrow{\alpha(\iota_W)}
\rder\shhomm_{Z \times X}(\pi^*_X \pi_{X *} \iota_{W *} \mathcal{O}_W,
\iota_{W *} \mathcal{O}_W)
$$
where $\alpha(\iota_W)$ is the natural bifunctorial isomorphism 
$\iota_{W *} \rder\shhomm(\iota^*_W\text{-},\text{-}) \rightarrow 
\rder\shhomm(\text{-},\iota_{W *}\text{-})$. 

Dualizing \cite[Prop. 4.1]{AnnoLogvinenko-OnTakingTwistsOfFourierMukaiFunctors}
under the relative duality $D_{\bullet/X}$ 
(where $X$ is in the notation of loc. cit.)
we see that the morphism 
$$ 
\mathcal{O}_{Z \times X} 
\xrightarrow{\gamma(\iota_{W *} \mathcal{O}_W)}
\rder\shhomm_{Z \times X}(\iota_{W *} \mathcal{O}_W, \iota_{W *} \mathcal{O}_W) 
$$
equals the morphism 
$$
\mathcal{O}_{Z \times X} 
\xrightarrow{\id \rightarrow \iota_{W *} \iota^*_W}
\iota_{W *} \mathcal{O}_W 
\xrightarrow{\gamma(\mathcal{O}_W)}
\iota_{W *} \rder\shhomm_{W}(\mathcal{O}_{W}, \mathcal{O}_{W}) 
\xrightarrow{\beta(\iota_W)}
\rder\shhomm_{Z \times X}(\iota_{W *} \mathcal{O}_W, \iota_{W *} \mathcal{O}_W) 
$$
where given a scheme map $f$ we denote by $\beta(f)$ the natural morphism 
$f_* \rder\shhomm(,) \rightarrow
\rder\shhomm(f_*,f_*)$. 

It remains to establish the commutativity of the diagram 
\begin{tiny}
\begin{align*}
\xymatrix{
\iota_{W *}
\rder\shhomm_{W}(\mathcal{O}_{W}, \mathcal{O}_{W})  
\ar[rr]^{(\iota^*_W \iota_{W*} \rightarrow \id)^{\text{opp}}}
\ar[d]^{\beta(\iota_W)}
& &
\iota_{W *} \rder\shhomm_{W}(\iota^*_W \iota_{W *} \mathcal{O}_W, \mathcal{O}_W) 
\ar[rr]^{(\pi^*_X \pi_{X *} \rightarrow \id)^{\text{opp}}}
& &
\iota_{W *} \rder\shhomm_{W}(\iota^*_W \pi^*_X \pi_{X *} \iota_{W *} \mathcal{O}_W, \mathcal{O}_W) 
\ar[d]_{\alpha(\iota_W)}
\\
\rder\shhomm_{Z \times X}(\iota_{W *} \mathcal{O}_{W}, \iota_{W *}
\mathcal{O}_{W})  
\ar[rrrr]_{(\pi^*_X \pi_{X *} \rightarrow \id)^{\text{opp}}}
& & & &
\rder\shhomm_{Z \times X}(\pi^*_X \pi_{X *} \iota_{W *} \mathcal{O}_W,
\iota_{W *} \mathcal{O}_W) 
}
\end{align*}
\end{tiny}
By the functoriality of $\alpha(\iota_W)$ it suffices to show that
the diagram 
\begin{align}
\label{eqn-alpha-beta-diagram}
\vcenter{
\xymatrix{
\iota_{W *}
\rder\shhomm_{W}(\mathcal{O}_{W}, \mathcal{O}_{W})  
\ar[rr]^{(\iota^*_W \iota_{W*} \rightarrow \id)^{\text{opp}}}
\ar[d]_{\beta(\iota_W)}
 & &
\iota_{W *} \rder\shhomm_{W}(\iota^*_W \iota_{W *} \mathcal{O}_{W}, \mathcal{O}_{W}) 
\ar[dll]^{\alpha(\iota_W)} \\
\rder\shhomm_{Z \times X}(\iota_{W *} \mathcal{O}_{W}, \iota_{W *}
\mathcal{O}_{W}) & & 
}
}
\end{align}
commutes. But the isomorphism $\alpha(\iota_W)$ was defined as the 
composition 
$$
\iota_{W *} \rder\shhomm(\iota^*_W\text{-},\text{-})
\xrightarrow{\beta(\iota_W)} 
\rder\shhomm(\iota_{W *} \iota^*_W \text{-},\iota_{W *}\text{-})
\xrightarrow{(\id \rightarrow \iota_{W*} \iota^*_W)^\text{opp}}
\rder\shhomm(\text{-},\iota_{W *}\text{-})
$$ and therefore we can re-write the diagram \eqref{eqn-alpha-beta-diagram} as 
\begin{align*}
\xymatrix{
\iota_{W *}
\rder\shhomm_{W}(\mathcal{O}_{W}, \mathcal{O}_{W})  
\ar[rr]^{(\iota^*_W \iota_{W*} \rightarrow \id)^{\text{opp}}}
\ar[d]_{\beta(\iota_W)}
 & &
\iota_{W *} \rder\shhomm_{W}(\iota^*_W \iota_{W *} \mathcal{O}_{W}, \mathcal{O}_{W}) 
\ar[d]^{\beta(\iota_W)} \\
\rder\shhomm_{Z \times X}(\iota_{W *} \mathcal{O}_{W}, \iota_{W *}
\mathcal{O}_{W}) & & 
\rder\shhomm_{Z \times X}(\iota_{W *} \iota^*_W \iota_{W *} \mathcal{O}_{W}, \iota_{W *} \mathcal{O}_{W})
\ar[ll]_{(\id \rightarrow \iota_{W*} \iota^*_W)^\text{opp}}
}
\end{align*}
By the functoriality of $\beta(\iota_W)$ it remains only to check that 
the diagram
\begin{align*}
\xymatrix{
\iota_{W *}
\rder\shhomm_{W}(\mathcal{O}_{W}, \mathcal{O}_{W})  
\ar[rr]_{\beta(\iota_W)}
\ar[d]_{\beta(\iota_W)}
 & &
\rder\shhomm_{Z \times X}(\iota_{W *} \mathcal{O}_{W}, \iota_{W *} \mathcal{O}_{W}) 
\ar[d]^{(\iota^*_W \iota_{W*} \rightarrow \id)^{\text{opp}}}
\\
\rder\shhomm_{Z \times X}(\iota_{W *} \mathcal{O}_{W}, \iota_{W *} \mathcal{O}_{W}) 
& & 
\rder\shhomm_{Z \times X}(\iota_{W *} \iota^*_W \iota_{W *} \mathcal{O}_{W}, \iota_{W *} \mathcal{O}_{W})
\ar[ll]_{(\id \rightarrow \iota_{W*} \iota^*_W)^\text{opp}}
}
\end{align*}
commutes, which follows from 
\begin{align*}
\iota_{W*} \mathcal{O}_W
\xrightarrow{\id \rightarrow \iota_{W*} \iota^*_W}
\iota_{W *} \iota^*_W \iota_{W*} \mathcal{O}_W
\xrightarrow{\iota^*_W \iota_{W*} \rightarrow \id} 
\iota_{W*} \mathcal{O}_W
\end{align*}
being an identity morphism. 
\end{proof}

Next we give an analogue of Proposition 
\ref{prps-tfae-left-cotwist-is-an-equivalence}:

\begin{prps}
\label{prps-tfae-left-cotwist-is-an-equivalence-for-fibrations}
Let $W$ be a flat and perfect fibration in $X$ over $Z$ with proper
fibres. The following are equivalent:
\begin{enumerate}
\item
\label{item-sphericity-condition-on-exts-for-fibres-for-fibrations}
There exists $d \in \mathbb{Z}$ such that for every closed point 
$p \in Z$ we have 
$$ \rder\homm_{D(X)}(\xi_{*} \mathcal{O}_W, \xi_{p *}
\mathcal{O}_{W_p}) = k \oplus k[d].$$

\item \label{item-invertibility-of-L_W-for-fibrations}
We have $\mathcal{L}_W \simeq L[d]$ for some $L \in \picr Z$ and 
$d \in Z$. 
\item \label{item-the-left-co-twist-is-an-equivalence-for-fibrations}
The co-twist $F_E$ is an autoequivalence of $D(Z)$. 
\end{enumerate}
When the conditions above are satisfied $F_E \simeq (\text{-}) \otimes 
\mathcal{L}_W[-1]$ and the integers $d$ in
$(\ref{item-sphericity-condition-on-exts-for-fibres-for-fibrations})$  
are
$(\ref{item-invertibility-of-L_W-for-fibrations})$ are equal.
\end{prps}

\begin{proof}

Since $Z$ is connected
any invertible object of $D(Z)$ is
a shift of line bundle \cite[Theorem 1.5.2]{AvramovIyengarLipman-ReflexivityAndRigidityForComplexesIISchemes}.
Thus our conditions \eqref{item-invertibility-of-L_W-for-fibrations}
and \eqref{item-the-left-co-twist-is-an-equivalence-for-fibrations}
are equivalent to conditions \eqref{item-invertibility-of-L_E} and \eqref{item-the-left-co-twist-is-an-equivalence}
of Proposition \ref{prps-tfae-left-cotwist-is-an-equivalence}.

As $\xi = \pi_X \circ \iota_W$ we 
have $\pi_{X *} E \simeq \pi_{X *} \iota_{W *} \mathcal{O}_W = \xi_{*}
\mathcal{O}_W$. By Lemma 
\ref{lemma-fibres-of-O_W-in-Z-x-X-are-precisely-O_Wp} the categorical
fibre $E_p$ is $\xi_{p *} \mathcal{O}_{W_p}$. Under these identifications
the morphism $\pi_{X *} E
\xrightarrow{ \eqref{eqn-nat-morphism-pi_X*-E-to-E_p} } E_p$
is readily seen to be the sheaf restriction map
 $\xi_{*} \mathcal{O}_W \rightarrow \xi_{*} \mathcal{O}_{W_p}$ and thus
 non-zero for every $p \in Z$. Therefore our condition 
(\ref{item-sphericity-condition-on-exts-for-fibres-for-fibrations}) is
equivalent to condition
(\ref{item-sphericity-condition-on-exts-for-fibres})
of Proposition \ref{prps-tfae-left-cotwist-is-an-equivalence} with an extra 
assumption that the integer $d_p$ is the same for all $p \in Z$. 

Now the assertion of this Proposition can be seen to follow directly from those
of Proposition \ref{prps-tfae-left-cotwist-is-an-equivalence}. 
\end{proof}

We could similarly re-state 
Theorem \ref{theorem-sphericity-for-orthogonal-objects-of-ZxX}. 
However under a mild non-degeneracy assumption on $W$ we can apply 
the results of Section \ref{section-the-canonical-morphism-alpha} 
to make a stronger and more geometric statement. Since $Z$ and $X$ 
are abstract varieties they are generically non-singular. Hence 
the Gorenstein locus of $Z \times X$ is certaily dense in $Z \times X$.
Our non-degeneracy assumption is that the graph of $W$ doesn't lie
 completely outside this locus:
\begin{theorem}
\label{theorem-sphericity-for-perfect-flat-fibrations}
Let $W$ be a flat and perfect fibration in $X$ with proper fibres. 
Then $W$ is spherical if:
\begin{enumerate}
\item 
\label{item-exts-from-W-into-W_p-are-two-dimensional}
For any closed $p \in Z$ we have 
\begin{align*}
\rder\homm_X(\xi_* \mathcal{O}_W, \xi_{p *} \mathcal{O}_{W_p}) = k
\oplus k[- (\dim X - \dim Z)].
\end{align*}
\item
\label{item-relative-dualizing-complexes-are-isomorphic} 
There exists an isomorphism
\begin{align*}
\iota_{W *} \xi^!(\mathcal{O}_X) \xrightarrow{\sim} 
\iota_{W *} \pi^!(\mathcal{L}_W).
\end{align*}
\end{enumerate}
If the graph of $W$ in $Z \times X$ doesn't lie outside 
the Gorenstein locus the reverse implication also holds. 
\end{theorem}
\begin{proof}
We have the following natural isomorphisms:
\begin{align*}
\iota_{W *} \xi^!(\mathcal{O}_X)
\xrightarrow{\sim}
\iota_{W*} \rder\shhomm_{Z \times X}\bigl(\mathcal{O}_W, 
\iota^!_W \pi^!_X(\mathcal{O}_X)\bigr)
\xrightarrow{\sim}
\rder\shhomm_{Z \times X}\bigl(\iota_{W *} \mathcal{O}_W, 
\pi^!_X(\mathcal{O}_X)\bigr) 
\xrightarrow{\sim}
E^\vee \otimes \pi^!_X(\mathcal{O}_X) 
\end{align*}
where the second isomorphism is due to the sheafified Grothendieck
duality and the third is due to $E = \iota_{W *} \mathcal{O}_W$
being perfect. Similarly we obtain
$\iota_{W *} \pi^{!}(\mathcal{L}_W) \simeq 
E^\vee \otimes \pi^!_Z(\mathcal{L}_W)$.  
Therefore $(\ref{item-relative-dualizing-complexes-are-isomorphic})$ is
equivalent to there existing an isomorphism 
$$E^\vee \otimes \pi^!_X(\mathcal{O}_X) \simeq 
E^\vee \otimes \pi^!_Z(\mathcal{L}_W).$$

Suppose that $(\ref{item-exts-from-W-into-W_p-are-two-dimensional})$
and
$(\ref{item-relative-dualizing-complexes-are-isomorphic})$
hold. By Proposition 
\ref{prps-tfae-left-cotwist-is-an-equivalence} the assumption
$(\ref{item-exts-from-W-into-W_p-are-two-dimensional})$
implies that the co-twist $F_E$ is an autoequivalence and 
$\mathcal{L}_W \simeq L[-(\dim X - \dim Z) ]$
for some $L \in \picr(Z)$. Since $\dim X - \dim Z > 0$ it follows
from Proposition \ref{prps-getting-rid-of-canonical-map-alpha}
that the existence of any isomorphism 
$E^\vee \otimes \pi^!_X(\mathcal{O}_X) \simeq 
E^\vee \otimes \pi^!_Z(\mathcal{L}_E)$ implies that the canonical 
morphism $\alpha$ of Definition \ref{defn-canonical morphism-alpha-E} 
is an isomorphism. Thus $F_E$ is an autoequivalence and $\alpha$ is 
an isomorphism, and so $W$ is a spherical fibration.

Conversely, suppose that $W$ is a spherical fibration whose graph 
doesn't lie  outside the Gorenstein locus of $Z \times X$.
Then  the co-twist $F_E$ is an autoequivalence and so by  
Proposition \ref{prps-tfae-left-cotwist-is-an-equivalence} we have
$\mathcal{L}_E \simeq L[d]$ for some $L \in \picr Z$ and $d \in \mathbb{Z}$. 
By the non-degeneracy assumption there exists a point $p \in W$ such that $\xi(p)$ 
is Gorenstein in $X$ and $\pi(p)$ is Gorenstein in $Z$. 
By Proposition \ref{prps-the-shift-of-L_E-is-the-difference-in-dimensions}
we then have $d = -(\dim X - \dim Z)$. Applying 
Proposition \ref{prps-tfae-left-cotwist-is-an-equivalence} again
yields the assertion
$(\ref{item-exts-from-W-into-W_p-are-two-dimensional})$. On the other hand, 
since $W$ is spherical the canonical morphism $\alpha$ is an
isomorphism $E^\vee \otimes \pi^!_X(\mathcal{O}_X) \xrightarrow{\sim} 
E^\vee \otimes \pi^!_Z(\mathcal{L}_W)$ whence the assertion 
$(\ref{item-relative-dualizing-complexes-are-isomorphic})$. 
\end{proof}

Recall the notion of a Gorenstein map, cf. \S 2.4
of \cite{AvramovIyengarLipman-ReflexivityAndRigidityForComplexesIISchemes}
or \cite{AvramovFoxby-GorensteinLocalHomomorphisms} for the
local picture. A scheme map $f\colon S \rightarrow T$ is
called Gorenstein if it is perfect and if $f^!(\mathcal{O}_T)$ is 
an invertible object of $D(S)$. If $S$ is connected
$f^!(\mathcal{O}_T)$ is a shift of some line bundle in $\picr S$.  
We call this line bundle the \em relative dualizing sheaf \rm 
and denote it by $\omega_{S/T}$. For any Gorenstein scheme 
$S$ over $k$ the \em (global) dualizing sheaf \rm of $S$
is the relative dualizing sheaf of $S \rightarrow \spec k$ and 
we denote it by $\omega_S$. If $S$  is  smooth then $\omega_S$ is the canonical bundle.  
 
In our case the map $\pi\colon W \rightarrow Z$ 
is faithfully flat and thus Gorenstein if and only if its fibres 
are Gorenstein schemes
\cite[Prop, 2.5.10]{AvramovIyengarLipman-ReflexivityAndRigidityForComplexesIISchemes}. 
On the other hand, the map $\xi$ is the composition 
$$ W \xrightarrow{\iota_W} Z \times X \xrightarrow{\pi_X} X. $$
The closed immersion $\iota_W$ is perfect by the assumption
that $\iota_{W *} \mathcal{O}_W$ is perfect 
\cite[Prop. 4.4]{IllusieConditionsDeFinitudeRelative}. 
Hence $\xi$ is perfect as it is 
a composition of two perfect maps. Thus $\xi$ is Gorenstein 
if and only if $\xi^!(\mathcal{O}_X)$ is invertible. 

If either $\pi$ or $\xi$ are Gorenstein
we can re-state the second part of
the Theorem \ref{theorem-sphericity-for-perfect-flat-fibrations}
in terms of the line bundles involved and 
get rid of the non-degeneracy assumption on $W$:

\begin{prps}
\label{prps-sphericity-for-Gorenstein-flat-fibrations}
Let $W$ be a flat and perfect fibration in $X$ over $Z$ with proper fibres
and assume that either the immersion $\xi\colon W \hookrightarrow X$
or the fibration $\pi\colon W \rightarrow Z$ is Gorenstein. Then 
$W$ is spherical if and only if:
\begin{enumerate}
\item
\label{item-fibrewise-k+kd-in-sphericity-condition-for-gorenstein}
For any closed $p \in Z$ we have 
\begin{align*}
\rder\homm_X(\xi_* \mathcal{O}_W, \xi_{p *} \mathcal{O}_{W_p}) = k
\oplus k[-(\dim X - \dim Z)] 
\end{align*}
By Proposition \ref{prps-tfae-left-cotwist-is-an-equivalence-for-fibrations}
this implies that $\mathcal{L}_W = L [-(\dim X - \dim Z)]$ for 
some $L \in \picr Z$. 
\item 
\label{item-dualizing-sheaves-isomorphism-in-sphericity-condition-for-gorenstein}
Both $\xi$ and $\pi$ are Gorenstein and we have in $\picr(W)$ an isomorphism 
\begin{align*}
\omega_{W/X} \simeq \pi^* L \otimes \omega_{W/Z}.
\end{align*} 
\end{enumerate}
\end{prps}

\begin{proof}

\em`If': \rm 
We have
\begin{align*}
& \xi^!(\mathcal{O}_X) \simeq \omega_{W/X}[- (\dim X - \dim W)] 
& \pi^!(\mathcal{O}_X) \simeq \omega_{W/Z}[\dim W - \dim Z] 
\end{align*}
and therefore the condition 
$(\ref{item-dualizing-sheaves-isomorphism-in-sphericity-condition-for-gorenstein})$
implies $\xi^!(\mathcal{O}_X) \simeq \pi^!(\mathcal{L}_W)$. 
Therefore $W$ is spherical by 
Theorem \ref{theorem-sphericity-for-perfect-flat-fibrations}. 

\em `Only if': \rm Suppose $W$ is spherical. Arguing as in the proof of 
Theorem \ref{theorem-sphericity-for-perfect-flat-fibrations} shows that $\mathcal{L}_W$ is invertible and we have 
an isomorphism 
\begin{align}
\label{eqn-immersed-iso-of-relative-dualizing-sheaves}
\iota_{W *} \xi^!(\mathcal{O}_X) \xrightarrow{\sim} 
\iota_{W *} \pi^!(\mathcal{L}_W).
\end{align}
Our assumptions imply that one of $\xi^!(\mathcal{O}_X)$ 
or $\pi^!(\mathcal{L}_W)$ is invertible. Thus \eqref{eqn-immersed-iso-of-relative-dualizing-sheaves}
is an isomorphism of (shifted) coherent sheaves. Since $\iota_{W *}$ 
is a closed immersion it restricts to a fully faithful functor 
$\cohcat(W) \rightarrow \cohcat(Z \times X)$. Hence  
isomorphism \eqref{eqn-immersed-iso-of-relative-dualizing-sheaves}
lifts to an isomorphism 
\begin{align}
\label{eqn-xi!-equals-pi!L_E}
\xi^!(\mathcal{O}_X) \xrightarrow{\sim} 
\pi^!(\mathcal{L}_W).
\end{align}
Therefore $\xi^!(\mathcal{O}_X)$ and $\pi^!(\mathcal{L}_W)$ are both 
invertible, i.e. $\pi$ and $\xi$ are both Gorenstein. 

Since $\mathcal{L}_W$ is invertible it is of form $L[d]$ for some 
$L \in \picr Z$ and $d \in \mathbb{Z}$. We can re-write 
\eqref{eqn-xi!-equals-pi!L_E} as 
$$
\omega_{W/X}[-(\dim X - \dim W)] \simeq
\pi^* L[d] \otimes \omega_{W/Z}[\dim W - \dim Z] 
$$
whence $d = - (\dim X - \dim Z)$ and the isomorphism 
$(\ref{item-dualizing-sheaves-isomorphism-in-sphericity-condition-for-gorenstein})$.
Finally, since $\mathcal{L}_W \simeq L[-(\dim X - \dim Z)]$
we can apply Proposition 
\ref{prps-tfae-left-cotwist-is-an-equivalence-for-fibrations} 
to obtain the assertion
$(\ref{item-fibrewise-k+kd-in-sphericity-condition-for-gorenstein})$. 
\end{proof}

\begin{cor}
\label{cor-if-one-is-Gorenstein-both-are}
Let $W$ be a spherical fibration in $X$ over $Z$. Then $\xi\colon 
W \hookrightarrow X$ is a Gorenstein immersion if and only 
if all the fibres of $\pi\colon W \rightarrow Z$ are Gorenstein 
schemes. 
\end{cor}

\subsection{Regularly immersed fibrations}

One class of Gorenstein maps is that of regular immersions, 
cf. \cite{Grothendieck-EGA-IV-4}, \S16 and \S19 and 
\cite{Berthelot-ImmersionsRegulieresEtCalculDuKDUnSchemaEclate}. 
A closed immersion $\iota\colon Y \hookrightarrow X$ 
of two schemes is called \em regular \rm if the ideal sheaf $\mathcal{I}_Y$ 
of $Y$ in $X$ is locally generated by a regular sequence.
It follows that locally on $X$ the Koszul complex of $Y$ is 
a resolution of the sheaf $\iota_* \mathcal{O}_Y$ by free sheaves.
In particular, the co-normal sheaf $\mathcal{I}_Y / \mathcal{I}^2_Y$ 
is a locally free sheaf on $Y$ whose rank $c$ is 
the codimension of $Y$ in $X$. We denote by $\mathcal{N}_{Y/X}$ 
its dual $(\mathcal{I}_Y / \mathcal{I}^2_Y)^\vee$, the normal sheaf 
of $Y$ in $X$. 

It follows by \S III.7 of \cite{Hartshorne-Residues-and-Duality} that 
$$ \iota^!(\mathcal{O}_X) = \wedge^c \mathcal{N}_{Y/X}[-c] $$ 
i.e. the relative dualizing sheaf $\omega_{Y/X}$ is the 
line bundle $\wedge^c \mathcal{N}_{Y/X}$. 
By \cite[Prop. 2.5]{Berthelot-ImmersionsRegulieresEtCalculDuKDUnSchemaEclate}
the cohomology sheaves of $\iota^* \iota_* \mathcal{O}_Y$ are
$$ \mathcal{H}^{-i}(\iota^* \iota_* \mathcal{O}_Y) = 
\wedge^i \mathcal{N}^\vee_{Y/X} \quad \quad \forall\; i \in \mathbb{Z}. $$
Let $A$ be any object of $D(Y)$. By projection formula we have
$$ \iota_*(\iota^*\iota_* A) \simeq \iota_* A \otimes \iota_* \mathcal{O}_Y
\simeq \iota_*(A \otimes \iota^* \iota_* \mathcal{O}_Y). $$
As $\iota_*$ is exact and fully faithful on the level of abelian 
categories of coherent sheaves it follows that 
the cohomology sheaves of $\iota^* \iota_* A$ 
are isomorphic to those of $A \otimes \iota^* \iota_* \mathcal{O}_Y$. 

One can ask when does $\iota^* \iota_* \mathcal{O}_Y$ 
split up as the direct sum of its cohomology sheaves:
\begin{align}
\label{eqn-Arinkin-Caldararu-property}
\iota^* \iota_* \mathcal{O}_Y \xrightarrow{\sim} \bigoplus_i \wedge^i
\mathcal{N}^\vee_{Y/X}[i]
\end{align}
This is true when a global Koszul resolution of $Y$ in $X$
exists, i.e. when $Y$ is carved out in $X$ by a section of a vector bundle. 
For smooth $X$ a more general answer was provided by Arinkin and Caldararu in
\cite{ArinkinCaldararu-WhenIsTheSelfIntersectionOfASubvarietyAFibration}:
$\iota^* \iota_* \mathcal{O}_Y$ is isomorphic to 
$\bigoplus_i \wedge^i \mathcal{N}^\vee_{Y/X}[i]$ if and only if 
the normal sheaf $\mathcal{N}_{Y/X}$ extends to the first
infinitesimal neighborhood of $Y$ in $X$. The examples of when 
this holds include: when $Y$ is carved out
by a section of a vector bundle, when the immersion 
$\iota\colon Y \hookrightarrow X$ can be split 
and when $Y$ is the fixed locus of a finite group action on $X$. 

For arbitrary schemes we make the the following definition:

\begin{defn}
Let $Y$ and $X$ be a pair of schemes and let 
$\iota\colon Y \rightarrow X$ be a regular immersion. We say 
that $\iota$ is an \em Arinkin-Caldararu immersion \rm 
if $\iota^* \iota_* \mathcal{O}_Y$ is isomorphic to 
$\bigoplus_i \wedge^i \mathcal{N}_{Y/X}^\vee[i]$ in $D(X)$. 
\end{defn}

Going back to our setup,  we say that a fibration $W$ in $X$ over $Z$ is 
\em regularly immersed \rm if $\xi\colon W \hookrightarrow X$ 
is a regular immersion. Knowing the cohomology sheaves of 
$\xi^* \xi_* \mathcal{O}_W$ allows us to reduce the condition 
$$ \rder\homm_{X}(\xi_* \mathcal{O}_{W}, \xi_{p *} \mathcal{O}_{W_p}) = 
k \oplus k[-(\dim X - \dim Z)] $$
via a spectral sequence argument to a statement on the vanishing of 
the sheaf cohomologies of  $\wedge^i \mathcal{N}$ on $W_p$. If, moreover,
$\xi^* \xi_* \mathcal{O}_W$ breaks up as a sum of its cohomologies then
there is no need for the spectral sequence argument and we also
obtain the converse implication: 
\begin{theorem}
\label{theorem-sphericity-for-regular-immersions}
Let $W$ be a regularly immersed flat and perfect fibration in $X$ over
$Z$ with proper fibres. Let $c$ be the codimension of $W$ in $X$, 
let $d$ be the dimension of the fibres of $W$ and let $\mathcal{N} = \mathcal{N}_{W/X}$.

Then $W$ is spherical if for any closed point $p \in Z$
the fibre $W_p$ is a connected Gorenstein scheme and
\begin{enumerate}
\item 
\label{eqn-cohomological-vanishing-of-the-normal-sheaf}
$H^i_{W_p}(\wedge^j \mathcal{N}|_{W_p}) = 0$  unless $i = j = 0$
or $i = d \;,\; j = c$.
\item 
\label{eqn-restriction-of-the-normal-sheaf}
$(\omega_{W/X})|_{W_p} \simeq \omega_{W_p}$. 
\end{enumerate}
Conversely, if $W$ is spherical then each fibre $W_p$ is
a connected Gorenstein scheme and
(\ref{eqn-restriction-of-the-normal-sheaf}) holds. 
If, moreover, $\xi$ is an Arinkin-Caldararu immersion then 
$(\ref{eqn-cohomological-vanishing-of-the-normal-sheaf})$ also holds. 
\end{theorem}
 
The following lemma is a global version of the fibrewise conditions of  
Theorem \ref{theorem-sphericity-for-regular-immersions}:
\begin{lemma}
\label{lemma-fibrewise-to-global-for-regular-immersions}
Let $W$ be a regularly immersed flat and perfect fibration in $X$ over
$Z$ with proper fibres. Assume that for any closed point $p \in Z$
the fibre $W_p$ is a connected Gorenstein scheme.

Then having for every closed point $p \in Z$ 
\begin{align}
\label{eqn-fibrewise-version-of-conditions-in-theorem-for-regular-immersions}
H^i_{W_p}(\wedge^j \mathcal{N}|_{W_p}) = 0  \text{ unless } i = j = 0
\text{ or } i = d \;,\; j = c \\
\notag
\omega_{W/X}|_{W_p} = \omega_{W_p}
\end{align}
is equivalent to having
\begin{align}
\label{eqn-global-version-of-conditions-in-theorem-for-regular-immersions}
\pi_* \mathcal{O}_W = \mathcal{O}_Z, \quad 
\pi_* \wedge^j\mathcal{N} = 0  \text{ for all } \; 0 < j < c, \quad 
\pi_* \omega_{W/X} = L[d], \\ 
\notag
\omega_{W/X} = \pi^*L \otimes \omega_{W/Z} 
\end{align}
for some $L \in \picr Z$. In particular,  
\eqref{eqn-fibrewise-version-of-conditions-in-theorem-for-regular-immersions}
implies that $H^0_{W_p}(\mathcal{O}_{W_p}) \simeq 
H^d_{W_p}(\omega_{W/X}|_{W_p}) \simeq k$. 
\end{lemma}
\begin{proof}
By flat base change around the square
\begin{align*}
\xymatrix{
W_p \; \ar@{^{(}->}[r]^{\iota_{W_p}} \ar[d]_{\pi_k} & 
W \ar[d]^{\pi} \\
\spec k \; \ar@{^{(}->}[r]_{\iota_p} &
Z 
}
\end{align*}
we have a functorial isomorphism 
$\iota^*_p \pi_* \simeq \pi_{k *} \iota^*_{W_p}$. 
Since $H^i_{W_p}(\wedge^j \mathcal{N}|_{W_p})$ is 
the $i$-th cohomology of $\pi_{k *} \iota^*_{W_p} (\wedge^j\mathcal{N})$ restricting 
\eqref{eqn-global-version-of-conditions-in-theorem-for-regular-immersions}
to any closed $p \in Z$ by $\iota_p^*$ gives
\eqref{eqn-fibrewise-version-of-conditions-in-theorem-for-regular-immersions}. 

Conversely, assume that 
\eqref{eqn-fibrewise-version-of-conditions-in-theorem-for-regular-immersions}
holds for every closed $p \in Z$.
By the Grothendieck duality for $W_p$ we have
$$ H_{W_p}^d(\omega_{W/X}|_{W_p}) \simeq H_{W_p}^d(\omega_{W_p}) \simeq 
H_{W_p}^0(\mathcal{O}_{W_p}) $$
and since $W_p$ is proper and connected we have 
$H_{W_p}^0(\mathcal{O}_{W_p}) \simeq k$.
Thus by the above base change we have for every closed $p \in Z$ 
\begin{align*}
&\iota^*_p \pi_*\; \mathcal{O}_Z \simeq k \\
&\iota^*_p \pi_*\; \wedge^j \mathcal{N}|_{W_p} = 0 \quad\text{ for all } 0 < j < c\\
&\iota^*_p \pi_*\; \omega_{W/X} \simeq k[-d].
\end{align*}
Therefore $\pi_* \wedge^j \mathcal{N}$ vanishes for $0 < j < c$, while
$\pi_* \mathcal{O}_W \simeq L'$ and $\pi_* \omega_{W/X}[d] \simeq L$ for
some $L', L \in \picr Z$. But then $L' \simeq \mathcal{O}_Z$ 
since the adjunction unit 
$\mathcal{O}_Z \rightarrow  \pi_* \pi^* \mathcal{O}_Z$ gives
a nowhere vanishing morphism $\mathcal{O}_Z \rightarrow  L'$ of line
bundles. This is because the restriction of 
the adjunction unit $\mathcal{O}_Z \rightarrow \pi_* \pi^* \mathcal{O}_Z$
to any $p \in Z$ is the adjunction unit $k \rightarrow \pi_{k *}
\pi^*_k \; k$ 
which certainly doesn't vanish. 

Similarly, by the sheafified Grothendieck duality 
$$ L \simeq \pi_* \omega_{W/X}[d] 
\simeq \pi_* \rder\shhomm(\omega^{-1}_{W/X} \otimes \omega_{W/Z},
\omega_{W/Z}[d]) 
\simeq \left(\pi_*(\omega^{-1}_{W/X} \otimes
\omega_{W/Z})\right)^\vee.
$$
Therefore the adjunction co-unit 
$\pi^* \pi_* (\omega^{-1}_{W/X} \otimes \omega_{W/Z}) \rightarrow 
(\omega^{-1}_{W/X} \otimes \omega_{W/Z})$ gives a nowhere vanishing
line bundle morphism $\pi^* L^{\vee} \rightarrow \omega^{-1}_{W/X} \otimes
\omega_{W/Z}$, whence the final assertion that $\omega_{W/X} \simeq
\pi^* L \otimes \omega_{W/Z}$.
\end{proof}

\begin{proof}[Proof of Theorem
\ref{theorem-sphericity-for-regular-immersions}]

\em `If' direction: \rm
Since $\xi_p$ is the composition 
$ W_p \xrightarrow{\iota_{W_p}} W \xrightarrow{\xi} X $
we have by adjunction 
$$
\rder\homm_X(\xi_* \mathcal{O}_W, \xi_{p *} \mathcal{O}_{W_p})
\simeq
\rder\homm_{W}(\xi^* \xi_* \mathcal{O}_W, \iota_{W_p *} \mathcal{O}_{W_p}).$$ 

Consider the standard spectral sequence
$$ E^{i,j}_2 = \ext^i_{W}(\mathcal{H}^{-j}(\xi^* \xi_*
\mathcal{O}_W), \iota_{W_p *} \mathcal{O}_{W_p})
\quad \Rightarrow \quad 
E^{i+j}_\infty = \homm^{i+j}_{D(W)}(\xi^* \xi_* \mathcal{O}_W, \iota_{W_p *}
\mathcal{O}_{W_p}).$$
Since for any $j \in \mathbb{Z}$ we have 
$\mathcal{H}^{-j}(\xi^* \xi_* \mathcal{O}_W) = \wedge^j
\mathcal{N}^\vee$ it follows by adjunction that
$$ E^{i,j}_2 \simeq 
\ext^i_{W}(\wedge^j \mathcal{N}^\vee, \iota_{W_p *} \mathcal{O}_{W_p}) 
\simeq \ext^i_{W_p}(\wedge^j \mathcal{N}^\vee|_{W_p}, \mathcal{O}_{W_p}) 
\simeq H^i_{W_p}(\wedge^j \mathcal{N}|_{W_p}).$$

Since the fibers of $W$ are proper and connected
$H_{W_p}^0(\mathcal{O}_{W_p}) \simeq k$. Moreover 
$$H_{W_p}^{d}(\omega_{W/X}|_{W_p}) \simeq H^d_{W_p}(\omega_{W_p}) \simeq 
H^0_{W_p}(\mathcal{O}_{W_p}) \simeq k.$$
by the assumption 
$(\ref{eqn-restriction-of-the-normal-sheaf})$ and the Grothendieck
duality.

Thus by assumption
$(\ref{eqn-cohomological-vanishing-of-the-normal-sheaf})$ 
and by Lemma \ref{lemma-fibrewise-to-global-for-regular-immersions}
all $E^{i,j}_2$ are zero except for 
$$ E_2^{0,0} = H_{W_p}^0(\mathcal{O}_{W_p}) \simeq k \text{ and }
E_2^{d, c} = H_{W_p}^{d}(\omega_{W/X}|_{W_p}) \simeq k. $$ 
Since $d + c = \dim X - \dim Z \neq 0$ 
the convergence of the spectral sequence implies that 
$$ \rder\homm_{X}(\xi_* \mathcal{O}_W, \xi_{p *} \mathcal{O}_{W_p}) 
\simeq k \oplus k [ -(\dim X - \dim Z) ]. $$

By Prop.~\ref{prps-tfae-left-cotwist-is-an-equivalence-for-fibrations}
we have $\mathcal{L}_W \simeq L[c+d]$ for some $L \in \picr Z$. 
By Prop.~\ref{prps-alternative-description-of-L_W}
we have an exact triangle $$ \mathcal{O}_Z \rightarrow 
\pi_{*} \rder\shhomm(\xi^* \xi_* \mathcal{O}_W, \mathcal{O}_W)
\rightarrow L[-(c+d)]. $$ Since $c + d > 0$ it follows that 
the $(c+d)$-th cohomology sheaf of the complex 
$\pi_{*} \rder\shhomm(\xi^* \xi_* \mathcal{O}_W, \mathcal{O}_W)$
is isomorphic to $L$. 
On the other hand, computing this cohomology sheaf via a 
spectral sequence similar to the one above yields 
$\pi_* \omega_{W/X}[d]$. Thus $L \simeq \pi_* \omega_{W/X}[d]$
which implies by Lemma \ref{lemma-fibrewise-to-global-for-regular-immersions}
that $\omega_{W/X} \simeq \pi^* L \otimes \omega_{W/Z}$. 
By Prop.~\ref{prps-sphericity-for-Gorenstein-flat-fibrations}
we conclude that $W$ is spherical. 

\em `Only If' direction: \rm

Conversely, suppose $W$ is spherical. 
By Proposition \ref{prps-sphericity-for-Gorenstein-flat-fibrations} 
the fibres of $\pi$ are Gorenstein schemes and we have
$$ \rder\homm_{X}(\xi_* \mathcal{O}_W, \xi_{p *} \mathcal{O}_{W_p}) 
\simeq k \oplus k [ -(\dim X - \dim Z) ] $$
for each fibre $W_p$. The same spectral sequence as before shows that 
the $0$-th cohomology of the complex 
$\rder\homm^i_{X}(\xi_* \mathcal{O}_W, \xi_{p *} \mathcal{O}_{W_p})$
is isomorphic to $H^0_{W_p}(\mathcal{O}_{W_p})$. Therefore 
$H^0_{W_p}(\mathcal{O}_{W_p}) = k$ and so the fibers $W_p$ are connected. 
By Prop.~\ref{prps-sphericity-for-Gorenstein-flat-fibrations} 
we have $\omega_{W/X} \simeq \pi^* L \otimes \omega_{W/Z}$ for some 
$L \in \picr Z$.  Restricting this to every fiber gives the assertion 
$(\ref{eqn-restriction-of-the-normal-sheaf})$. 

Finally, suppose that $\xi$ is Arinkin-Caldararu. Then $\xi^* \xi_*
\mathcal{O}_{W} \simeq \bigoplus_i \wedge^i \mathcal{N}^\vee[-i]$, so
\begin{align*}
\rder\homm_{X}(\xi_* \mathcal{O}_{W}, \xi_{p *} \mathcal{O}_{W_p}) \simeq 
\rder\homm_{W}(\xi^* \xi_* \mathcal{O}_{W}, 
\iota_{W_p *} \mathcal{O}_{W_p}) \simeq \\
\simeq \bigoplus_i \rder\homm_{W}(\wedge^i \mathcal{N}^\vee, 
\iota_{W_p *} \mathcal{O}_{W_p})[i] \simeq 
\bigoplus_i \rder\homm_{W_p}(\mathcal{O}_{W_p}, 
\wedge^i \mathcal{N})[i]
\end{align*} 
and we see that the assertion 
$(\ref{eqn-cohomological-vanishing-of-the-normal-sheaf})$ 
is equivalent to 
$$ \rder\homm_{X}(\xi_* \mathcal{O}_{W}, \xi_{p *} \mathcal{O}_{W_p})
\simeq k \oplus k[-(\dim X - \dim Z)]. $$
\end{proof}

\appendix
\section{An example}

It is well-known
that the derived category $D(T^* Fl_n)$
where $Fl_n$ is the full flag variety for  some Lie algebra
${\mathfrak g}$ carries an action of the affine
braid group
\cite{KhovanovThomas-BraidCobordismsTriangulatedCategoriesAndFlagVarieties},
\cite{Bezrukavnikov-NonCommutativeCounterpartsOfTheSpringerResolution}.
It is shown in \cite{KhovanovThomas-BraidCobordismsTriangulatedCategoriesAndFlagVarieties}
that the action of the usual braid group $Br_n$ is by spherical twists
$T_i$, $i=1,\ldots, n-1$ in spherical functors
$S_i\colon D(T^*\mathcal{P}_i) \to D(T^*Fl_n)$, where
$\mathcal{P}_i$ are the partial flag varieties with the space of
dimension $i$ missing from the flag.
The functor $S_i$ is obtained as the composition
$\iota_*\pi^*$, where $\iota\colon D_i \hookrightarrow  T^*Fl_n$
is the embedding of
the divisor $D_i=Fl_n\times_{\mathcal{P}_i} T^*\mathcal{P}_i$,
and $\pi\colon D_i\to T^* \mathcal{P}_i$ is a $\mathbb{P}^1$-bundle.
The Fourier--Mukai kernel of $S_i$ is an example of a spherical fibration, being
the structure sheaf of $D_i\subset T^*Fl_n\times T^*\mathcal{P}_i$ where
$D_i$ embeds into $T^*Fl_n$ and is fibered over $T^*\mathcal{P}_i$.

Recall that the usual braid group is generated by $n-1$ ``crossings''
$t_1, \ldots, t_{n-1}$,
with the relations $t_it_{i+1}t_i=t_{i+1}t_it_{i+1}$.
The affine braid group is generated by the same $t_1,\ldots, t_{n-1}$, plus
a ``rotation'' generator $r$ (if the affine braid group is viewed as the group of
braids in an annulus, this generator shifts strands, say, counterclockwise).
The relations then are $rt_ir^{-1}=t_{i+1}$ and $r^2t_nr^{-2}=t_1$.
One can add one more "crossing" $r^{-1}t_1r=t_0=t_n=rt_{n-1}r^{-1}$,
keeping the relations
$t_it_{i+1}t_i=t_{i+1}t_it_{i+1}$.
In the above affine braid group action on $D(T^* Fl_n)$ the action of
the functor corresponding to $t_n$
is not known to have an interpretation as a spherical twist. This can be mended
in a specific case, and the relative spherical object that induces
the twist will not be a structure sheaf of a subscheme.
For the details and proofs please see
\cite{Anno-AffineTanglesAndIrreducibleExoticSheaves}.

Let $\mathfrak{g}$ be ${\mathfrak{s}\mathfrak{l}}_n(\mathbb{C})$.
Consider the Grothendieck-Springer resolution $\tilde{\pi}:\tilde{\mathfrak{g}}\to \mathfrak{g}$.
It provides a resolution of singularities $\pi: T^*Fl_n \to
\mathcal{N}$ of the nilpotent cone $\mathcal{N} \subset \mathfrak{g}$.
Let $z_{2n}$ be a nilpotent element of
${\mathfrak{s}\mathfrak{l}}_{2n}(\mathbb{C})$, with
two Jordan blocks of rank $n$, let
$\mathcal{S}_{2n} \subset\mathfrak{g}$ be a transversal slice to the orbit of
${z_{2n}}$ under the adjoint action of $\gsl_{2n}(\mathbb{C})$, and
let $\mathcal{U}_{2n}\subset\tilde{\mathfrak{g}}$ be the
preimage of $\mathcal{S}_{2n}$ under the resolution $\pi$. 
By \cite[Remark 2.2]{Bezrukavnikov-NonCommutativeCounterpartsOfTheSpringerResolution}
the action of the affine braid group on $D(\tilde{\mathfrak{g}})$ restricts to 
an action of the same group on $D(\mathcal{U}_{2n})$; one can construct this action 
explicitly in a manner similar to \cite{KhovanovThomas-BraidCobordismsTriangulatedCategoriesAndFlagVarieties}.
The variety $\mathcal{U}_{2n}$ is smooth symplectic of complex dimension
$2n$.  The preimage $\mathcal{X}_{2n}$ of ${z_{2n}}$ is a
projective variety of dimension $n$. It is a union
of smooth components intersecting normally.
For simplicity, denote the derived category
$D_{\mathcal{X}_{2n}}(\mathcal{U}_{2n})$ by $\mathcal{D}_{2n}$.

The non-affine braid group action on $\mathcal{D}_{2n}$ is generated
by twists in functors $S_i$, $1\leq i \leq 2n-1$ defined by certain
spherical fibrations, cf. \cite{AnnoNandakumar-ExoticTStructuresFor2BlockSpringerFibres} for explicit formulas; it is the
special property of the nilpotent element $z_{2n}$ that the sources of
these functors $S_i$ are all equivalent to $\mathcal{D}_{2n-2}$.
Apart from these functors, there is an autoequivalence
$R:\mathcal{D}_{2n}\to\mathcal{D}_{2n}$, cf. \cite[\S
4.1]{Anno-AffineTanglesAndIrreducibleExoticSheaves} that corresponds
to the affine generator $r$ described above.  The remaining twist $T_{2n}$
can be obtained by conjugating $T_1$ or $T_{2n-1}$ by $R$.

It is proven in \cite{Anno-AffineTanglesAndIrreducibleExoticSheaves}
that the generator $T_n$ is indeed a twist in some
functor $S_{2n}: \mathcal{D}_{2n-2}\to \mathcal{D}_{2n}$.
In fact, $S_{2n}$ is isomorphic to $RS_1$ or $R^{-1}S_{2n-1}$.
The remarkable thing about $S_{2n}$ is that being a composition of
$S_1$ or $S_{2n-1}$ and an autoequivalence of
$\mathcal{D}_{2n}$, it retains many properties of $S_i$'s. In particular,
its kernel $\mathcal{K} \in D(\mathcal{U}_{2n-2}\times
\mathcal{U}_{2n})$
is orthogonally spherical over $\mathcal{U}_{2n-2}$.
At the same time $\mathcal{K}$ is a genuine object of
the derived category $D(\mathcal{U}_{2n-2}\times \mathcal{U}_{2n})$,
that is, not isomorphic to the direct sum of its cohomology sheaves.
It may be seen in the computation carried out in
\cite{Anno-AffineTanglesAndIrreducibleExoticSheaves}, section 7.2, for $n=2$;
in this case $\mathcal{U}_{2n-2}=\mathcal{U}_2 \simeq T^*\mathbb{P}^1$,
and while the image of $\mathcal{O}_{\mathbb{P}^1}$
is a sheaf on $\mathcal{U}_4$, the image
of $\mathcal{O}_{\mathbb{P}^1}(-1)$ is not.
If $\mathcal{K}$ was actually a spherical fibration, that is,
a structure sheaf of some $D \subset \mathcal{U}_4$ fibered
over $\mathcal{U}_2$, this would not be possible.

\bibliography{references}
\bibliographystyle{amsalpha}

\end{document}